\documentclass[11pt]{article}
\usepackage[letterpaper,twoside,outer=1.3in,vmargin=1.3in,]{geometry}
\usepackage{amsmath,amssymb,amsthm,amsfonts,enumerate,times,epsfig,color,hyperref,cleveref}
\allowdisplaybreaks

\newcommand{\ignore}[1]{}

\newcommand{\tr}{\triangle}

\DeclareMathOperator{\dfc}{def}
\DeclareMathOperator{\wt}{\text{wt}}

\DeclareMathOperator{\disc}{\text{disc}}
\DeclareMathOperator{\dc}{\text{dc}}
\DeclareMathOperator{\eqt}{\text{eqt}}
\DeclareMathOperator{\pl}{\text{pl}}

\newcommand{\1}{\mathbf{1}}
\newcommand{\0}{\mathbf{0}}

\newtheorem{theorem}{Theorem}[section]
\newtheorem{CO}[theorem]{Corollary}
\newtheorem{LE}[theorem]{Lemma}

\newtheorem{CN}[theorem]{Conjecture}
\newtheorem{RE}[theorem]{Remark}
\newtheorem{QU}[theorem]{Question}

\newtheorem{DE}[theorem]{Definition}
\newtheorem{EG}[theorem]{Example}

\newcounter{claim_nb}[theorem]
\newtheorem{claim}[claim_nb]{Claim}
\newtheorem*{claim*}{Claim}
\newtheorem*{subclaim*}{Subclaim}

\newcounter{claim_nbs}[section]
\newcounter{subclaim_nb}[claim_nbs]
\newenvironment{cproof}
{\begin{proof}
 [Proof of Claim.]
 \vspace{-1.2\parsep}}
{\renewcommand{\qed}{\hfill $\Diamond$} \end{proof}}

\title{On packing dijoins in digraphs and weighted digraphs}

\author{Ahmad Abdi \and G\'{e}rard Cornu\'{e}jols \and Michael Zlatin
}
\begin{document}

\maketitle

\begin{abstract}
Let $D=(V,A)$ be a digraph. A \emph{dicut} is a cut $\delta^+(U)\subseteq A$ for some nonempty proper vertex subset $U$ such that $\delta^-(U)=\emptyset$, a \emph{dijoin} is an arc subset that intersects every dicut at least once, and more generally a \emph{$k$-dijoin} is an arc subset that intersects every dicut at least $k$ times. Our first result is that $A$ can be partitioned into a dijoin and a $(\tau-1)$-dijoin where $\tau$
denotes the smallest size of a dicut. Woodall conjectured the stronger statement that $A$ can be
partitioned into $\tau$ dijoins.

Let $w\in \mathbb{Z}^A_{\geq 0}$ and suppose every dicut has weight at least $\tau$, for some integer $\tau\geq 2$.
Let $\rho(\tau,D,w):=\frac{1}{\tau}\sum_{v\in V} m_v$, where each $m_v$ is the integer in $\{0,1,\ldots,\tau-1\}$ equal to $w(\delta^+(v))-w(\delta^-(v))$ mod $\tau$. We prove the following results: \begin{itemize}
\item[(i)] If $\rho(\tau,D,w)\in \{0,1\}$, then there is an equitable $w$-weighted packing of dijoins of size $\tau$.
\item[(ii)] If $\rho(\tau,D,w)= 2$, then there is a $w$-weighted packing of dijoins of size $\tau$.
\item[(iii)] If $\rho(\tau,D,w)=3$,  $\tau=3$, and $w=\1$, then $A$ can be partitioned into three dijoins.
\end{itemize}

Each result is best possible: (i) does not hold for $\rho(\tau,D,w)=2$ even if $w=\1$, (ii) does not hold for $\rho(\tau,D,w)=3$, and (iii) do not hold for general $w$.
\end{abstract}

\bigskip
\noindent {\bf Keywords:} min-max theorem, dijoins, strongly base orderable matroid, packing common bases, submodular function, integer decomposition property

\section{Introduction}\label{sec:intro}

A \emph{weighted digraph} is a pair $(D=(V,A),w)$ where $D$ is a digraph, and $w\in \mathbb{Z}_{\geq 0}^A$. We allow parallel or opposite arcs but not loops.\footnote{Two arcs are \emph{parallel} if they have the same head and the same tail; they are \emph{opposite} if one's head/tail is the other's tail/head.} A \emph{dicut} is a cut $\delta^+(U)\subseteq A$ for some nonempty proper subset $U\subseteq V$ such that $\delta^-(U)=\emptyset$. The \emph{$w$-weight of $F\subseteq A$}, or simply the \emph{weight of $F$}, is $w(F)=\sum_{a\in F}w_a$. Denote by $\tau(D,w)$ the minimum weight of a dicut. Observe that if $\tau(D,w)\geq 1$, then $D$ is weakly connected. For this reason, we may focus only on instances where $D$ is weakly connected.
A \emph{dijoin} is a subset $J\subseteq A$ such that $D/J$ is strongly connected; equivalently, a dijoin is an arc subset that intersects every dicut at least once. A \emph{$w$-weighted packing of dijoins of size $\nu$} is a collection of $\nu$ (not necessarily distinct) dijoins such that every arc $a$ belongs to at most $w_a$ of the dijoins. Denote by $\nu(D,w)$ the maximum size of a $w$-weighted packing of dijoins. It follows from Weak LP Duality that $\tau(D,w)\geq \nu(D,w)$.

\begin{QU}\label{main-question}
When does equality hold in $\tau(D,w)\geq \nu(D,w)$?
\end{QU}

Consider replacing an arc $a$ of nonzero weight $w_a\geq 1$ with $w_a$ arcs of weight $1$ with the same head and tail as $a$. This operation preserves both the $\tau$ and $\nu$ parameters in \Cref{main-question}. For this reason, we may restrict our attention to $0,1$ weights. Observe that deleting an arc of weight $0$ may create new dicuts, and thus potentially decrease the covering parameter.

Woodall~\cite{Woodall78} conjectured that $\tau(D,w)=\nu(D,w)$ when $w=\1$, while Edmonds and Giles~\cite{Edmonds77} conjectured equality holds in general.\footnote{$\0,\1$ denote the all-zeros and all-ones vectors of appropriate dimensions.}  Schrijver~\cite{Schrijver80} refuted the Edmonds-Giles Conjecture. However, Woodall's Conjecture remains one of the most appealing and challenging unsolved problems in Combinatorial Optimization~\cite{Cornuejols01,Schrijver03,Frank11}.
In this paper we introduce a new approach to Question~\ref{main-question}, and we propose
a fix to the refuted Edmonds-Giles Conjecture.

\subsection{Highlights of this paper}

We need to define a few notions and notations. Let $(D=(V,A),w)$ be a weighted digraph.

\begin{DE}
Given an integer $\tau\geq 2$, let
\begin{align*}
\rho(\tau,D,w) &:= \frac{1}{\tau}\sum_{v\in V}\big(w(\delta^+(v))-w(\delta^-(v)) \mod{\tau}\big)\\
\rho(\tau,D)&:=\rho(\tau,D,\1)\\
\bar{\rho}(\tau,D,w) &:= \frac{1}{\tau}\sum_{v\in V}\big(w(\delta^-(v))-w(\delta^+(v)) \mod{\tau}\big).
\end{align*} 
Here, for an integer $n$, $n$ mod $\tau$ is the integer in $\{0,1,\ldots,\tau-1\}$ that is equal to $n$ mod $\tau$.
\end{DE}

Observe that $\rho(\tau,D,w)$ is a nonnegative integer since $\sum_{v\in V}(w(\delta^+(v))-w(\delta^-(v)))=0$. The significance of this parameter becomes clear as we explain our approach to Woodall's Conjecture in  \S\ref{sec:intro/approach}, and discuss the two matroids $M_0,M_1$. Vaguely speaking, the smaller this parameter, the `closer' is $(D,w)$ to being `$\tau$-regular and bipartite'. Moving on, observe further that $\bar{\rho}(\tau,D,w)=\rho(\tau,D',w')$, where $(D',w')$ is obtained from $(D,w)$ by replacing every arc by the reverse arc of the same weight. For this reason, we will only work with $\rho$. 

\begin{DE} A \emph{$k$-dijoin} of $D$ is an arc subset that intersects every dicut at least $k$ times. \end{DE}

\begin{DE} A $w$-weighted packing $J_1,\ldots,J_\nu$ of dijoins of $D$ is \emph{equitable} if $|J_i\cap \delta^+(U)|-|J_j\cap \delta^+(U)|\in \{-1,0,1\}$ for all $i,j\in [\nu]$ and for every (inclusionwise) minimal dicut $\delta^+(U)$. \end{DE}

\paragraph{Four principal results.} Let $\tau\geq 2$ be an integer, and $(D=(V,A),w)$ a weighted digraph where every dicut has weight at least~$\tau$. We prove the following \hypertarget{primary-results}{statements}: \begin{enumerate}
\item[{\bf P1}] If $w=\1$, then there exist a dijoin and a $(\tau-1)$-dijoin that are disjoint. 
\item[{\bf P2}] If $\rho(\tau,D,w)\in \{0,1\}$, then $(D,w)$ has an equitable $w$-weighted packing of dijoins of size $\tau$.
\item[{\bf P3}] If $\rho(\tau,D,w)= 2$, then $(D,w)$ has a $w$-weighted packing of dijoins of size $\tau$.
\item[{\bf P4}] If $\rho(\tau,D,w)=3$,  $\tau=3$, and $w=\1$, then there exist three (pairwise) disjoint dijoins.
\end{enumerate}
It should be pointed out that {\bf P1} and {\bf P4} hold more generally for $w>\0$, by simply applying the two results to the weighted digraphs obtained by replacing every arc $a$ of weight $w_a\geq 1$ with $w_a$ arcs of weight $1$ with the same head and tail as $a$.

\paragraph{The edge.} {\bf P1} and {\bf P4} do not hold for general $w$, and ${\bf P3}$ does not extend to $\rho(\tau,D,w)\geq 3$, as we note later in \S\ref{subsec:background}. Moreover, {\bf P2} does not extend to $\rho(\tau,D,w)\geq 2$, even if $w=\1$; we see a demonstration of this through an example in \S\ref{subsec:examples}.

\subsection{Our approach and two secondary results}\label{sec:intro/approach}

\paragraph{Weighted $(\tau,\tau+1)$-bipartite digraphs.} All of our results are made possible through a reduction to a special class of weighted digraphs which we define now. Given a digraph, a \emph{source} is a vertex with only outgoing arcs, and a \emph{sink} is a vertex with only incoming arcs.
A \emph{bipartite digraph} is a digraph where every vertex is either a source or a sink, in other words, it is obtained from a bipartite graph where all the edges are oriented from one part of a bipartition to the other part. A \emph{weighted bipartite digraph} is a pair $(D,w)$ where $D$ is a bipartite digraph. The weighted degree of a vertex $v$ is $w(\delta(v))$.

\begin{DE}
Given an integer $\tau\geq 1$, a \emph{weighted $(\tau,\tau+1)$-bipartite digraph} is a weighted bipartite digraph where every vertex has weighted degree $\tau$ or $\tau+1$, the vertices of weighted degree $\tau+1$ form a stable set, and every dicut has weight at least $\tau$.
A weighted $(\tau,\tau+1)$-bipartite digraph is \emph{sink-regular} if every sink has weighted degree $\tau$, and it is \emph{balanced} if it has an equal number of sources and sinks.
\end{DE}

Observe that in a weighted $(\tau,\tau+1)$-bipartite digraph, the minimum weight of a dicut is $\tau$, and every arc belongs to a minimum weight dicut as it is incident with a source or a sink of weighted degree~$\tau$. 

\begin{DE}
Given an integer $\tau\geq 2$, a \emph{$(\tau,\tau+1)$-bipartite digraph} is a bipartite digraph $D$ such that $(D,\1)$ is a weighted $(\tau,\tau+1)$-bipartite digraph.
\end{DE}

\paragraph{Decompose, Lift, and Reduce.} We \emph{reduce} the problem of finding a weighted packing of dijoins, and more generally $k$-dijoins, of size $\tau$ in a weighted digraph $(D,w)$ to the same problem in a set of weighted $(\tau,\tau+1)$-bipartite digraphs. This is done via \emph{Decompose-and-Lift}, a flexible and versatile operation applied to $(D,w)$ which can preserve planarity, adhere to equitability, and can also be done for unweighted digraphs if $\tau\geq 3$. The weighted $(\tau,\tau+1)$-bipartite digraphs encountered can be picked to be sink-regular or balanced, though it is the former that is most useful in this paper. We prove the four principal results for sink-regular weighted $(\tau,\tau+1)$-bipartite digraphs, and then use the Decompose, Lift, and Reduce procedure to deduce the results for all weighted digraphs. The procedure is explained in \S\ref{sec:DLR}, and the technical proofs appear in the appendix \S\ref{sec:DLRproof}.

\paragraph{The matroids $M_0,M_1$.} Let $\tau\geq 2$ be an integer, and $(D,w)$ a sink-regular weighted $(\tau,\tau+1)$-bipartite digraph. Note that $(D,w)$ has $\tau\cdot \rho(\tau,D,w)$ sources of weighted degree $\tau+1$. For each $i\in \{0,1\}$, we define a matroid $M_i$ whose ground set is the set of sources of weighted degree $\tau+1$, and whose rank is $\rho(\tau,D,w)$. (The matroids are introduced in \S\ref{sec:two-dijoins} and \S\ref{sec:rho=2}.) These two matroids are relevant in that if $(D,w)$ has a $w$-weighted packing of dijoins of size $\tau$, then the ground set of the matroids can be partitioned into $\tau$ common bases of $M_0,M_1$. (We do not know if the converse holds.) We shall see that $M_0$ is a strongly base orderable matroid, while $M_1$ may not be. 

\paragraph{Two secondary results.} We prove two secondary \hypertarget{secondary-results}{results}: \begin{enumerate}
\item[{\bf S1}] If $M_1$ is a strongly base orderable matroid, then $(D,w)$ has a $w$-weighted packing of dijoins of size $\tau$.
\item[{\bf S2}] If the ground set of $M_1$ can be partitioned into a $1$-admissible set $Q$ and a $(\tau-1)$-admissible set $Q'$ such that $M_1|Q'$ is strongly base orderable, then $(D,w)$ has a $w$-weighted packing of dijoins of size $\tau$.
\end{enumerate} Here, for an integer $k\in \{1,\ldots,\tau\}=:[\tau]$, a set is \emph{$k$-admissible} if it is the union of $k$ disjoint bases of $M_0$, and also the union of $k$ disjoint bases of $M_1$. We shall see that {\bf S2} is strictly stronger than {\bf S1}.

\subsection{Contextual background}\label{subsec:background}

It will be convenient to symbolize statements.

\begin{DE}
For integers $\tau\geq 2$ and $\rho\geq 0$, symbolize the following statement: \begin{itemize}
\item $[[\wt,\tau,\rho;\pl,\eqt]]$: Given a weighted digraph $(D,w)$ that is planar such that every dicut has weight at least $\tau$, and satisfies $\rho(\tau,D,w)\leq\rho$, there exists an equitable $w$-weighted packing of dijoins of size~$\tau$.
\end{itemize}
If the key $\wt$ is missing then we have the unweighted analogue of the statement, if $\pl$ is missing the planarity condition is removed, if $\eqt$ is missing then the adjective ``equitable" is removed from the conclusion, and if the parameter $\rho$ is missing then the upper bound on the function $\rho$ is removed. (Any combination of the keys and parameter can be missing.)
\end{DE}

$[[\wt,\tau]]$ is known to be true for instances $(D,w)$ where the underlying undirected graph of $D$ is series-parallel~\cite{Lee01}, or more generally has no $K_5\setminus e$ minor~\cite{Lee06}. Edmonds and Giles~\cite{Edmonds77} conjectured that $[[\wt,\tau]]$ is true, but Schrijver refuted the conjecture~\cite{Schrijver80} by exhibiting a counterexample to $[[\wt,2]]$ (see Figure~\ref{fig:dijoins-mnp}); others were also found later~\cite{CG2002,Williams04}.
An extension of Schrijver's example disproves $[[\wt,\tau]]$ for any $\tau\geq 2$: for each of the three paths of solid arcs, change the weight of the middle arc from $1$ to $\tau-1$~\cite{Harvey11}; this extension shows that $[[\wt,\tau,3;\pl]]$ is false for any integer $\tau\geq 2$.

\begin{figure}[h]
\centering\includegraphics[scale=.3]{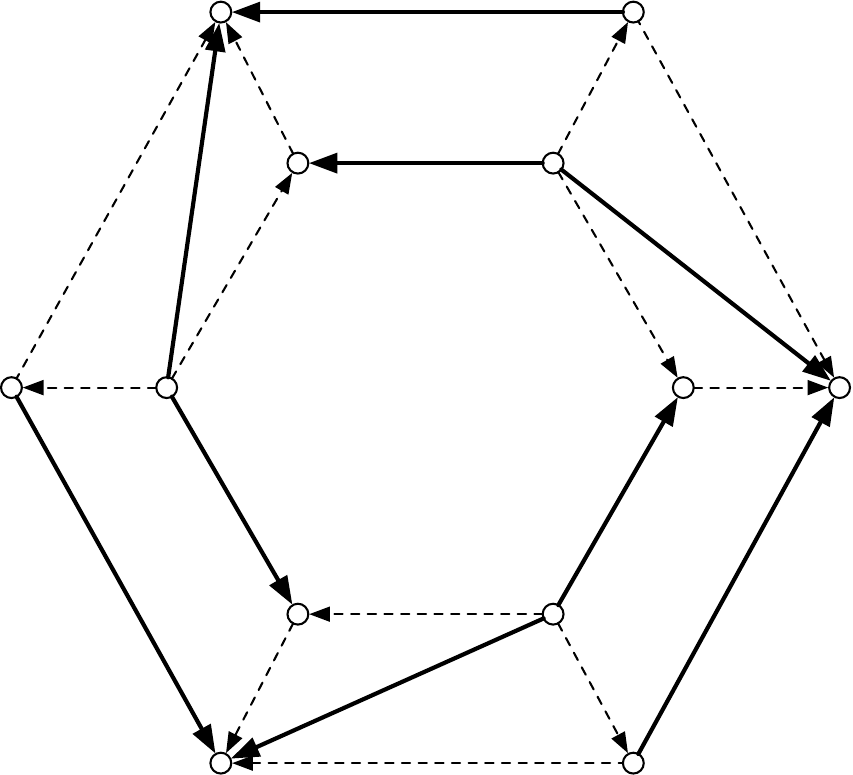}\\
\caption{Solid arcs have weight $1$, and dashed arcs have weight $0$.}
\label{fig:dijoins-mnp}
\end{figure}

That $[[\wt,2]]$ is false implies that {\bf P1} does not hold for general weights.
That $[[\wt,\tau,3;\pl]]$ is false for any $\tau\geq 2$ implies that {\bf P3} does not hold for $\rho(w,\tau,D)\geq 3$, and {\bf P4} does not hold for general weights even for $\tau=3$.

A \emph{super source} in a digraph is a source that has a directed path to every sink; a \emph{super sink} is defined similarly. It has been conjectured that $[[\wt,\tau]]$ is true for weighted digraphs $(D,w)$ where $D$ has a super source and a super sink~\cite{Guenin05}. Schrijver~\cite{Schrijver82}, and Feofiloff and Younger~\cite{FY1987}, proved this conjecture for \emph{source-sink connected} instances, i.e. digraphs in which every source has a directed path to every sink.

It has been conjectured in a recent paper~\cite{Chudnovsky16} that $[[\wt,\tau]]$ is true for weighted digraphs $(D,w)$ where $D[\{a\in A:w_a\neq 0\}]$ is a spanning subdigraph of $D$ that is connected as an undirected graph. In the same paper, this conjecture was proved in two special cases: $\tau=2$ and $D$ is planar (we revisit this result in \S\ref{sec:future}), or $\tau=2$ and $D[\{a\in A:w_a\neq 0\}]$ is a caterpillar subdivision.

Woodall's Conjecture predicts that $[[\tau]]$ is true~\cite{Woodall78}. The correctness of $[[\tau]]$ for $\tau=2$ is folklore (see \cite{Schrijver03}, Theorem 56.3), which is convenient as the Decompose, Lift, and Reduce procedure in the unweighted setting only works for $\tau\geq 3$. The correctness of $[[\tau]]$, or even $[[\tau;\pl]]$, remains unknown for any $\tau\geq 3$. Recently, M\'{e}sz\'{a}ros~\cite{Meszaros18} proved the statement $[[\tau,0]]$ by using a general result on totally unimodular matrices; see Lemma 9 of that paper, the argument essentially proves $[[\wt,\tau,0;\eqt]]$. Then, in an elegant fashion, he combines $[[\tau,0]]$ with \emph{Olson's Lemma} from Number Theory to prove $[[q]]$ where $q$ is a prime power and the underlying undirected graph of $D$ is $(q-1,1)$-partition-connected.

\subsection{Some examples}\label{subsec:examples}

\paragraph{About {\bf P2}.} {\bf P2} does not extend to $\rho(\tau,D,w)\geq 2$, even if $w=\1$. To see this, consider the digraph $D$ displayed in Figure~\ref{fig:D3}. It can be readily checked that every dicut has size at least $\tau:=2$, every arc belongs to a minimum dicut (of size two), and $\rho(\tau,D)=2$. \begin{figure}[h]
\centering\includegraphics[scale=.5]{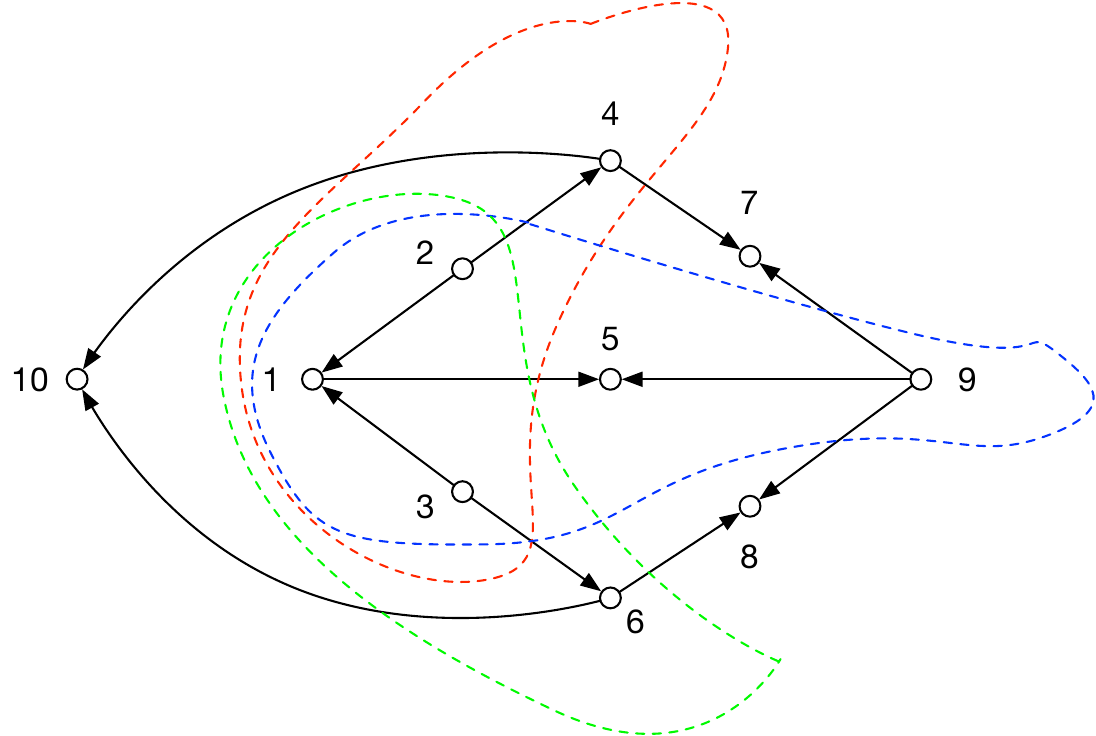}
\caption{A counterexample to $[[2,2;\eqt]]$.}
\label{fig:D3}
\end{figure} 
We claim that the digraph has no equitable $\1$-weighted packing of dijoins of size $\tau$. Suppose otherwise, and let $J_1,J_2$ be such a packing. Since every arc belongs to a dicut of size two, $J_1,J_2$ must form a partition of the arc set. We may assume that $J_1$ picks two of the arcs in the dicut $\delta^+(\{1,2,3\})$, while $J_2$ picks the other arc of the dicut. By symmetry, we may assume that $(2,4)\in J_1$. 
Assume in the first case that $(3,6)\in J_1$ and $(1,5)\in J_2$. By equitability, $J_1$ must pick exactly two arcs from the displayed dicut $\delta^+(\{1,2,3,5,9\})$. Thus, $(9,7),(9,8)\in J_2$, so $(4,7),(6,8)\in J_1$, so equitability along the displayed dicut $\delta^+(\{1,2,3,4\})$ tells us $(4,10)\in J_2$, and equitability along the displayed dicut $\delta^+(\{1,2,3,6\})$ tells us $(6,10)\in J_2$, a contradiction as the sink $10$ is not incident with an arc from $J_1$. 
Assume in the remaining case that $(1,5)\in J_1$ and $(3,6)\in J_2$. By equitability along $\delta^+(\{1,2,3,6\})$, $(6,8),(6,10)\in J_2$, so $(9,8),(4,10)\in J_1$. By equitability along $\delta^+(\{1,2,3,5,9\})$, $(9,7)\in J_2$, so $(4,7)\in J_1$, a contradiction to the equitability of $\delta^+(\{1,2,3,4\})$.

\paragraph{About {\bf P1}.} Let $D=(V,A)$ be a digraph where every dicut has size at least $\tau\geq 2$. We mentioned in \S\ref{subsec:background} that $D$ has two disjoint dijoins. In fact, let $J$ be any minimal dijoin. Then reversing the arcs of $J$ makes the digraph strongly connected (see \cite{Schrijver03}, Theorem 55.1), implying that $J$ does not contain a dicut, implying in turn that $A-J$ is a dijoin. Given this observation, and {\bf P1}, a natural question is whether $A-J$ is necessarily a $(\tau-1)$-dijoin? Unfortunately, the answer is no, if $\tau\geq 3$. \begin{quote}For example, suppose $D$ is the bipartite digraph with sources $\{1,2,3\}$ and sinks $\{4,5,6\}$ and an arc from every source to every sink, and $\tau=3$. Then $J=\{(1,4),(1,5),(2,6),(3,6)\}$ is a minimal dijoin, but $A-J$ is not a $(\tau-1)$-dijoin because $|\delta^+(\{1\})-J|=\tau-2$.\end{quote}

Despite the negative answer, {\bf P1} guarantees the existence of some minimal dijoin $J^\star$ such that $A-J^\star$ is a $(\tau-1)$-dijoin.

\begin{figure}[h]
\centering\includegraphics[scale=.3]{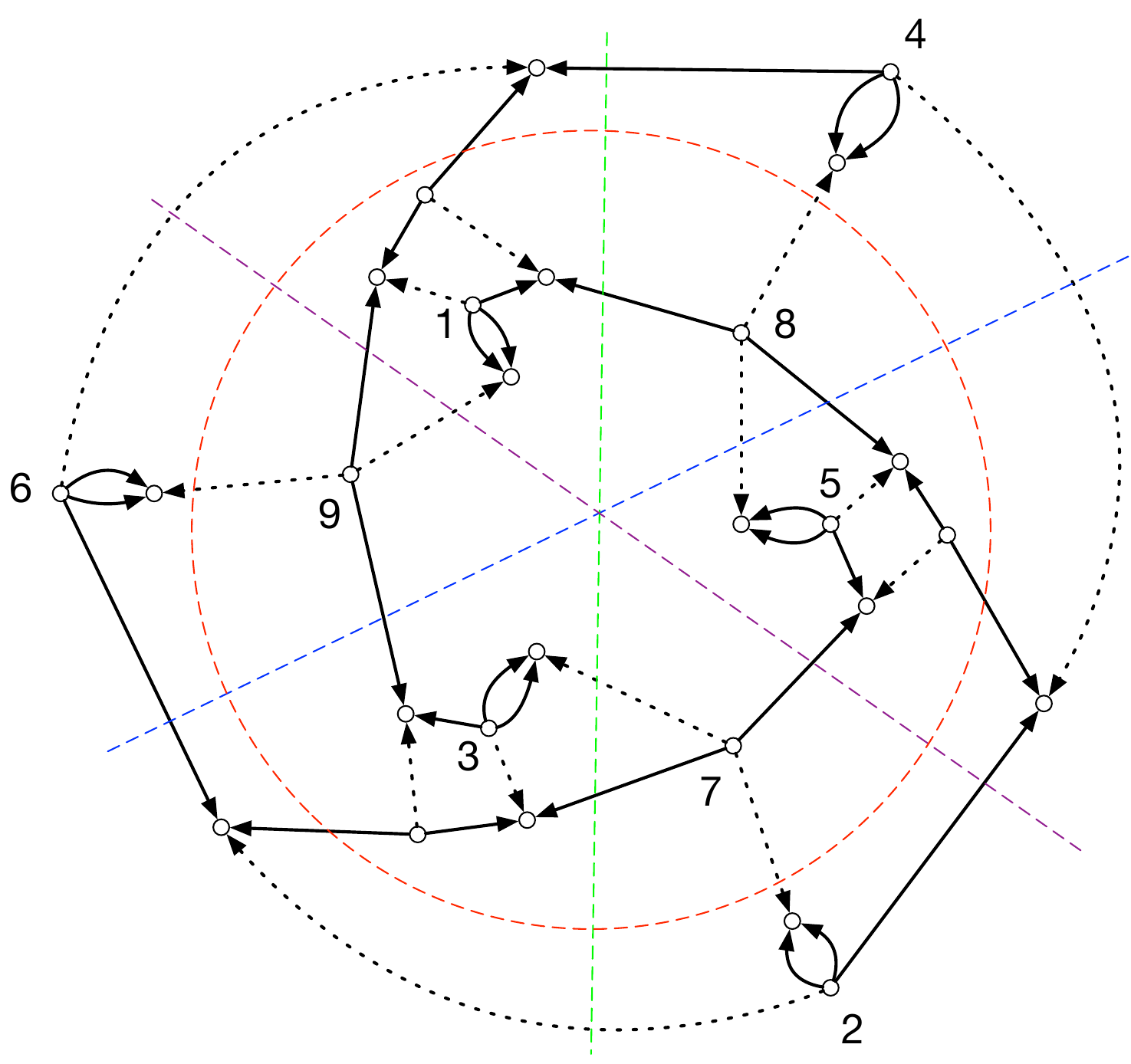}
\caption{The sink-regular $(3,4)$-bipartite digraph $D_{27}=(V,A)$.
The dashed and solid arcs partition the arc set into a minimal dijoin $J^\star$ and a $2$-dijoin $A-J^\star$, respectively. The four dicuts depicted show that $A-J^\star$ cannot be partitioned into two dijoins.}
\label{fig:D27-1}
\end{figure}

\paragraph{About {\bf P4}.} A natural question is whether $A-J^\star$ can necessarily be partitioned into $\tau-1$ dijoins? Unfortunately, the answer is no, if $\tau\geq 3$. Let us demonstrate this through an example, displayed in Figure~\ref{fig:D27-1}, which is analyzed further in \S\ref{sec:D27}.\begin{quote}
Figure~\ref{fig:D27-1} displays a sink-regular $(3,4)$-bipartite digraph $D_{27}=(V,A)$ on $27$ vertices, where four distinguished dicuts are highlighted.
Denote by $J^\star$ the set of dashed arcs. It can be checked that $J^\star$ is a minimal dijoin, and $A-J^\star$ is a $2$-dijoin. However, $A-J^\star$ cannot be partitioned into two dijoins.

Suppose for a contradiction that $A-J^\star$ is partitioned into dijoins $R,B$. Call the arcs in $R$ \emph{red}, and the arcs in $B$ \emph{blue}. Denote by $P_{12}$ the $12$-path in $A-J^\star$, by $P_{34}$ the $34$-path in $A-J^\star$, and by $P_{56}$ the $56$-path in $A-J^\star$. Since every internal vertex of each path is a source or a sink incident with exactly two arcs from $A-J^\star$, it follows that the arcs of each path are alternately colored red and blue. A simple argument now tells us that one of the four dicuts displayed is \emph{monochromatic}, in that all the arcs of $A-J^\star$ in the dicut have the same color, implying in turn that one of $R,B$ is disjoint from one of the four dicuts, a contradiction.
\end{quote}

It can be readily checked that $\rho(3,D_{27})=3$. Thus {\bf P4} guarantees the existence of three disjoint dijoins.

\subsection{Outline of the paper}

The four principal results {\bf P1}-{\bf P4}, as well as the two secondary results {\bf S1}-{\bf S2}, are proved in five stages. Let $(D,w)$ be a weighted digraph where every dicut has weight at least $\tau$, for some integer $\tau\geq 2$.
\begin{enumerate}
\item[Stage 0] In \S\ref{sec:DLR} we introduce the Decompose, Lift, and Reduce procedure which reduces the problem of finding a weighted packing of dijoins, or $k$-dijoins, of size~$\tau$ in weighted digraphs to the same problem for a set of weighted $(\tau,\tau+1)$-bipartite digraphs obtained via the Decompose-and-Lift operation. The operation applies to $(D,w)$ and turns it into a set of weighted $(\tau,\tau+1)$-bipartite digraphs that can be chosen sink-regular or balanced.
\end{enumerate}

After this stage, we assume that $(D,w)$ is a sink-regular weighted $(\tau,\tau+1)$-bipartite digraph.
\begin{enumerate}
\item[Stage 1] In \S\ref{sec:rho=1} we interpret $\rho(\tau,D,w)$ as the \emph{discrepancy} between the number of sources and the number of sinks of $D$. We also see that every dijoin used in a weighted packing of size $\tau$, if any, is a \emph{rounded $1$-factor}. We then apply an alternating path technique to prove {\bf P2}.

\item[Stage 2] In \S\ref{sec:two-dijoins} we discuss crossing families and crossing submodular functions. We then review two results of Frank and Tardos~\cite{Frank84} and Fujishige~\cite{Fujishige84} on matroids and box-TDI systems from crossing submodular functions, as well as the integer decomposition property of matroid base polytopes. By applying these results, we introduce the matroid $M_1$, and prove that its ground set can be partitioned into $\tau$ bases. We then leverage rounded $1$-factors, and more generally the notion of \emph{perfect $b$-matchings}, to prove {\bf P1}.

\item[Stage 3] In \S\ref{sec:rho=2} we introduce the matroid $M_0$. By using results from Stage 1, we see that $M_0$ is a strongly base orderable matroid whose ground set can be partitioned into $\tau$ bases. We introduce and study \emph{admissible sets}, which are common bases of $M_0,M_1$. We use a result of Davies and McDiarmid~\cite{Davies76} to prove {\bf S1}. Finally, by using a result of Brualdi~\cite{Brualdi69} on symmetric basis exchange in matroids, we prove {\bf P3}.

\item[Stage 4] In \S\ref{sec:rho=3} we unravel and extend some of the notions and results from Stages 2 and 3. More specifically, we introduce and study the notion of $k$-admissibility, extend the second principal result to the weighted setting under a certain common base packing assumption, and then prove {\bf S2}. A result of Brualdi~\cite{Brualdi71} guides us to study $M(K_4)$, the cycle matroid of $K_4$. We then study $M(K_4)$-restrictions in matroids of rank $3$ with at most $9$ elements, and then prove {\bf P4}.
\end{enumerate}

In \S\ref{sec:D27} we revisit the example of Figure~\ref{fig:D27-1} addressing questions raised in \S\ref{sec:rho=2} and \S\ref{sec:rho=3}. Finally, in \S\ref{sec:future}, we present several directions for future research towards tackling Woodall's Conjecture, we propose a fix to the refuted Edmonds-Giles Conjecture, and we also make connections to \emph{Barnette's Conjecture} and to  the \emph{$\tau=2$ Conjecture}.

\subsection{Notation and terminology} \label{notation}

We fix some notation and terminology used throughout the paper.

\paragraph{Graphs.} Given a graph $G=(V,E)$, a \emph{cycle} is a subset $C\subseteq E$ where every vertex $v\in V$ is incident with an even number of edges from $C$. In particular, $\emptyset$ is a cycle. A \emph{circuit} is a nonempty cycle that does not contain another nonempty cycle. For $U\subseteq V$, denote by $G[U]$ the induced subgraph on vertex set $U$. For $F\subseteq E$, denote by $G[F]$ the subgraph with edge set $F$.

\paragraph{Digraphs.} Sometimes, when there is no risk of ambiguity, we treat a digraph $D$ as an undirected graph $G$ obtained by dropping the orientation of the arcs, which we call the \emph{underlying undirected graph}. For example, we denote $\delta_D(X):=\delta_D^+(X)\cup \delta_D^-(X)$ and $\deg_D(x):=\deg_D^+(x)+\deg_D^-(x)$. We say $D$ is \emph{connected as an undirected graph} if $G$ is connected, a \emph{connected component of $D$} is simply a connected component of $G$, $D$ is \emph{planar} if $G$ is planar, $D$ is a \emph{plane digraph} if $G$ is a plane graph, i.e. $G$ is a planar graph embedded in the plane. Minor operations in $D$ are defined similarly as for the undirected graph $G$; loops created after contraction are then deleted of course. For $U\subseteq V$, denote by $D[U]$ the induced subdigraph on vertex set $U$. For $F\subseteq A$, denote by $D[F]$ the subdigraph with arc set $F$.

\paragraph{Matroids.} Let $M$ be a matroid over ground set $E$. $M$ is \emph{strongly base orderable} if for every two bases $B_1,B_2$, there exists a bijection $\pi:B_1-B_2\to B_2-B_1$ such that $B_1\tr (X\cup \pi(X)),B_2\tr (X\cup \pi(X))$ are bases for all $X\subseteq B_1-B_2$~\cite{Brualdi70}. Given $X\subseteq E$, the \emph{restriction $M|X$} is the deletion $M\setminus (E-X)$. Given a graph $G=(V,E)$, the \emph{cycle matroid of $G$}, denoted $M(G)$, is the matroid over ground set $E$ whose circuits correspond to the circuits of $G$.

\paragraph{Clutters.} Some knowledge of Clutter Theory will be useful and insightful. Let $A$ be a finite set of \emph{elements}, and $\mathcal{C}$ a family of subsets of $A$ called \emph{members}. $\mathcal{C}$ is a \emph{clutter} over \emph{ground set} $A$ if no member contains another~\cite{Edmonds70}. A \emph{cover} of $\mathcal{C}$ is a subset of $A$ that intersects every member of $\mathcal{C}$.
A cover of $\mathcal{C}$ is \emph{minimal} if it does not contain another cover. The family of minimal covers of $\mathcal{C}$ forms another clutter over the same ground set, called the \emph{blocker of $\mathcal{C}$}, and denoted $b(\mathcal{C})$. It is well-known that $b(b(\mathcal{C}))=\mathcal{C}$~\cite{Edmonds70,Isbell58}. We call $(\mathcal{C},b(\mathcal{C}))$ a \emph{blocking pair}.

\begin{RE}\label{dicut-dijoin}
Let $D=(V,A)$ be a digraph. Then the clutter of minimal dijoins and the clutter of minimal dicuts form a blocking pair.
\end{RE}

Let $w\in \mathbb{Z}^A_{\geq 0}$. The \emph{$w$-weight}, or simply \emph{weight}, of a cover $B$ is $w(B)=\sum_{a\in B}w_a$. The minimum weight of a cover is denoted $\tau(\mathcal{C},w)$. A \emph{$w$-weighted packing of $(\mathcal{C},w)$ of size $\nu$} is a collection of $\nu$ members of $\mathcal{C}$ such that every element $a\in A$ is contained in at most $w_a$ of the members. A $\1$-weighted packing is simply called a \emph{packing}. Denote by $\nu(\mathcal{C},w)$ the maximum size of a $w$-weighted packing. It can be readily checked from Weak LP Duality that $\tau(\mathcal{C},w)\geq \nu(\mathcal{C},w)$. We say that a $w$-weighted packing $C_1,\ldots,C_\nu$ is \emph{equitable} if for every $B\in b(\mathcal{C})$, the difference $|C_i\cap B|-|C_j\cap B|$ is in $\{-1,0,1\}$ for all $i,j\in [\nu]$.

Let $a\in A$. The \emph{deletion $\mathcal{C}\setminus a$} is the clutter over ground set $A-\{a\}$ whose members are $C\in \mathcal{C},a\notin C$, while the \emph{contraction $\mathcal{C}/a$} is the clutter over ground set $A-\{a\}$ whose members are the minimal sets in $\{C-\{a\}:C\in \mathcal{C}\}$~\cite{Fulkerson71}.
Deletion in $\mathcal{C}$ corresponds to contraction in $b(\mathcal{C})$, and vice versa~\cite{Seymour76}. A clutter obtained from $\mathcal{C}$ after a series of contractions is called a \emph{contraction minor} of $\mathcal{C}$; \emph{deletion minor} and more generally \emph{minor} are defined similarly.

The clutter obtained from $\mathcal{C}$ after \emph{replicating $a$} is the clutter over ground set $A\cup \{a'\}$ for a new element $a'$ whose members are those in $\mathcal{C}\cup \{C\tr \{a,a'\}:a\in C\in \mathcal{C}\}$. The clutter obtained from $\mathcal{C}$ after \emph{duplicating $a$} is the clutter over ground set $A\cup \{a'\}$ for a new element $a'$ whose members are those in $\{C\in \mathcal{C}:a\notin C\}\cup \{C\cup \{a'\}:a\in C\in \mathcal{C}\}$. It can be readily checked that replication in $\mathcal{C}$ corresponds to duplication in $b(\mathcal{C})$, and vice versa.

\begin{RE}
Let $\mathcal{C}'$ be obtained from $\mathcal{C}$ after deleting each $a\in A$ with $w_a=0$, and replicating each $a\in A$ with $w_a\geq 1$ exactly $w_a-1$ times. Then $(\mathcal{C},w)$ has an (equitable) $w$-weighted packing of size $\nu$ if, and only if, $\mathcal{C}'$ has an (equitable) packing of size $\nu$.
\end{RE}

Let $\mathcal{C}_1,\ldots,\mathcal{C}_k$ be clutters over disjoint ground sets $A_1,\ldots,A_k$, respectively. The \emph{product} of $\mathcal{C}_1,\ldots,\mathcal{C}_k$, denoted $\prod_{i\in [k]}\mathcal{C}_i$, is the clutter over ground set $\cup_{i\in [k]} A_i$ whose members are the sets in $\{\cup_{i\in [k]}C_i:C_i\in\mathcal{C}_i,i\in [k]\}$. The \emph{coproduct} of $\mathcal{C}_1,\ldots,\mathcal{C}_k$, denoted $\otimes_{i\in [k]}\mathcal{C}_i$, is the clutter over ground set $\cup_{i\in [k]} A_i$ whose members are the minimal sets in $\cup_{i\in[k]}\mathcal{C}_i$~\cite{Abdi-cuboids}. Observe that if $\mathcal{C}_i\neq \{\emptyset\}$ for each $i\in [k]$, then $\otimes_{i\in [k]}\mathcal{C}_i=\cup_{i\in[k]}\mathcal{C}_i$.

\begin{RE}\label{product-coproduct}
$b\big(\prod_{i\in [k]}\mathcal{C}_i\big) = \otimes_{i\in[k]}b(\mathcal{C}_i)$ and $b\big(\otimes_{i\in [k]}\mathcal{C}_i\big) = \prod_{i\in[k]}b(\mathcal{C}_i)$.
\end{RE}

\section{Decompose, Lift, and Reduce Procedure}\label{sec:DLR}

In this section we introduce the \emph{Decompose, Lift, and Reduce procedure}.

\begin{DE}
Given a digraph $D$, denote by $\mathcal{C}(D)$ the clutter of minimal dijoins of $D$. Given a weighted digraph $(D,w)$, denote by $\mathcal{C}(D,w)$ the clutter obtained from $\mathcal{C}(D)$ after deleting each element $a$ with $w_a=0$ and replicating each element $a$ with $w_a\geq 1$ exactly $w_a-1$ times.
\end{DE}

\subsection{Decompose-and-Lift Operation}\label{sec:DnL}

The proofs of the following two results, which are technical, are given the appendix (\S\ref{sec:DLRproof}).

\begin{theorem}[Decompose-and-Lift]\label{weighted-DnL}
Let $(D,w)$ be a weighted digraph, where every dicut has weight at least $\tau$, and $\tau\geq 2$. Then there exist weighted $(\tau,\tau+1)$-bipartite digraphs $\{(D'_i,w'_i):i\in I\}$ for a finite index set $I$, such that the following statements hold: \begin{enumerate}[(1)]
\item if $w>\0$ and $\tau\geq 3$, then $w'_i>\0$ for each $i\in I$,
\item if $D$ is planar, then so is each $D'_i,i\in I$,
\item $\mathcal{C}(D,w)=\prod_{i\in I}\mathcal{C}_i$ where $\mathcal{C}_i$ is a contraction minor of $\mathcal{C}(D'_i,w'_i)$, for each $i\in I$.
\end{enumerate} Moreover, we can choose each $(D'_i,w'_i),i\in I$ to satisfy any one of the following statements: \begin{enumerate}
\item[i.] $(D'_i,w'_i)$ is balanced,
\item[ii.] $(D'_i,w'_i)$ is sink-regular and $\rho(\tau,D'_i,w'_i)\leq \rho(\tau,D,w)$.
\end{enumerate}
\end{theorem}

\begin{theorem}[Unweighted Decompose-and-Lift]\label{DnL}
Let $D$ be a digraph where every dicut has size at least $\tau$, and $\tau\geq 3$. Then there exist $(\tau,\tau+1)$-bipartite digraphs $\{D'_i:i\in I\}$ for a finite index set $I$, such that the following statements hold: \begin{enumerate}[(1)]
\item if $D$ is planar, then so is each $D'_i,i\in I$,
\item $\mathcal{C}(D)=\prod_{i\in I}\mathcal{C}_i$ where $\mathcal{C}_i$ is a contraction minor of $\mathcal{C}(D'_i)$, for each $i\in I$.
\end{enumerate} Moreover, we can choose each $D'_i,i\in I$ to satisfy any one of the following statements: \begin{enumerate}
\item[i.] $D'_i$ is balanced,
\item[ii.] $D'_i$ is sink-regular and $\rho(\tau,D'_i)\leq \rho(\tau,D)$.
\end{enumerate}
\end{theorem}

\subsection{Reducing}\label{sec:reduce}

Having described \emph{Decompose-and-Lift}, we are now ready to explain how the problem of packing dijoins, or more generally $k$-dijoins, of size $\tau$ in arbitrary weighted digraphs can be \emph{reduced} to the same problem in weighted $(\tau,\tau+1)$-bipartite digraphs. We need the following two remarks.

\begin{RE}\label{contraction-packings}
If $\mathcal{C}$ has a packing of size $\tau$, then so does every contraction minor of $\mathcal{C}$. 
\end{RE}

\begin{RE}\label{contraction-packings-equitable}
If $D$ has an equitable packing of $\tau$ dijoins, then so does every contraction minor of $D$. 
\end{RE}

\begin{RE}\label{product-packings}
If $\mathcal{C}_i,i\in [k]$ have (equitable) packings of size $\tau$, then so does $\prod_{i\in[k]}\mathcal{C}_i$.
\end{RE}

\begin{theorem}[Reducing]\label{weighted-reduction}
Take an integer $\tau\geq 2$. The following statements hold: \begin{enumerate}[(1)]
\item Given any statement $[[\wt,\tau;\cdot,\cdot]]$ that includes $\wt,\tau$, excludes $\rho$, and may include $\pl,\eqt$, the statement is true if, and only if, it is true for all weighted $(\tau,\tau+1)$-bipartite digraphs that are balanced.
\item Given any statement $[[\wt,\tau,\rho;\cdot,\cdot]]$ that includes $\wt,\tau,\rho$, and may include $\pl,\eqt$, the statement is true if, and only if, it is true for all weighted $(\tau,\tau+1)$-bipartite digraphs that are sink-regular.
\end{enumerate}
\end{theorem}
\begin{proof}
$(\Rightarrow)$ holds clearly for both (1) and (2).
$(\Leftarrow)$ is a straightforward consequence of \Cref{contraction-packings}, \Cref{contraction-packings-equitable}, \Cref{product-packings}, and \Cref{weighted-DnL}, where for (1) we pick each $(D'_i,w'_i),i\in I$ to satisfy (i), and for (2) we pick each to satisfy (ii) of \Cref{weighted-DnL}.
\end{proof}

\begin{theorem}[Unweighted Reducing]\label{reduction}
Take an integer $\tau\geq 3$. The following statements hold: \begin{enumerate}[(1)]
\item Given any statement $[[\tau;\cdot,\cdot]]$ that includes $\tau$, excludes $\wt,\rho$, and may include $\pl,\eqt$, the statement is true if, and only if, it is true for all $(\tau,\tau+1)$-bipartite digraphs that are balanced.
\item Given any statement $[[\tau,\rho;\cdot,\cdot]]$ that includes $\tau,\rho$, excludes $\wt$, and may include $\pl,\eqt$, the statement is true if, and only if, it is true for all $(\tau,\tau+1)$-bipartite digraphs that are sink-regular.
\item Let $k\in [\tau-1]$. Consider the following statement: \begin{quote}
Let $D=(V,A)$ be a digraph where every dicut has size at least $\tau$. Then $A$ can be partitioned into a $k$-dijoin and a $(\tau-k)$-dijoin. \end{quote} If this statement is true for every sink-regular $(\tau,\tau+1)$-bipartite digraph, then it is true for every digraph.
\end{enumerate}
\end{theorem}
\begin{proof}
$(\Rightarrow)$ holds clearly for both (1) and (2).
$(\Leftarrow)$ is a straightforward consequence of \Cref{contraction-packings}, \Cref{contraction-packings-equitable}, \Cref{product-packings}, and \Cref{DnL} (which requires $\tau\geq 3$), where for (1) we pick each $D'_i,i\in I$ to satisfy (i), and for (2) we pick each to satisfy (ii).

{\bf (3)}
Let us prove the statement for an arbitrary instance $D=(V,A)$.
By \Cref{DnL}, there exist sink-regular $(\tau,\tau+1)$-bipartite digraphs $D'_i,i\in I$ for a finite index set $I$, such that $\mathcal{C}(D)=\prod_{i\in I}\mathcal{C}_i$, where $\mathcal{C}_i$ is a contraction minor of $\mathcal{C}(D'_i)$ for each $i\in I$. In other words, $b(\mathcal{C}(D))=\otimes_{i\in I}b(\mathcal{C}_i)$, where $b(\mathcal{C}_i)$ is a deletion minor of $b(\mathcal{C}(D'_i))$ for each $i\in I$. By our hypothesis, $A(D'_i)$ can be partitioned into a $k$-dijoin $J^1_i$ and a $(\tau-k)$-dijoin $J^2_i$ of $D'_i$, for each $i\in I$. Let $J^1:=\bigcup_{i\in I} J^1_i$ and $J^2:=\bigcup_{i\in I} J^2_i$. Clearly, $J^1$ and $J^2$ are disjoint. We claim that $J^1$ is a $k$-dijoin and $J^2$ is a $(\tau-k)$-dijoin of $D$, thereby finishing the proof. To this end, let $\delta_D^+(U)$ be a minimal dicut of $D$, that is, $\delta_D^+(U)\in b(\mathcal{C}(D))$. Then $\delta_D^+(U)\in b(\mathcal{C}_j)$, so $\delta_D^+(U)$ is also a minimal dicut of some $D'_j,j\in I$, implying in turn that \begin{align*}
|\delta_{D}^+(U)\cap J^1|&=|\delta_{D'_j}^+(U)\cap J^1_j|\geq k\\
|\delta_{D}^+(U)\cap J^2|&=|\delta_{D'_j}^+(U)\cap J^2_j|\geq \tau-k.
\end{align*} Since these inequalities hold for every minimal dicut of $D$, we get that $J^1$ is a $k$-dijoin and $J^2$ is a $(\tau-k)$-dijoin of $D$, as required.
\end{proof}

We will not use \Cref{weighted-reduction}~(1) and \Cref{reduction}~(1) in the rest of this paper.

\section{$[[\wt,\tau,1;\eqt]]$ is true.}\label{sec:rho=1}

In this section we prove result \hyperlink{primary-results}{{\bf P2}}. Throughout the section, unless stated otherwise, we are given an integer $\tau\geq 2$, a weighted $(\tau,\tau+1)$-bipartite digraph $(D=(V,A),w)$ that is sink-regular, and $w\in \{0,1\}^A$. Let $A_1:=\{a\in A:w_a=1\}$.

\subsection{Rounded $1$-factors}

\begin{DE}
A vertex $v$ of $(D,w)$ is \emph{active} if $w(\delta(v))=\tau+1$, and is \emph{inactive} if $w(\delta(v))=\tau$. Given $U\subseteq V$, denote by $a(U)$ the set of active vertices of $(D,w)$ in $U$.
\end{DE}

\begin{RE}\label{R1F-birth}
Suppose there are $\tau$ disjoint dijoins $J_1,\ldots,J_\tau$ contained in $A_1$. Then the following statements hold: \begin{enumerate}[(1)]
\item for each inactive vertex $v$, $|J_i\cap \delta(v)|=1$ for each $i\in [\tau]$,
\item $J_1,\ldots,J_\tau$ partition $A_1$, and
\item for each active vertex $v$, $|J_i\cap \delta(v)|\in \{1,2\}$ for each $i\in [\tau]$, and $|J_j\cap \delta(v)|= 2$ for exactly one~$j$.
\end{enumerate}
\end{RE}
\begin{proof}
{\bf (1)} For every inactive vertex $v$, $\delta(v)$ is a dicut of weight $\tau$, so (1) follows.
{\bf (2)} follows from (1) combined with the fact that every arc is incident to an inactive vertex, because the active vertices form a stable set of $D$.
{\bf (3)} follows from (2) and the fact that $\delta(v)$ is a dicut of weight $\tau+1$.
\end{proof}

\begin{DE}
Let $J\subseteq A$. We say that $J$ is a \emph{rounded $1$-factor of $(D,w)$} if for each vertex $v$, $|J\cap \delta(v)|$ is $\frac{w(\delta(v))}{\tau}$ rounded up or down. For a rounded $1$-factor $J$, a \emph{dyad center} is a vertex incident with two arcs from~$J$; such a pair of arcs is called a \emph{dyad}; denote by $\dc(J)$ the set of dyad centers of $J$.
\end{DE}

Observe that a rounded $1$-factor is the vertex disjoint union of arcs and dyads saturating every vertex of $D$. Observe further that a dyad center is necessarily active. Using the following general result for bipartite graphs, we get a partition of $A_1$ into $\tau$ rounded $1$-factors.

\begin{theorem}[de Werra~\cite{deWerra71}, see  \cite{LP09}, Corollary 1.4.21]\label{deWerra}
Let $G=(V,E)$ be a bipartite graph, and $k\geq 1$ an integer. Then $E$ can be partitioned into $k$ sets $J_1,\ldots,J_k$ such that $|J_i\cap \delta(v)|$ is $\frac{|\delta(v)|}{k}$ rounded up or down, for each $i\in [k]$ and $v\in V$.
\end{theorem}

Subsequently,

\begin{LE}\label{R1F-decomposition}
$A_1$ can be partitioned into $\tau$ rounded $1$-factors.
\end{LE}
\begin{proof}
This is an immediate consequence of \Cref{deWerra} applied to $G=D[A_1]$ and $k=\tau$.
\end{proof}

\subsection{Discrepancy}

\begin{DE}
For each sink $u$ of $(D,w)$, denote $\disc(u):=1$, and for each source $u$, denote $\disc(u):=-1$. For every $U\subseteq V$, the \emph{discrepancy of $U$} in $(D,w)$, denoted $\disc(U)$, is the number of sinks in $U$ minus the number of sources in $U$, that is, $\disc(U) = \sum_{u\in U}\disc(u)$.
\end{DE}

Observe that $\disc:2^V\to \mathbb{Z}$ is a \emph{modular function}, that is, $\disc(U\cap W)+\disc(U\cup W)=\disc(U)+\disc(W)$ for all $U,W\subseteq V$.

\begin{LE}\label{disc}
The following statements hold: \begin{enumerate}[(1)]
\item $|a(V)|=\tau\cdot \disc(V)$, and $\rho(\tau,D,w) = \disc(V)$,
\item for every dicut $\delta^+(U)$ of $D$, $w(\delta^+(U))=|a(U)|-\tau\cdot \disc(U)$,
\item for every dicut $\delta^+(U)$ of $D$, $\disc(U)\leq \disc(V)-1$, and if equality holds, then $w(\delta^+(U))=\tau$ and $a(U)=a(V)$.
\end{enumerate}
\end{LE}
\begin{proof}
{\bf (1)} Let us double-count $w(A)$. On one hand, $w(A)=\sum_{v \text{ a sink}}w(\delta^-(v))=\tau\cdot |\{v:v \text{ a sink}\}|$, where the last equality holds because $(D,w)$ is sink-regular. On the other hand, $$w(A)=\sum_{v \text{ a source}}w(\delta^+(v))=\tau\cdot |\{v:v \text{ a source}\}|+|a(V)|.$$ Thus, $|a(V)|=\tau(|\{v:v \text{ a sink}\}|-|\{v:v \text{ a source}\}|)=\tau\cdot \disc(V)$. Moreover, $$\rho(\tau,D,w)=\frac{1}{\tau}\sum_{v\in V} (w(\delta^+(v))-w(\delta^-(v)) \mod{\tau})=\frac{1}{\tau}|a(V)|=\disc(V),$$ where the second to last equality holds because $(D,w)$ is sink-regular. Thus, (1) holds.
{\bf (2)} We have \begin{align*}
w(\delta^+(U))&=w(\delta^+(U))-w(\delta^-(U))\\
&=\sum_{v\in U} (w(\delta^+(v))-w(\delta^-(v)))\\
&=|a(U)|+\tau(|\{v:v \text{ a source in $U$}\}|-|\{v:v \text{ a sink in $U$}\}|)\\
&=|a(U)| - \tau\cdot \disc(U).
\end{align*}
{\bf (3)} Since $(D,w)$ is a weighted $(\tau,\tau+1)$-bipartite digraph, $w(\delta^+(U))\geq \tau$, so $|a(U)|-\tau\cdot\disc(U)\geq \tau$ by (2), implying in turn that $|a(U)|\geq \tau(1+\disc(U))$. Moreover, $|a(V)|\geq |a(U)|$, and $|a(V)|=\tau\cdot\disc(V)$ by (1). Combining these we get $\disc(V)\geq 1+\disc(U)$. If equality holds here, then it must hold throughout, so $w(\delta^+(U))=\tau$ and $a(V)=a(U)$, as claimed.
\end{proof}

\begin{LE}\label{R1F}
Let $J\subseteq A$ be a rounded $1$-factor. Then the following statements hold: \begin{enumerate}[(1)]
\item $|\dc(J)| = \disc(V)$,
\item
$|J\cap \delta^+(U)|= |\dc(J)\cap U|- \disc(U)$
for every dicut $\delta^+(U)$ of $D$,
\item $J$ is a dijoin of $D$ if, and only if,
$|\dc(J)\cap U| \geq 1+ \disc(U)$ for every dicut $\delta^+(U)$ of $D$,
\item Let $J_1,\ldots,J_\tau$ be a partition of $A_1$ into rounded $1$-factors. Pick $i,j\in [\tau]$.
Then for every dicut $\delta^+(U)$ of $D$,
$$
|J_i\cap \delta^+(U)| - |J_j\cap \delta^+(U)| =
|\dc(J_i)\cap U| - |\dc(J_j)\cap U|,
$$ and so $$
-\disc(V)\leq
|J_i\cap \delta^+(U)| - |J_j\cap \delta^+(U)|\leq \disc(V).
$$
\end{enumerate}
\end{LE}
\begin{proof}
{\bf (1)} Since $(D,w)$ is sink-regular, every active vertex is a source, so every dyad center is a source. We know that $J$ is the vertex-disjoint union of arcs and dyads saturating every vertex. A simple double-counting tells us that the number of dyads of $J$ is $\disc(V)$, so $|\dc(J)|=\disc(V)$.
{\bf (2)} We have \begin{align*}
|J\cap \delta^+(U)| &= |J\cap \delta^+(U)|-|J\cap \delta^-(U)|\\
&=\sum_{v\in U} (|J\cap \delta^+(v)|-|J\cap \delta^-(v)|)\\
&=|\dc(J)\cap U|+|\{v:v\text{ a source in $U$}\}|-|\{v:v\text{ a sink in $U$}\}|\\
&=|\dc(J)\cap U|-\disc(U)
\end{align*} where the third equality uses the fact that every vertex in $U$ has exactly one arc in $J$ incident to it, except for the dyad centers of $J$, which are active and therefore sources, and have exactly two arcs in $J$. {\bf (3)} and {\bf (4)} are immediate consequences of (1) and (2).
\end{proof}

\begin{RE}\label{R1F-remark}
Observe that the equality in \Cref{disc}~(2) holds more generally for every dicut of $D[A_1]$. Also, the (in)equalities of \Cref{R1F}~(2) and~(4) hold more generally for every dicut of $D[A_1]$ if $J\subseteq A_1$, which will be the case in most applications.

Note the significance of \Cref{R1F}~(2) and~(3): a rounded $1$-factor being a dijoin is solely a function of its dyad centers, and not of the arcs.
\end{RE}

\subsection{$[[\wt,\tau,1;\eqt]]$ is true.}

\begin{theorem}\label{rho=1-sink-reg}
Let $(D=(V,A),w)$ be a sink-regular weighted $(\tau,\tau+1)$-bipartite digraph such that $\rho(\tau,D,w)\leq 1$. Then there exists an equitable $w$-weighted packing of dijoins of size $\tau$.
\end{theorem}
\begin{proof}
By \Cref{R1F-decomposition}, $A_1$ can be partitioned into $\tau$ rounded $1$-factors $J_1,\ldots,J_\tau$.
We claim that each $J_i$ is a dijoin of $D$. By \Cref{disc}~(1), $\rho(\tau,D,w) = \disc(V)$, so $\disc(V)\in \{0,1\}$. Let $\delta^+(U)$ be a dicut of~$D$. It suffices to prove that $J_i\cap \delta^+(U)\neq \emptyset$ for each $i\in [\tau]$. By \Cref{R1F}~(4), for all $i,j\in [\tau]$, \begin{equation}\tag{$\star$}-1\leq |J_i\cap \delta^+(U)| - |J_j\cap \delta^+(U)| \leq 1.\end{equation}
 Since the dicut $\delta^+(U)$ has weight at least $\tau$, and $(J_i\cap \delta^+(U):i\in [\tau])$ partition the arcs in $A_1\cap \delta^+(U)$, $(\star)$ implies that $|J_i\cap \delta^+(U)|>0$ for each $i\in [\tau]$. Thus, each $J_i,i\in [\tau]$ is a dijoin of $D$. In fact, $(\star)$ implies that $J_1,\ldots,J_\tau$ is an equitable packing of dijoins, as required.
\end{proof}

\begin{theorem}\label{rho=1}
Let $(D=(V,A),w)$ be a weighted digraph where every dicut has weight at least $\tau$, and $\tau\geq 2$. Suppose $\rho(\tau,D,w)\leq 1$. Then there exists an equitable $w$-weighted packing of dijoins of size $\tau$. That is, $[[\wt,\tau,1;\eqt]]$ is true.
\end{theorem}
\begin{proof}
By \Cref{rho=1-sink-reg}, $[[\wt,\tau,1;\eqt]]$ holds for all sink-regular weighted $(\tau,\tau+1)$-bipartite digraphs. In particular, by \Cref{weighted-reduction}~(2), $[[\wt,\tau,1;\eqt]]$ is true.
\end{proof}

\begin{CO}\label{weighted-GCD}
Let $(D=(V,A),w)$ be a weighted digraph, and let $g:=\gcd\{w(\delta^+(v))-w(\delta^-(v)):v\in V\}$. Then there exists an equitable $w$-weighted packing of dijoins of size $g$.
\end{CO}
\begin{proof}
Observe that $g=
\gcd\{w(\delta^+(U))-w(\delta^-(U)):\emptyset\neq U\subsetneq V\}$. In particular, every dicut has weight at least $g$. If $g=1$, then we are done. Otherwise, $g\geq 2$. Then $\rho(g,D,w)=0$, so the result follows from $[[\wt,g,0;\eqt]]$, which holds by \Cref{rho=1}.
\end{proof}

\section{Packing a dijoin and a $(\tau-1)$-dijoin in digraphs}\label{sec:two-dijoins}

In this section we prove \hyperlink{primary-results}{{\bf P1}}. To this end, we need to recall some notions from Combinatorial Optimization~\cite{Schrijver03,Frank11}.

\subsection{Crossing families, box-TDI systems, and the integer decomposition property}

Let $\mathcal{U}$ be a family of subsets of a finite ground set $V$. A pair of sets $U,W\in \mathcal{U}$ is \emph{crossing} if $U\cap W\neq \emptyset$ and $U\cup W\neq V$. $\mathcal{U}$ is a \emph{crossing family} if $U\cap W,U\cup W\in \mathcal{U}$ for all crossing pairs $U,W\in \mathcal{U}$.

\begin{RE}\label{dicut-crossing}
Let $D=(V,A)$ be a digraph. Then $\{U\subseteq V: \text{ $\delta^+(U)$ is a dicut of $D$}\}$ is a crossing family.
\end{RE}

Given a crossing family $\mathcal{U}$, a function $g:\mathcal{U}\to \mathbb{Z}$ is \emph{crossing submodular} if $g(U)+g(W)\geq g(U\cap W)+g(U\cup W)$ for all crossing pairs $U,W\in \mathcal{U}$.

\begin{theorem}[Frank and Tardos~\cite{Frank84}, see \cite{Schrijver03}, Theorem 49.7a]\label{crossing-matroid}
Let $\mathcal{U}$ be a crossing family over ground set $V$, $g:\mathcal{U}\to \mathbb{Z}$ a crossing submodular function, and $k\geq 1$ an integer. Then $\{B\subseteq V:|B|=k; |B\cap U|\leq g(U) ~\forall U\in \mathcal{U}\}$, if nonempty, is the set of bases of a matroid.
\end{theorem}

Given $A\in \mathbb{Z}^{m\times n}$ and $b\in \mathbb{Z}^m$, the linear system $Ax\leq b$ is \emph{totally dual integral (TDI)} if the linear program $\min\{y^\top b: A^\top y = w, y\geq \0\}$, if feasible and bounded, has an integral optimal solution for each $w\in \mathbb{Z}^n$~\cite{Edmonds77}. Observe that the linear program is the dual of $\max\{w^\top x:Ax\leq b\}$. The system $Ax\leq b$ is \emph{box-TDI} if the system $Ax\leq b, d\leq x\leq c$ is TDI for all $d,c\in \mathbb{Z}^n$ such that $d\leq c$.

\begin{theorem}[Fujishige~\cite{Fujishige84}, see \cite{Schrijver03}, Theorem 49.8]\label{box-TDI}
Let $\mathcal{C}_i,i\in [2]$ be a crossing family over ground set $V$, $g_i:\mathcal{C}\to \mathbb{Z},i\in[2]$ a crossing submodular function, and $k$ an integer. Then the system $x(V)=k$; $x(U)\leq g_1(U)~\forall U\in \mathcal{C}_1$; $x(U)\leq g_2(U)~\forall U\in \mathcal{C}_2$ is box-TDI.
\end{theorem}

An important result is that if $Ax\leq b$ is TDI (and $b\in \mathbb{Z}^m$), then the primal linear program $\max\{w^\top x:Ax\leq b\}$, if feasible and bounded, has an integral optimal solution for each $w\in \mathbb{Z}^n$, that is, $\{x:Ax\leq b\}$ is an \emph{integral} polyhedron~\cite{Hoffman74,Edmonds77}. In particular, if $Ax\leq b$ is box-TDI, then $\{x:Ax\leq b,c\leq x\leq d\}$ is an integral polytope for all $c,d\in \mathbb{Z}^n$ such that $d\leq c$.

A polyhedron $P$ has the \emph{integer decomposition property}
if for every integer $k\geq 1$, every integral point in $kP$ can be written as the sum of $k$ integral points in $P$, that is, every integral point written as the sum of $k$ points in $P$ can be written as the sum of $k$ integral points in $P$.

\begin{theorem}[Edmonds~\cite{Edmonds65}, see \cite{Schrijver03}, Corollary 42.1e]\label{IDP}
The base polytope of a matroid has the integer decomposition property.
\end{theorem}

\subsection{The matroid $M_1$, and basis partitions}

Throughout this subsection we are given an integer $\tau\geq 2$, a weighted $(\tau,\tau+1)$-bipartite digraph $(D=(V,A),w)$ that is sink-regular, and $w\in \{0,1\}^A$. 

\begin{LE}\label{R1F-M1}
Let $J$ be a rounded $1$-factor, and $Q:=\dc(J)$. Then $|Q|=\disc(V)$. Moreover, $J$ is a dijoin of $D$ if, and only if, $|Q\cap U|\geq 1+\disc(U)$ for every dicut $\delta^+(U)$ of $D$.
\end{LE}
\begin{proof}
This is a restatement of \Cref{R1F}~(1) and~(3).
\end{proof}

Let us study subsets of $V$ satisfying the equality and inequalities above.

\begin{theorem}\label{M1-major-matroid}
$\{Q\subseteq V:|Q|=\disc(V); |Q\cap U|\geq 1+\disc(U) \text{ for every dicut $\delta^+(U)$ of $D$}\}$, if nonempty, is the set of bases of a matroid.
\end{theorem}
\begin{proof}
By \Cref{dicut-crossing}, $\mathcal{U}:=\{U\subseteq V:\delta^+(U) \text{ is a dicut of $D$}\}$ is a crossing family over ground set~$V$. Let $$
\mathcal{Q}:=\{Q\subseteq V:|Q|=\disc(V), |Q\cap U|\geq 1+\disc(U)~\forall U\in \mathcal{U}\}.
$$ Assume that $\mathcal{Q}$ is nonempty. The family of the complements of the sets in $\mathcal{Q}$ can be described as $$
\{\overline{Q}\subseteq V:|\overline{Q}|=|V|-\disc(V), |\overline{Q}\cap U|\leq |U|-1-\disc(U)~\forall U\in \mathcal{U}\}.
$$ Since $g(U):=|U|-1-\disc(U)$ is a modular, hence submodular, function, and since the family above is nonempty, it follows from \Cref{crossing-matroid} that the family forms the set of bases of a matroid, implying in turn that the sets in $\mathcal{Q}$ form the bases of the dual matroid.
\end{proof}

We shortly prove by using \Cref{box-TDI} that the family in \Cref{M1-major-matroid} is indeed nonempty. For now, recall that $a(V)$ denotes the set of active vertices of $(D,w)$.

\begin{DE}\label{M1-matroid}
Let $M_1(D,w)$ be the matroid over ground set $a(V)$ whose bases are the sets in $\{Q\subseteq a(V):|Q|=\disc(V), |Q\cap U|\geq 1+\disc(U) \text{ for every dicut $\delta^+(U)$ of $D$}\}$.
\end{DE}

Observe that $M_1(D,w)$ is the restriction of the matroid in \Cref{M1-major-matroid} to $a(V)$. Note that $a(V)$ is determined by $w$.

Observe that if $\{a\in A:w_a=1\}$ contained $\tau$ disjoint dijoins, then by \Cref{R1F-birth} and \Cref{R1F-M1}, $a(V)$ could be partitioned in $\tau$ disjoint bases of $M_1(D,w)$ -- let us verify this consequence independently of the assumption.

\begin{theorem}\label{M1-basis-partition}
The ground set of $M_1(D,w)$ can be partitioned into $\tau$ bases.
\end{theorem}
\begin{proof}
By \Cref{dicut-crossing}, $\mathcal{U}:=\{U\subseteq V:\delta^+(U) \text{ is a dicut of $D$}\}$ is a crossing family over ground set $V$. Consider the system $x(V)=\disc(V)$, $x(U)\geq 1+\disc(U)~\forall U\in \mathcal{U}$. By \Cref{box-TDI}, this linear system is box-TDI. In particular, the polytope $P$ defined by \begin{equation*}
\begin{array}{rll}
x(V)&=\disc(V)&\\
x(U)&\geq 1+\disc(U)&\forall U\in \mathcal{U}\\
x_u&\in [0,1]&\forall u\in a(V)\\
x_u&=0&\forall u\in V-a(V)
\end{array}
\end{equation*}
is integral. Observe that $P$ is the base polytope of the matroid $M_1(D,w)$, so by \Cref{IDP}, $P$ has the integer decomposition property. Let $x:=\chi_{a(V)}\in \{0,1\}^V$, the incidence vector of $a(V)$. Then $x(V) = |a(V)|=\tau\cdot \disc(V)$ by \Cref{disc}~(1). Moreover, by \Cref{disc}~(2),
$$x(U)-\tau\cdot \disc(U) = |a(U)| - \tau\cdot \disc(U) = w(\delta^+(U))$$
for every dicut $\delta^+(U)$ of $D$. In particular,
$x(U)\geq \tau\cdot (1+\disc(U))$ for every dicut $\delta^+(U)$ of $D$, as every dicut of $D$ has weight at least $\tau$.
Thus, $\frac{1}{\tau}x\in P$, and so $x$ is the sum of $\tau$ points in $P$. Thus, by the integer decomposition property of $P$, $x$ is the sum of $\tau$ integer points in $P$. That is, $a(V)$ admits a partition into $\tau$ bases of $M_1(D,w)$.
\end{proof}

\subsection{Perfect $b$-matchings}

Let $G=(V,E)$ be a graph, and $b\in \mathbb{Z}^V_{\geq 0}$. A vector $x\in \mathbb{Z}^E_{\geq 0}$ is a \emph{perfect $b$-matching} if $x(\delta(v))=b_v$ for each $v\in V$. A subset $J\subseteq E$ is a \emph{perfect $b$-matching} if $\chi_J$ is a perfect $b$-matching. A \emph{vertex cover} of $G$ is a vertex subset that contains at least one end of every edge.

\begin{theorem}[\cite{Schrijver03}, Corollary 21.1b]\label{perfect-b-matching-char}
Let $G = (V,E)$ be a bipartite graph, and $b\in \mathbb{Z}^V_{\geq 0}$. Then there exists a perfect $b$-matching $x\in \mathbb{Z}^E_{\geq 0}$ if, and only if, $b(K)\geq\frac{1}{2}b(V)$ for each vertex cover $K$.
\end{theorem}

The theorem above has a neat reformulation in terms of bipartite digraphs and dicuts.

\begin{theorem}\label{perfect-b-matching}
Let $D=(V,A)$ be a bipartite digraph with sources $S$ and sinks $T$, where every vertex has nonzero degree. Let $b\in \mathbb{Z}^V_{\geq 0}$. Then there exists a perfect $b$-matching $x\in \mathbb{Z}^A_{\geq 0}$ if, and only if, $b(S)=b(T)$ and
$b(U\cap S) - b(U\cap T)\geq 0$
for every dicut $\delta^+(U)$.
\end{theorem}
\begin{proof}
$(\Rightarrow)$ Let $x\in \mathbb{Z}^A_{\geq 0}$ be a perfect $b$-matching. Clearly, $b(S)=b(T)$. Let $\delta^+(U)$ be a dicut. Then $\delta^-(U)=\emptyset$, so
$$
x(\delta^+(U)) =x(\delta^+(U))-x(\delta^-(U))
=\sum_{u\in U} \left(x(\delta^+(u))-x(\delta^-(u))\right)=b(U\cap S)-b(U\cap T),$$
so $b(U\cap S)-b(U\cap T)\geq 0$.

$(\Leftarrow)$ Let $K\subseteq V$ be a vertex cover of the underlying undirected graph of $D$. If $K$ contains $S$ or $T$, then $b(K)\geq \frac{1}{2}b(V)$ since $b(S)=b(T)$. Otherwise, let $X:=K\cap S,Z:=K\cap T$ and $Y:=T-Z$. Then $X\cup Y\neq \emptyset,V$. Moreover, since $K$ is a vertex cover of the underlying undirected graph, there is no arc in $D[V-K]$, so $\delta^+(X\cup Y)$ is a dicut of $D$. Subsequently, $b(X)\geq b(Y)$ by the hypothesis. Thus,
$$
b(K)-b(T) = b(X)+b(Z)-b(Y)-b(Z)=b(X)-b(Y)\geq 0,
$$ so $b(K)\geq \frac{1}{2}b(V)$. Thus, by \Cref{perfect-b-matching-char}, there exists a perfect $b$-matching $x\in \mathbb{Z}^A_{\geq 0}$.
\end{proof}

The connection to dicuts allows us to bring $k$-dijoins into the picture as well.

\begin{LE}\label{k-dijoin}
Let $\tau\geq 2$ be an integer, and $(D=(V,A),w)$ a sink-regular weighted $(\tau,\tau+1)$-bipartite digraph, where $w\in \{0,1\}^A$. Let $Q_1,\ldots,Q_k$ be disjoint bases of $M_1(D,w)$,  $b:=k\cdot\chi_V+\sum_{i=1}^{k} \chi_{Q_i}$, and $J\subseteq A$ a perfect $b$-matching. Then $J$ is a $k$-dijoin.
\end{LE}
\begin{proof}
Denote by $S$ the set of sources, and by $T$ the set of sinks of $D$. Let $J\subseteq A$ be a perfect $b$-matching. Let $\delta^+(U)$ be a dicut of $D$. Then
\begin{align*}
|J\cap \delta^+(U)| &=|J\cap \delta^+(U)|-|J\cap \delta^-(U)|\\
&=\sum_{u\in U} \left(|J\cap \delta^+(u)|-|J\cap\delta^-(u)|\right)\\
&=b(U\cap S)-b(U\cap T)\\
&= \sum_{i=1}^{k}\big(|U\cap Q_i| - \disc(U)\big)\\
&\geq k
\end{align*}
where the last inequality holds because each $Q_i$ is a basis of $M_1(D,w)$ so $|U\cap Q_i|\geq 1+\disc(U)$. As the inequality above holds for every dicut $\delta^+(U)$, it follows that $J$ is a $k$-dijoin.\end{proof}

\subsection{Packing a dijoin and a $(\tau-1)$-dijoin in digraphs}

\begin{theorem}\label{two-dijoins-sink-reg}
Let $\tau\geq 2$ be an integer, and $D=(V,A)$ a sink-regular $(\tau,\tau+1)$-bipartite digraph. Then $A$ can be partitioned into a dijoin and a $(\tau-1)$-dijoin.
\end{theorem}
\begin{proof}
Consider the sink-regular weighted $(\tau,\tau+1)$-bipartite digraph $(D,\1)$. By \Cref{M1-basis-partition}, $M_1(D,\1)$ has disjoint bases $Q_1,\ldots,Q_\tau$.
Let $b:=\chi_V+\chi_{Q_1}$. Let $S$ be the set of sources, and $T$ the set of sinks of $D$.

\begin{claim} 
For every dicut $\delta^+(U)$, $b(U\cap S)-b(U\cap T)\geq 0$.
\end{claim}
\begin{cproof}
We have
$$b(U\cap S)-b(U\cap T) = |U\cap Q_1| - \disc(U)\geq 1
$$ where the last inequality holds because $Q_1$ is a basis of $M_1(D,\1)$. In particular, Claim~1 follows.
\end{cproof}

\begin{claim} 
There exists a perfect $b$-matching $J\subseteq A$.
\end{claim}
\begin{cproof}
Observe that $$b(S)=|S|+ |Q_1|=|S|+\disc(V)=
|T|=b(T).$$
Moreover, by Claim~1, $b(U\cap S)-b(U\cap T)\geq 0$ for every dicut $\delta^+(U)$. Thus, by \Cref{perfect-b-matching}, there exists a perfect $b$-matching $x\in \mathbb{Z}^A_{\geq 0}$. Since $b_u=1$ for every sink $u$, it follows that $x\leq {\bf 1}$, so $x$ is the incidence vector of a subset $J\subseteq A$, which is a perfect $b$-matching by definition.
\end{cproof}

By \Cref{k-dijoin}, $J$ is a dijoin. Let $\bar{b}:= (\tau-1)\cdot \chi_V+\sum_{i=2}^{\tau}\chi_{Q_i}$. The vertex degrees of $D$ imply that $A$ is a perfect $(b+\bar{b})$-matching. In particular, by the choice of $J$, $A-J$ is a perfect $\bar{b}$-matching. It now follows from \Cref{k-dijoin} that $A-J$ is a $(\tau-1)$-dijoin. Thus, we have a partition of $A$ into a dijoin $J$ and a $(\tau-1)$-dijoin $A-J$, as desired.
\end{proof}

The theorem above does not extend to sink-regular weighted $(\tau,\tau+1)$-bipartite digraphs, because Claim~2 of the proof does not extend to the weighted setting. 
We shall extend \Cref{two-dijoins-sink-reg} appropriately to this setting in \S\ref{sec:two-dijoins-weighted}. 

\begin{theorem}\label{two-dijoins}
Let $D=(V,A)$ be a digraph where every dicut has size at least $\tau$, and $\tau\geq 2$. Then $A$ can be partitioned into a dijoin and a $(\tau-1)$-dijoin. That is, there exists a dijoin $J\subseteq A$ such that $|\delta^+(U)-J|\geq \tau-1$ for every dicut $\delta^+(U)$.
\end{theorem}
\begin{proof}
If $\tau=2$, then the result follows from the correctness of the statement $[[\tau]]$ for $\tau=2$.
Otherwise, $\tau\geq 3$. The result now follows from \Cref{reduction}~(3) and \Cref{two-dijoins-sink-reg}.
\end{proof}

Observe that if $D$ had a packing of $\tau$ dijoins, then every dijoin $J$ in the packing would satisfy the conclusion of \Cref{two-dijoins}. The reader may wonder if any dijoin satisfying the conclusion of \Cref{two-dijoins} belongs to such a packing; we already saw in \S\ref{subsec:examples} that this unfortunately is not the case.

\section{$[[\wt,\tau,2]]$ is true.}\label{sec:rho=2}

In this section we prove \hyperlink{primary-results}{{\bf P3}} and \hyperlink{secondary-results}{{\bf S1}}.. Throughout the section, unless stated otherwise, we are given an integer $\tau\geq 2$, a weighted $(\tau,\tau+1)$-bipartite digraph $(D=(V,A),w)$ that is sink-regular, and $w\in \{0,1\}^A$. Let $A_1:=\{a\in A:w_a=1\}$.

\subsection{Alternating circuits, cycles and paths}

\begin{DE}
Let $J_1,J_2$ be rounded $1$-factors. A \emph{$\{J_1,J_2\}$-alternating circuit} is a circuit $C$ contained in $J_1\tr J_2$ whose arcs alternatively belong to $J_1,J_2$. A \emph{$\{J_1,J_2\}$-alternating cycle} is a (possibly empty) arc-disjoint union of $\{J_1,J_2\}$-alternating circuits, that is, it is a subset $C\subseteq J_1\tr J_2$ such that $|\delta(v)\cap C\cap J_1|=|\delta(v)\cap C\cap J_2|$ for every vertex $v$.
\end{DE}

Note that $\emptyset$ is a $\{J_1,J_2\}$-alternating cycle, but not a $\{J_1,J_2\}$-alternating circuit.

\begin{DE}
Let $J_1,J_2$ be rounded $1$-factors, and let $Q_1:=\dc(J_1)$ and $Q_2:=\dc(J_2)$. A \emph{$(J_1,J_2)$-alternating path} is a $(u,v)$-path contained in $J_1\tr J_2$, for some $u\in Q_1-Q_2$ and $v\in Q_2-Q_1$, whose arcs alternatively belong to $J_1,J_2$ with the first arc belonging to $J_1$.
\end{DE}

\begin{LE}\label{alternating}
Let $J_1,J_2$ be rounded $1$-factors, $Q_1:=\dc(J_1)$ and $Q_2:=\dc(J_2)$, and let $2k:=Q_1\tr Q_2$ for some integer $k\geq 0$. Then there exists a bijection $\pi:Q_1-Q_2\to Q_2-Q_1$ such that the following statements hold: \begin{enumerate}[(1)]
\item $J_1\tr J_2$ can be decomposed into a $\{J_1,J_2\}$-alternating cycle, and $k$ $(J_1,J_2)$-alternating paths $P_1,\ldots,P_k$ whose ends are $\{u_1,\pi(u_1)\},\ldots,\{u_k,\pi(u_k)\}$, respectively.
\item For any $X\subseteq [k]$, $J_1\tr \left(\cup_{i\in X} P_i\right)$ and $J_2\tr \left(\cup_{i\in X} P_i\right)$ are rounded $1$-factors with dyad centers $Q_1\tr \left(\cup_{i\in X} \{u_i,\pi(u_i)\}\right)$ and $Q_2\tr \left(\cup_{i\in X} \{u_i,\pi(u_i)\}\right)$, respectively.
\end{enumerate}
\end{LE}
\begin{proof}
{\bf (1)}
Let $D'$ be the digraph obtained from $D[J_1\tr J_2]$ after reversing the arcs in $J_2-J_1$. For each $v\in V$, let $\dfc(v):=|\delta_{D'}^+(v)|-|\delta_{D'}^-(v)|$. Observe that $\dfc(v)=0$ for every vertex $v\notin Q_1\tr Q_2$, $\dfc(v)=1$ for every vertex $v\in Q_1-Q_2$, and $\dfc(v)=-1$ for every vertex $v\in Q_2-Q_1$.
It can now be readily checked that $A(D')$ can be decomposed into arc subsets $P$, where $P$ is a directed path or circuit of $D'$; if $P$ is a path, then it starts at a vertex $v$ with $\dfc(v)=1$ and ends at a vertex with $\dfc(v)=-1$. (This is a routine topological argument, and follows, for example, from \cite{Schrijver03}, Theorem 11.1.) The paths in the decomposition provide the desired $\pi$, and the decomposition itself gives us (1).
{\bf (2)} is immediate.
\end{proof}

\subsection{The matroid $M_0$, and bimatchability}

\begin{DE}
Let $Q\subseteq a(V)$. We say that $Q$ is \emph{bimatchable in $(D,w)$} if $Q=\dc(J)$ for some rounded $1$-factor $J$ contained in $A_1$.
\end{DE}

\begin{theorem}\label{M0-SBO}
$\{Q\subseteq a(V):Q \text{ is bimatchable in $(D,w)$}\}$ is the set of bases of a strongly base orderable matroid, one whose ground set can be partitioned into bases.
\end{theorem}
\begin{proof}
By \Cref{R1F-birth}~(3) and \Cref{R1F-decomposition}, $a(V)$ can be partitioned into bimatchable sets of $(D,w)$. By \Cref{R1F}~(1), every bimatchable set has the same size, namely $\disc(V)$. To finish the proof, it remains to prove that for every two bimatchable sets $Q_1,Q_2$, there exists a bijection $\pi:Q_1-Q_2\to Q_2-Q_1$ such that $Q_1\tr (X\cup \pi(X)),Q_2\tr (X\cup \pi(X))$ are bimatchable for all $X\subseteq Q_1-Q_2$. To this end, let $J_1,J_2$ be rounded $1$-factors such that $\dc(J_i)=Q_i$ for $i=1,2$. Then the bijection from \Cref{alternating} is the desired one.
\end{proof}

\begin{DE}
Let $M_0(D,w)$ be the matroid over ground set $a(V)$ whose bases are the bimatchable sets of $(D,w)$.
\end{DE}

\begin{LE}\label{M0-matroid}
Let $Q\subseteq a(V)$. Then $Q$ is a basis of $M_0(D,w)$ if, and only if, $|Q|=\disc(V)$, and $|Q\cap U|\geq \disc(U)$ for every dicut $\delta^+(U)$ of $D[A_1]$.
\end{LE}
\begin{proof}
$(\Rightarrow)$ Suppose $Q$ is a basis of $M_0(D,w)$, i.e. $Q$ is a bimatchable set of $(D,w)$. Let $J\subseteq A_1$ be a rounded $1$-factor of $D[A_1]$ such that $\dc(J)=Q$. By \Cref{R1F}~(1) and~(2) and \Cref{R1F-remark}, $|Q|=\disc(V)$, and for every dicut $\delta^+(U)$ of $D[A_1]$, $|Q\cap U|-\disc(U)=|J\cap \delta^+(U)|\geq 0$ so $|Q\cap U|\geq\disc(U)$.
$(\Leftarrow)$ Suppose $|Q|=\disc(V)$, and $|Q\cap U|\geq \disc(U)$ for every dicut $\delta^+(U)$ of $D[A_1]$. Let $b:=\chi_{V}+\chi_{Q}$, $S$ the set of sources, and $T$ the set of sinks of $D[A_1]$. Then $b(S)=|S|+|Q|=|T|=b(T)$. Moreover, for every dicut $\delta^+(U)$ of $D[A_1]$, $b(U\cap S)-b(U\cap T)=|U\cap Q|-\disc(U)\geq 0$. Thus, by \Cref{perfect-b-matching}, there exists a perfect $b$-matching $J\subseteq A_1$. Observe that $J$ is a rounded $1$-factor in $A_1$ such that $\dc(J)=Q$, so $Q$ is bimatchable in $(D,w)$, so $Q$ is a basis of $M_0(D,w)$.
\end{proof}

\subsection{Common bases of $M_0,M_1$, and admissibility}\label{subsec:admissible}

\begin{DE}
Let $Q\subseteq a(V)$. We say that $Q$ is \emph{an admissible set of $(D,w)$} if it is a common basis of $M_0(D,w)$ and $M_1(D,w)$.
\end{DE}

\begin{LE}\label{admissible}
Let $Q\subseteq a(V)$. Then the following statements are equivalent: \begin{enumerate}[(1)]
\item $Q$ is admissible in $(D,w)$,
\item $Q$ is a bimatchable set of $(D,w)$, and every rounded $1$-factor $J$ satisfying $\dc(J)=Q$ is a dijoin of $D$,
\item $Q$ is a bimatchable set of $(D,w)$, and some rounded $1$-factor $J$ satisfying $\dc(J)=Q$ is a dijoin of $D$.
\end{enumerate}
\end{LE}
\begin{proof}
The equivalence of (1) and (2) follows from the definition of admissibility, \Cref{R1F}~(3), and the definitions of $M_0(D,w)$ and $M_1(D,w)$.
The equivalence of (2) and (3) is an immediate consequence of \Cref{R1F}~(2) and~(3) and \Cref{R1F-remark}.
\end{proof}

\begin{RE}\label{unweighted-admissible}
Suppose $w=\1$. Let $Q\subseteq a(V)$. Then $Q$ is admissible in $(D,w)$ if, and only if, $Q$ is a basis of $M_1(D,w)$.
\end{RE}
\begin{proof}
$(\Rightarrow)$ holds clearly. $(\Leftarrow)$ Clearly every dicut of $D$ is also a dicut of $D[A_1]$, and since $A_1=A$, every dicut of $D[A_1]$ is also a dicut of $D$. Thus, every basis of $M_1(D,w)$ is also a basis of $M_0(D,w)$ by \Cref{M0-matroid}. Thus,  every basis of $M_1(D,w)$ is admissible in $(D,w)$.
\end{proof}

Looking back, this remark illustrates why \Cref{two-dijoins-sink-reg} works for unweighted digraphs, and not necessarily weighted digraphs: in the unweighted setting, unlike in the weighted setting, being a basis of $M_1(D,w)$ implies admissibility, which in turn implies the existence of perfect $b$-matchings in $A_1$.

\subsection{Packing admissible sets under strong base orderability}

\begin{theorem}[Davies and McDiarmid~\cite{Davies76}, see \cite{Schrijver03}, Theorem 42.13]\label{SBO-common-base-packing}
Let $M_0,M_1$ be matroids over the same ground set, and suppose that the ground set can be partitioned into $\tau$ bases of $M_i$, for $i=0,1$. If $M_0,M_1$ are strongly base orderable, then the ground set can be partitioned into $\tau$ common bases of $M_0,M_1$.
\end{theorem}

\begin{theorem}\label{SBO-admissible-partition}
Suppose $M_1(D,w)$ is strongly base orderable. Then the ground set $a(V)$ can be partitioned into $\tau$ admissible sets.
\end{theorem}
\begin{proof}
Recall that a set is admissible if and only it is a common basis of $M_0(D,w)$ and $M_1(D,w)$. Thus, our task is to partition $a(V)$ into $\tau$ common bases of $M_0(D,w)$ and $M_1(D,w)$. We know by \Cref{M0-SBO} that $M_0(D,w)$ is a strongly base orderable matroid whose ground set $a(V)$ can be partitioned into $\tau$ bases. We know by \Cref{M1-basis-partition} that $a(V)$ can also be partitioned into~$\tau$ bases of $M_1(D,w)$, and $M_1(D,w)$ is a strongly base orderable matroid by the hypothesis. Thus, by \Cref{SBO-common-base-packing}, $a(V)$ can be partitioned into $\tau$ common bases of $M_0(D,w)$ and $M_1(D,w)$.~\end{proof}

\subsection{Weighted packings under strong base orderability}

\begin{LE}[Brualdi~\cite{Brualdi70}]\label{SBO-minor}
If a matroid is strongly base orderable, then so is every restriction of~it.
\end{LE}

\begin{theorem}\label{SBO-weighted-packing}
Suppose $M_1(D,w)$ is strongly base orderable. Then there exists a $w$-weighted packing of dijoins of size $\tau$.
\end{theorem}
\begin{proof}
We proceed by induction on $\tau\geq 2$; the base case and the induction step are both resolved in Claim~3.
By \Cref{SBO-admissible-partition}, $a(V)$ can be partitioned into $\tau$ admissible sets $Q_1,\ldots,Q_\tau$. By \Cref{admissible}, there exists a rounded $1$-factor $J_1\subseteq A_1$ such that $\dc(J_1)=Q_1$, and $J_1$ is a dijoin of $D$. Let $b:=(\tau-1)\cdot \chi_V + \sum_{i=2}^{\tau} \chi_{Q_i}$.

\begin{claim} 
$A_1-J_1$ is a perfect $b$-matching, and a $(\tau-1)$-dijoin of $D$.
\end{claim}
\begin{cproof}
Let $b_1:=\chi_V+\chi_{Q_1}$. Since $J_1$ is a perfect $b_1$-matching and $A_1$ a perfect $(b_1+b)$-matching, $A_1-J_1$ is a perfect $b$-matching.
By \Cref{k-dijoin}, $A_1-J_1$ is a $(\tau-1)$-dijoin of $D$.
\end{cproof}

\begin{claim} 
$(D,\chi_{A_1-J_1})$ is a sink-regular weighted $(\tau-1,\tau)$-bipartite digraph, and $M_1(D,\chi_{A_1-J_1})$ is a strongly base orderable matroid.
\end{claim}
\begin{cproof}
The first part follows from Claim~1. Since the active vertices of $(D,\chi_{A_1-J_1})$ are $Q_2\cup \cdots\cup Q_\tau$, we get that $M_1(D,\chi_{A_1-J_1})=M_1(D,w)\setminus Q_1$.
Thus, the second part of the claim follows from \Cref{SBO-minor}.
\end{cproof}

\begin{claim} 
$A_1-J_1$ can be partitioned into $\tau-1$ rounded $1$-factors $J_2,\ldots,J_\tau$ each of which is a dijoin of $D$.
\end{claim}
\begin{cproof}
If $\tau=2$, then $A_1-J_1$ is a rounded $1$-factor that is a dijoin by Claim~1. This proves the base case of the induction. Otherwise, $\tau\geq 3$. By Claim~2, we may apply the induction hypothesis to $(D,\chi_{A_1-J_1})$, and obtain that its set of weight $1$ arcs, namely $A_1-J_1$, can be partitioned into $\tau-1$ rounded $1$-factors each of which is a dijoin of $D$.
\end{cproof}

Claims~2 and~3 finish the proof.
\end{proof}

We shall see in \S\ref{sec:D27} that $M_1(D,w)$, even for $w=\1$, is not necessarily strongly base orderable. We shall strengthen this result in \S\ref{subsec:SBO-2}.

\subsection{$[[\wt,\tau,2]]$ is true.}

\begin{theorem}[Brualdi~\cite{Brualdi69}]\label{symmetric-exchange}
Let $M$ be a matroid over ground set $V$. Then for every two distinct bases $B_1,B_2$, and for each $u\in B_1-B_2$, there exists a $v\in B_2-B_1$ such that $B_1\tr \{u,v\},B_2\tr \{u,v\}$ are bases.
\end{theorem}

\begin{CO}\label{SBO-rank-2}
Every matroid of rank at most two is strongly base orderable.\qed
\end{CO}

\begin{theorem}\label{rho=2-sink-reg}
Let $(D=(V,A),w)$ be a sink-regular weighted $(\tau,\tau+1)$-bipartite digraph such that $\rho(\tau,D,w)\leq 2$. Then there exists a $w$-weighted packing of dijoins of size $\tau$.
\end{theorem}
\begin{proof}
Observe that $M_1(D,w)$ has rank $\disc(V)=\rho(\tau,D,w)\leq 2$, where the equality follows from \Cref{disc}~(1). Thus, by \Cref{SBO-rank-2}, $M_1(D,w)$ is strongly base orderable, so \Cref{SBO-weighted-packing} proves the existence of a $w$-weighted packing of dijoins of size $\tau$.
\end{proof}

In the appendix (\S\ref{sec:rho=2-elementary}), we give an elementary though less insightful proof of \Cref{rho=2-sink-reg}, based solely on the content of \S\ref{sec:rho=1}.

\begin{theorem}\label{rho=2}
Let $(D=(V,A),w)$ be a weighted digraph where every dicut has weight at least $\tau$, and $\tau\geq 2$. Suppose $\rho(\tau,D,w)\leq 2$. Then there exists a $w$-weighted packing of dijoins of size $\tau$. That is, $[[\wt,\tau,2]]$ is true.
\end{theorem}
\begin{proof}
By \Cref{rho=2-sink-reg}, $[[\wt,\tau,2]]$ holds for all sink-regular weighted $(\tau,\tau+1)$-bipartite digraphs. In particular, by \Cref{weighted-reduction}~(2), $[[\wt,\tau,2]]$ is true.
\end{proof}

\section{$[[3,3]]$ is true.}\label{sec:rho=3}

In this section we prove \hyperlink{primary-results}{{\bf P4}} and \hyperlink{secondary-results}{{\bf S2}}. Throughout the section, unless stated otherwise, we are given an integer $\tau\geq 2$, a weighted $(\tau,\tau+1)$-bipartite digraph $(D=(V,A),w)$ that is sink-regular, and $w\in \{0,1\}^A$. Let $A_1:=\{a\in A:w_a=1\}$.

\subsection{$k$-Admissible sets and their existence}

\begin{DE}
Let $k\in [\tau]$ and $Q\subseteq a(V)$. We say that $Q$ is \emph{a $k$-admissible set of $(D,w)$} if it is the union of $k$ disjoint bases of $M_0(D,w)$, and also the union of $k$ disjoint bases of $M_1(D,w)$.
\end{DE}

The notion of $k$-admissibility is crucial for the rest of this section. Let us prove the existence of $k$-admissible sets. This fact, though not needed, will help the reader contextualize the results of this section. Let $\mathcal{U}_0$ be the set of $U\subseteq V$ such that $\delta^+(U)$ is a dicut of $D[A_1]$, and $\mathcal{U}_1$ the set of $U\subseteq V$ such that $\delta^+(U)$ is a dicut of $D$.

\begin{LE}\label{k-admissible-description}
Let $k\in [\tau]$ and $Q\subseteq a(V)$. Then $Q$ is a $k$-admissible set of $(D,w)$ if, and only if, \begin{equation*}
\begin{array}{rll}
|Q|&=k\cdot\disc(V)&\\
|Q\cap U|&\geq k\cdot\disc(U)&\forall U\in \mathcal{U}_0\\
|Q\cap U|&\geq k(1+\disc(U))&\forall U\in \mathcal{U}_1.
\end{array}
\end{equation*}
\end{LE}
\begin{proof}
$(\Rightarrow)$ follows immediately from the definition combined with \Cref{M0-matroid}.
$(\Leftarrow)$ Let $x=\chi_Q\in \{0,1\}^{a(V)}$. Then $x\in (kP_0)\cap (kP_1)$, where $P_i$ is the base polytope of $M_i(D,w)$ for $i\in \{0,1\}$. Since $P_i$ has the integer decomposition property by \Cref{IDP}, it follows that $x$ can be written as the sum of $k$ integer points in $P_i$, for each $i\in \{0,1\}$. That is, $Q$ is the union of $k$ disjoint bases of $M_i(D,w)$, for each $i\in \{0,1\}$, so $Q$ is $k$-admissible.
\end{proof}

\begin{theorem}
Let $k\in [\tau]$. Then $(D,w)$ has a $k$-admissible set.
\end{theorem}
\begin{proof}
 By \Cref{dicut-crossing}, $\mathcal{U}_0,\mathcal{U}_1$ are crossing families over ground set $V$. Consider the system $x(V)=k\cdot\disc(V)$; $x(U)\geq k\cdot\disc(U)~\forall U\in \mathcal{U}_0$; $x(U)\geq k(1+\disc(U))~\forall U\in \mathcal{U}_1$. By \Cref{box-TDI}, this system is box-TDI. In particular, the polytope $P$ defined by \begin{equation*}
\begin{array}{rll}
x(V)&=k\cdot\disc(V)&\\
x(U)&\geq k\cdot\disc(U)&\forall U\in \mathcal{U}_0\\
x(U)&\geq k(1+\disc(U))&\forall U\in \mathcal{U}_1\\
x_u&\in [0,1]&\forall u\in a(V)\\
x_u&=0&\forall u\in V-a(V)
\end{array}
\end{equation*} if nonempty, is integral.
Thus, by \Cref{k-admissible-description}, the vertices of $P$ are precisely to the characteristic vectors of the $k$-admissible sets of $(D,w)$. Thus, to finish the proof, it suffices to show that $P$ is a nonempty polytope. To this end, let $x:=\chi_{a(V)}\in \{0,1\}^V$. Then $x(V) = |a(V)|=\tau\cdot \disc(V)$ by \Cref{disc}~(1). Moreover, by \Cref{disc}~(2) and \Cref{R1F-remark},
$$x(U)-\tau\cdot \disc(U) = |a(U)| - \tau\cdot \disc(U) = w(\delta^+(U))$$
for every dicut $\delta^+(U)$ of $D[A_1]$. In particular, $x(U)\geq \tau\cdot \disc(U)$ for every dicut $\delta^+(U)$ of $D[A_1]$, and
$x(U)\geq \tau\cdot (1+\disc(U))$ for every dicut $\delta^+(U)$ of $D$, as every dicut of $D$ has weight at least $\tau$. Thus, $\frac{k}{\tau}x\in P$, so $P$ is nonempty, as required.
\end{proof}

\subsection{Packing a dijoin and a $(\tau-1)$-dijoin in weighted digraphs}\label{sec:two-dijoins-weighted}

\begin{theorem}\label{weighted-two-dijoins}
Suppose the set of active vertices of $(D,w)$ is partitioned into an admissible set $Q$ and a $(\tau-1)$-admissible set $Q'$. Then there exist a sink-regular weighted $(1,2)$-bipartite digraph $(D,c)$ such that $M_1(D,c)=M_1(D,w)|Q$, and a sink-regular weighted $(\tau-1,\tau)$-bipartite digraph $(D,c')$ such that $M_1(D,c')=M_1(D,w)|Q'$, and $c+c'=w$. In particular, $A_1$ can be partitioned into a dijoin and a $(\tau-1)$-dijoin of $D$.
\end{theorem}
\begin{proof}
Let $b:=\chi_V+\chi_{Q}\in \mathbb{Z}^V_{\geq 0}$, $S$ the set of sources, and $T$ the set of sinks of $D$.

\begin{claim} 
There exists a perfect $b$-matching $J\subseteq A_1$. Moreover, $J$ is a dijoin of $D$.
\end{claim}
\begin{cproof}
As $Q$ is a basis of $M_0(D,w)$, it is bimatchable in $(D,w)$, so there exists a rounded $1$-factor $J\subseteq A_1$ such that $\dc(J)=Q$. Note that $J$ is a perfect $b$-matching. Moreover, since $Q$ is a basis of $M_1(D,w)$, it follows from \Cref{R1F-M1} (or \Cref{k-dijoin}) that $J$ is a dijoin of~$D$.~\end{cproof}

Let $\bar{b}:=(\tau-1)\cdot\chi_V+\chi_{Q'}$.

\begin{claim} 
$A_1-J$ is a perfect $\bar{b}$-matching, and also a $(\tau-1)$-dijoin of $D$.
\end{claim}
\begin{cproof}
 Observe that $A_1$ is a $(b+\bar{b})$-matching, so $A_1-J$ is a perfect $\bar{b}$-matching. Since $Q'$ is the union of $\tau-1$ disjoint bases of $M_1(D,w)$, it follows from \Cref{k-dijoin} that $A_1-J$ is a $(\tau-1)$-dijoin.
~\end{cproof}

\begin{claim} 
$(D,\chi_J)$ is a sink-regular weighted $(1,2)$-bipartite digraph and $(D,\chi_{A_1-J})$ is a sink-regular weighted $(\tau-1,\tau)$-bipartite digraph. Moreover, $M_1(D,\chi_J)=M_1(D,w)|Q$ and $M_1(D,\chi_{A_1-J})=M_1(D,w)|Q'$.
\end{claim}
\begin{cproof}
The first part of the claim follows from Claims~1 and~2.
The second part follows from the facts that $Q,Q'$ are the sets of active vertices of $(D,\chi_J),(D,\chi_{A_1-J})$, respectively.
\end{cproof}

Claims~1-3 finish the proof.
\end{proof}

Observe that the assumption of \Cref{weighted-two-dijoins} holds if $w=\1$ by \Cref{M1-basis-partition} and \Cref{unweighted-admissible}. Thus, \Cref{weighted-two-dijoins} extends \Cref{two-dijoins-sink-reg} to the weighted setting. Observe that the assumption of \Cref{weighted-two-dijoins} cannot always hold, because $[[\wt,\tau]]$ is not true for $\tau=2$.

\subsection{Weighted packings under strong base orderability, II}\label{subsec:SBO-2}

\begin{theorem}\label{2-SBO-weighted-packing}
Suppose the set of active vertices of $(D,w)$ is partitioned into an admissible set $Q$ and a $(\tau-1)$-admissible set $Q'$. Suppose further that $M_1(D,w)|Q'$ is strongly base orderable. Then there exists a $w$-weighted packing of dijoins of size $\tau$.
\end{theorem}
\begin{proof}
By \Cref{weighted-two-dijoins}, there exist a sink-regular weighted $(1,2)$-bipartite digraph $(D,c)$ such that $M_1(D,c)=M_1(D,w)|Q$, and a sink-regular weighted $(\tau-1,\tau)$-bipartite digraph $(D,c')$ such that $M_1(D,c')=M_1(D,w)|Q'$, and $c+c'=w$. Pick $J\subseteq A_1$ such that $\chi_J=c$; note that $\chi_{A_1-J}=c'$. Observe that $J$ is a dijoin, and $A_1-J$ is a $(\tau-1)$-dijoin.

We know that $M_1(D,c')=M_1(D,w)|Q'$ is strongly base orderable. Thus, by \Cref{SBO-weighted-packing}, $(D,c')$ has a $c'$-weighted packing of dijoins of size $\tau-1$. This weighted packing, together with the $c$-weighted packing $J$, yields a $w$-weighted packing of dijoins of size $\tau$ in $(D,w)$, as desired.
\end{proof}

We shall see in \Cref{sec:D27} that given a partition into $Q$ and $Q'$, the second assumption of \Cref{2-SBO-weighted-packing}, that $M_1(D,w)|Q'$ is strongly base orderable, does not necessarily hold, even if $w=\1$.

\subsection{$[[3,3]]$ is true.}

Denote by $K_4$ the complete graph on $4$ vertices. Recall that $M(K_4)$ is the cycle matroid of $K_4$.

\begin{LE}[Brualdi \cite{Brualdi71}]\label{K4}
Up to isomorphism, $M(K_4)$ is the only matroid on at most six elements that is not strongly base orderable.
\end{LE}

\begin{LE}[proved in \S\ref{sec:K4-lemma}]\label{9-elements}
Let $M$ be a matroid over $9$ elements whose ground set can be partitioned into bases $Q_1,Q_2,Q_3$. Then we may choose $Q_1,Q_2,Q_3$ such that $M|(Q_i\cup Q_j)\not\cong M(K_4)$ for some distinct $i,j\in [3]$.
\end{LE}

\begin{theorem}\label{rho=3-basis-partition}
Let $\tau\geq 3$ be an integer, and $D=(V,A)$ a sink-regular weighted $(\tau,\tau+1)$-bipartite digraph such that $\rho(\tau,D,w)=3$. There exist disjoint bases $Q_1,\ldots,Q_\tau$ of $M_1(D,w)$ such that $M_1(D,w)|(Q_1\cup Q_2)$ is strongly base orderable.
\end{theorem}
\begin{proof}
By \Cref{M1-basis-partition}, there exist disjoint bases $Q_1,\ldots,Q_\tau$ of $M_1(D,w)$. For each $i\in [\tau]$, $|Q_i|=\disc(V)=\rho(\tau,D,w)$ where the last equality follows from \Cref{disc}~(1), so $|Q_i|=3$.
Consider the matroid $M:=M_1(D,w)|(Q_1\cup Q_2\cup Q_3)$, which has $9$ elements and its ground set is partitioned into bases $Q_1,Q_2,Q_3$. By \Cref{9-elements}, we may choose $Q_1,Q_2,Q_3$ such that $M|(Q_1\cup Q_2)\not\cong M(K_4)$, so by \Cref{K4}, $M|(Q_1\cup Q_2)$ is strongly base orderable. Since $M_1(D,w)|(Q_1\cup Q_2)=M|(Q_1\cup Q_2)$, the disjoint bases $Q_1,Q_2,Q_3,Q_4,\ldots,Q_\tau$ prove the theorem.
\end{proof}

\begin{theorem}\label{tau=3-rho=3-sink-reg}
Let $D=(V,A)$ a sink-regular $(3,4)$-bipartite digraph such that $\rho(3,D)\leq 3$. Then $A$ can be partitioned into three disjoint dijoins.
\end{theorem}
\begin{proof}
If $\rho(3,D)\leq 2$, then the result follows from \Cref{rho=2-sink-reg}. Otherwise, $\rho(3,D)=3$. By \Cref{rho=3-basis-partition}, there exist disjoint bases $Q_1,Q_2,Q_3$ of $M_1(D,\1)$ such that $M_1(D,\1)|(Q_1\cup Q_2)$ is strongly base orderable. Since $w=\1$, it follows from \Cref{unweighted-admissible} that the sets $Q_1,Q_2,Q_3$ are admissible in $(D,\1)$. Thus, we have a partition of the active vertices into an admissible set $Q_3$ and a $2$-admissible set $Q_1\cup Q_2$ such that $M_1(D,\1)|(Q_1\cup Q_2)$ is strongly base orderable. Thus, by \Cref{2-SBO-weighted-packing}, $D$ contains $3$ disjoint dijoins.
\end{proof}

\begin{theorem}\label{tau=3-rho=3}
Let $D=(V,A)$ be a digraph where every dicut has size at least $3$. Suppose $\rho(3,D)\leq 3$. Then there exist $3$ disjoint dijoins. That is, $[[3,3]]$ is true.
\end{theorem}
\begin{proof}
This follows from \Cref{tau=3-rho=3-sink-reg} and \Cref{reduction}~(2).
\end{proof}

\subsection{$M(K_4)$-restrictions in matroids of rank three}\label{sec:K4-lemma}

It remains to prove \Cref{9-elements}, which requires a fair bit of Matroid Theory~\cite{Oxley11}. Let $M$ be a matroid of rank $r$ over a finite ground set $E$. A \emph{flat} is a subset $F\subseteq E$ that is closed in $M$, i.e.\ $|F\cap C|\neq |C|-1$ for any circuit $C$ of $M$. A \emph{hyperplane} is a subset $X\subseteq E$ satisfying any of the following equivalent conditions: (i) $X$ is a flat of rank $r-1$, (ii) $X$ is a maximal non-spanning set, (iii) $E-X$ is a cocircuit.

\begin{LE}[\cite{Oxley11}, Proposition 1.5.6]\label{geometric-rep}
Let $E$ be a finite set, and $\Lambda$ a family of subsets of $E$ each of size at least $3$ such that every two distinct members of $\Lambda$ meet in at most $1$ element. Let $\mathcal{I}$ be the set of subsets $X$ of $E$ of size at most $3$ such that no member of $\Lambda$ contains $3$ elements of $X$. Then $\mathcal{I}$ is the family of independent sets of a simple matroid of rank at most $3$ whose rank-$1$ flats are the $1$-element subsets of $E$, and whose rank-$2$ flats are the members of $\Lambda$ together with all $2$-element subsets $Y$ of $E$ for which no member of $\Lambda$ contains $Y$. Moreover, every simple matroid of rank at most $3$ arises in this way.
\end{LE}

We will be working with simple matroids of rank three. In this case, by \Cref{geometric-rep}, we may represent $M$ geometrically via a set of \emph{points} and \emph{lines}, as follows: the points correspond to the elements of $E$, and the set of lines $\Lambda$ correspond to a \emph{subset} of the set of the hyperplanes. More precisely, the lines in $\Lambda$ correspond to the hyperplanes with at least $3$ elements; the other hyperplanes are precisely the $2$-element subsets $Y$ for which no line in $\Lambda$ contains $Y$. In particular, every two distinct points belong to exactly one hyperplane, and therefore at most one line. Observe that three collinear points correspond to a circuit of size three. Note that the lines in $\Lambda$ displayed may not be straight, and may even be circles.

For example, consider $M(K_4)$. Observe that the hyperplanes are the three perfect matchings $\{i,\pi(i)\},\{j,\pi(j)\},\{k,\pi(k)\}$ and the four triangles including, say, $\{i,j,k\}$. Throughout this subsection, we assume that $M(K_4)$ follows this labeling. We can then represent the hyperplanes corresponding to the triangles as the four lines in Figure~\ref{fig:hype-rep}~(a).

\begin{figure}[h]
\centering
\includegraphics[scale=.3]{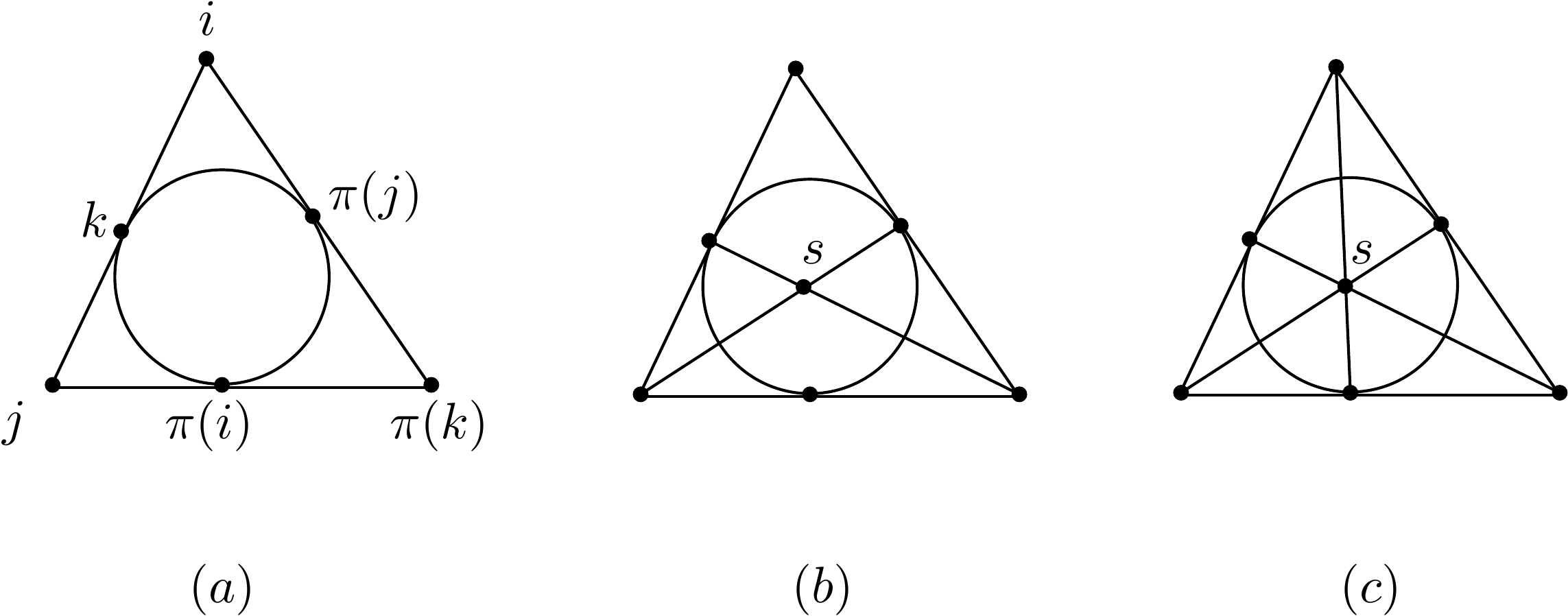}
\caption{(a) The geometric representation $M(K_4)$. (b)-(c) The geometric representations of the two matroids of \Cref{7-elements}, where (b) is the non-Fano matroid $F_7^-$, and $(c)$ is the Fano matroid $F_7$. The unlabeled points follow the labeling of (a).}
\label{fig:hype-rep}
\end{figure}

\begin{RE}\label{hyperplane-deletion}
Consider a matroid $M$ and a deletion minor $M\setminus e$ of the same rank. 
Then for every hyperplane $X$ of $M$ such that $e\notin X$, $X$ is a hyperplane of $M\setminus e$. Moreover, for every hyperplane $X'$ of $M\setminus e$, either $X'$ or $X'\cup \{e\}$ is a hyperplane of $M$ (but not both).\qed
\end{RE}

Denote by $F_7$ the \emph{Fano matroid}, and by $F_7^-$ the \emph{non-Fano matroid}, represented in Figure~\ref{fig:hype-rep}.

\begin{LE}\label{7-elements}
Let $M_7$ be a simple matroid of rank $3$ over ground set $\{i,\pi(i),j,\pi(j),k,\pi(k),s\}$ such that $M_7\setminus s=M(K_4)$ and $M_7\setminus i\cong M(K_4)$. Then $M_7\cong F_7^-$ or $F_7$ with the geometric representation provided in Figure~\ref{fig:hype-rep}~(b) or~(c), respectively.
\end{LE}
\begin{proof}
Consider the geometric representation of $M_7\setminus s$ in Figure~\ref{fig:M7-s}. \begin{figure}[ht]
\centering
\includegraphics[scale=.3]{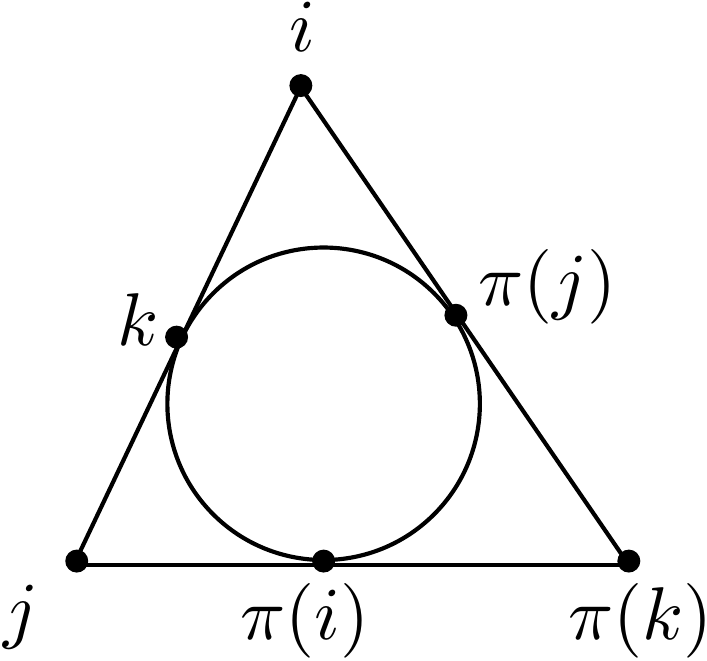}
\caption{Geometric representation of $M_7\setminus s$ (needed in the proof of \Cref{7-elements}).}
\label{fig:M7-s}
\end{figure}
By \Cref{hyperplane-deletion}, for every hyperplane $X'$ of $M_7\setminus s$, either $X'$ or $X'\cup \{s\}$ is a hyperplane of $M_7$. In particular, $M_7$ contains four lines $\ell_0,\ell_i,\ell_j,\ell_k$ where $\ell_0=\{i,j,k\}$ or $\{i,j,k,s\}$, $\ell_i= \{i,\pi(j),\pi(k)\}$ or $\{i,\pi(j),\pi(k),s\}$, $\ell_j= \{j,\pi(i),\pi(k)\}$ or $\{j,\pi(i),\pi(k),s\}$, and $\ell_k=\{k,\pi(i),\pi(j)\}$ or $\{k,\pi(i),\pi(j),s\}$.

Consider now the matroid $M_7\setminus i$, which is isomorphic to $M(K_4)$. In particular, in the geometric representation of $M_7\setminus i$, there are precisely $4$ lines. Observe that by \Cref{hyperplane-deletion}, the two lines $\ell_j,\ell_k$ of $M_7$ excluding $i$ are among the four lines of $M_7\setminus i$. In particular, $\ell_j,\ell_k$ have size $3$, so $\ell_j= \{j,\pi(i),\pi(k)\}$ and $\ell_k=\{k,\pi(i),\pi(j)\}$. The other two lines of $M_7\setminus i$ lead by \Cref{hyperplane-deletion} to two lines $\ell_s,\ell'_s$ of~$M_7$. The other two lines of $M_7\setminus i$ are either
$\{j,k,s\},\{\pi(j),\pi(k),s\}$ or $\{j,\pi(j),s\},\{k,\pi(k),s\}$, so there are two cases.
\begin{enumerate}
\item[Case 1:] $\ell_s-\{i\}=\{j,k,s\},\ell'_s-\{i\}=\{\pi(j),\pi(k),s\}$ are lines of $M_7\setminus i$. In this case, $|\ell_0\cap \ell_s|\geq 2$, so $\ell_0=\ell_s=\{i,j,k,s\}$. Similarly, $|\ell_i\cap \ell'_s|\geq 2$, so $\ell_i=\ell'_s=\{i,\pi(j),\pi(k),s\}$. But then $|\ell_0\cap \ell_i|\geq 2$, a contradiction as $\ell_0\neq \ell_i$.
\item[Case 2:] $\ell_s-\{i\}=\{j,\pi(j),s\},\ell'_s-\{i\}=\{k,\pi(k),s\}$ are lines of $M_7\setminus i$. In this case, $\ell_0$ and $\ell_s$ are distinct lines of $M_7$, so $|\ell_0\cap \ell_s|\leq 1$, so $s\notin \ell_0$ and $i\notin \ell_s$, so $\ell_0=\{i,j,k\}$ and $\ell_s=\{j,\pi(j),s\}$. Similarly, $\ell_i$ and $\ell'_s$ are distinct lines of $M_7$, so $|\ell_i\cap \ell'_s|\leq 1$, so $s\notin \ell_i$ and $i\notin \ell'_s$, so $\ell_i=\{i,\pi(j),\pi(k)\}$ and $\ell'_s=\{k,\pi(k),s\}$.
\end{enumerate}
Picking up where Case 2 left off, we have the partial geometric representation of $M_7$ in Figure~\ref{fig:M7-partial-gr}. \begin{figure}[h]
\centering
\includegraphics[scale=.3]{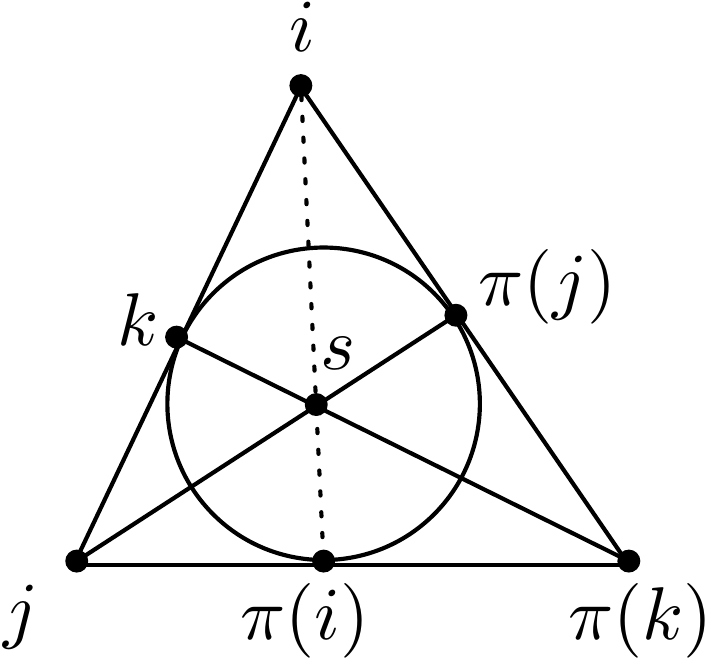}
\caption{Partial geometric representation of $M_7$ where the lines $\ell_0,\ell_i,\ell_j,\ell_k,\ell_s,\ell'_s$ have been drawn with solid curves (needed in the proof of \Cref{7-elements}).}
\label{fig:M7-partial-gr}
\end{figure} Since every two elements of $M_7$ belong to exactly one hyperplane, it can be readily checked that there is at most one additional line, namely $\{i,\pi(i),s\}$, thereby finishing the proof.
\end{proof}

Denote by $U_{2,4}$ the uniform matroid over ground set $[4]$ of rank $2$, i.e. it is the matroid whose circuits are $\{1,2,3\},\{2,3,4\},\{1,3,4\},\{1,2,4\}$, so $U_{2,4}$ shows up in any line with at least four points.

\begin{RE}\label{4-point-line}
Let $M$ be a simple matroid of rank $3$ whose geometric representation has a line containing distinct points $a,b,c,d$. Then $M|\{a,b,c,d\}\cong U_{2,4}$.
\end{RE}

\begin{figure}[h]
\centering
\includegraphics[scale=.3]{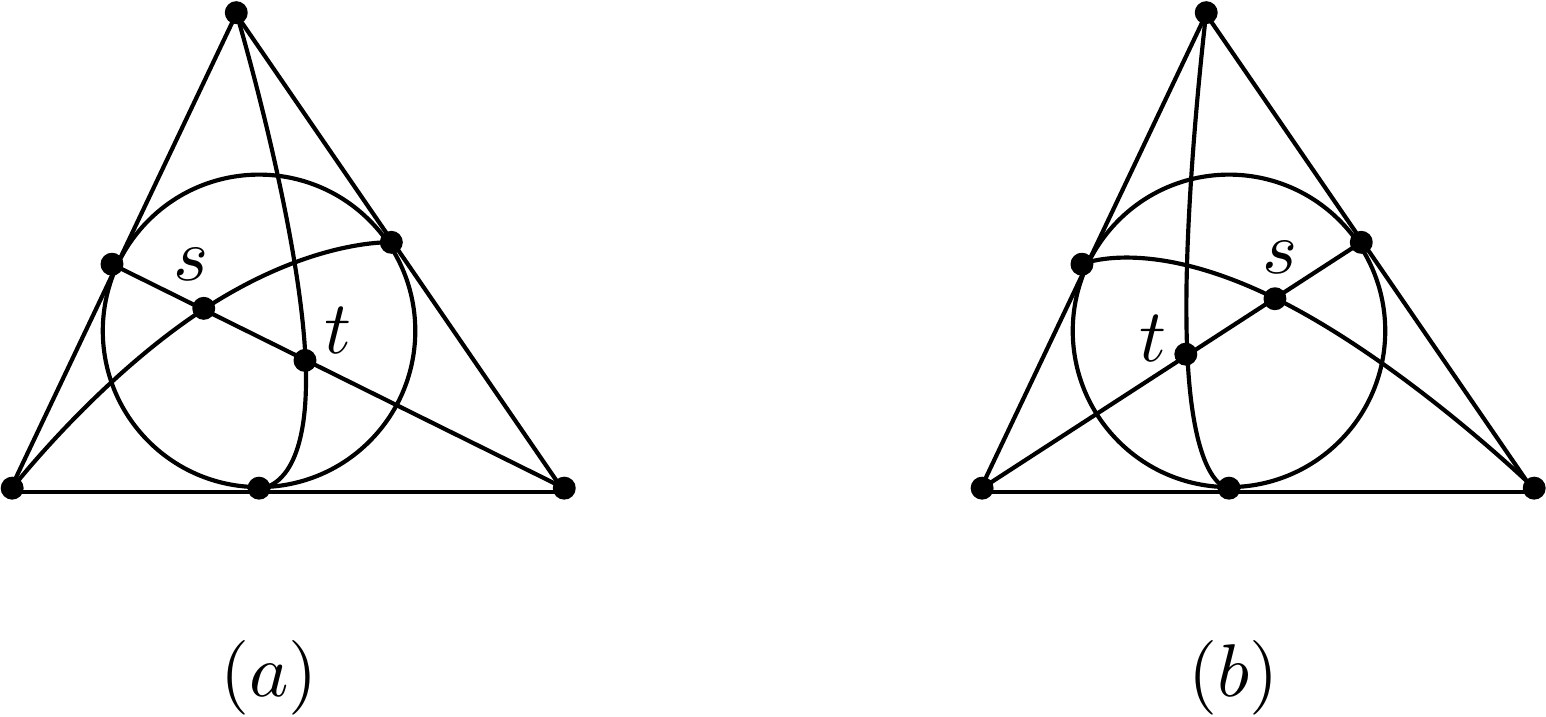}
\caption{The geometric representations of the two matroids of \Cref{8-elements}. The unlabeled points follow the labeling of Figure~\ref{fig:hype-rep}~(a).}
\label{fig:hype-rep-4}
\end{figure}

\begin{LE}\label{8-elements}
Let $M_8$ be a simple matroid of rank $3$ over ground set $\{i,\pi(i),j,\pi(j),k,\pi(k),s,t\}$ such that $M_8\setminus \{s,t\}=M(K_4)$, $M_8\setminus \{i,t\}\cong M(K_4)$, and $M_8\setminus \{i',s\}\cong M(K_4)$, for some $i'\in \{i,\pi(i),j,\pi(j),$ $k,\pi(k)\}$. Then either $i'\in \{j,\pi(j)\}$ and $M_8$ is the matroid with geometric representation in Figure~\ref{fig:hype-rep-4}~(a), or $i'\in \{k,\pi(k)\}$ and $M_8$ is the matroid with geometric representation in Figure~\ref{fig:hype-rep-4}~(b). In particular, $M_8|\{k,\pi(k),s,t\}\cong U_{2,4}$ or $M_8|\{j,\pi(j),s,t\}\cong U_{2,4}$.
\end{LE}
\begin{proof}
By \Cref{7-elements}, $M_8\setminus t,M_8\setminus s$ are the matroids with the partial geometric representations displayed in Figure~\ref{fig:M8-t-s-1}.
\begin{figure}[h]
\centering
\includegraphics[scale=.3]{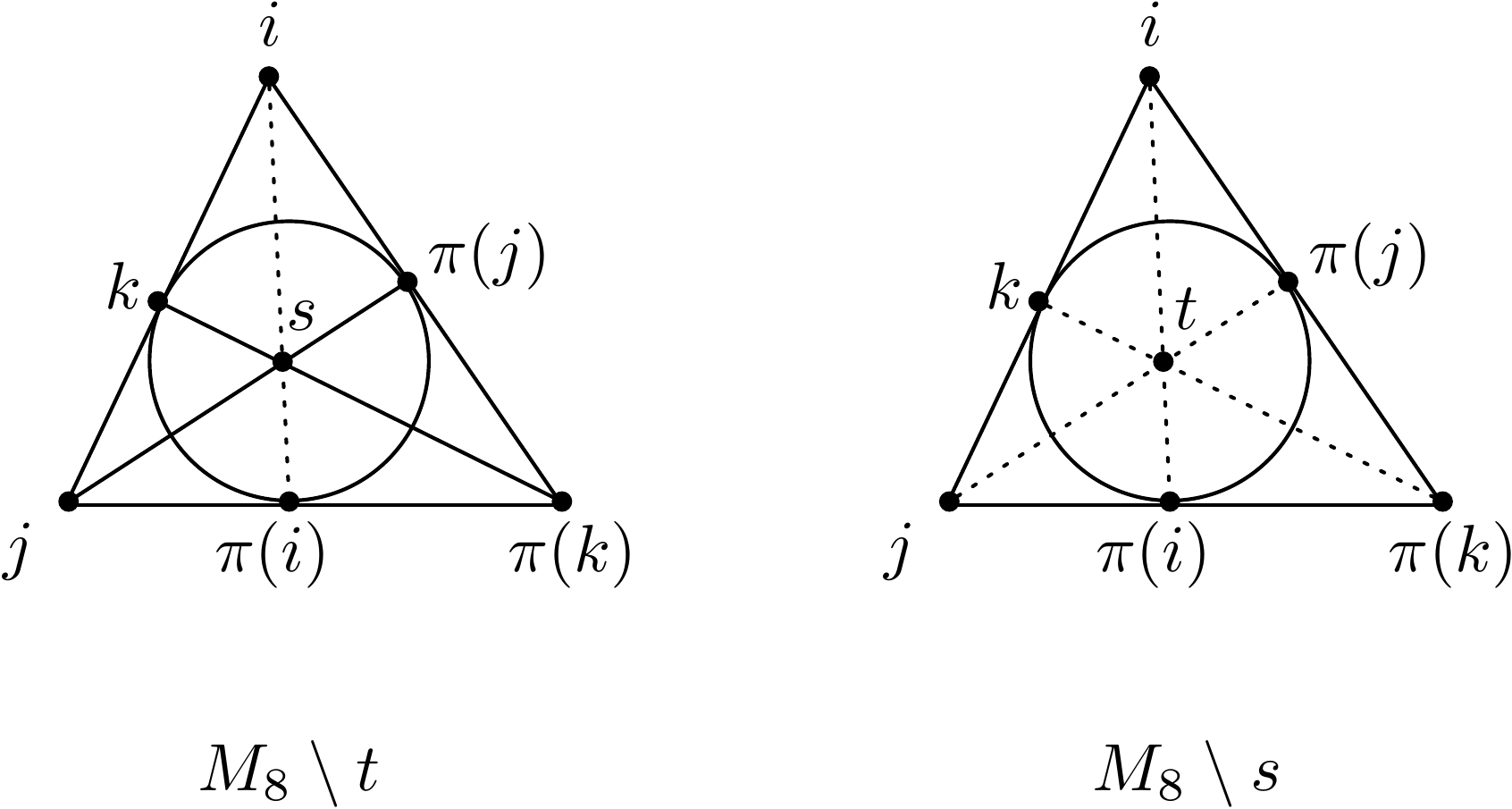}
\caption{Partial geometric representations of $M_8\setminus t$ and $M_8\setminus s$ (from the proof of \Cref{8-elements}). For $M_8\setminus t$ the dashed line may or may not be present, and for $M_8\setminus s$ at least two of the three dashed lines must be present.}
\label{fig:M8-t-s-1}
\end{figure}
For each hyperplane $X'$ of $M_8\setminus t$, either $X'$ or $X'\cup \{t\}$ is a hyperplane of $M_8$ by \Cref{hyperplane-deletion}. Thus, $M_8$ contains the four lines $\ell_0,\ell_i,\ell_j,\ell_k$ such that $\ell_0-\{t\}=\{i,j,k\},\ell_i-\{t\}=\{i,\pi(j),\pi(k)\},\ell_j-\{t\}=\{j,\pi(i),\pi(k)\},\ell_k-\{t\}=\{k,\pi(i),\pi(j)\}$. These lines exclude $s$, so they must be lines of $M_8\setminus s$ by \Cref{hyperplane-deletion}, so $\ell_0=\{i,j,k\},\ell_i=\{i,\pi(j),\pi(k)\},\ell_j=\{j,\pi(i),\pi(k)\},\ell_k=\{k,\pi(i),\pi(j)\}$.

Consider the matroid $M_8\setminus t$. For each $r\in \{i,j,k\}$, if the line through $s,r$ is present, then let $\ell_{sr}$ denote the line of $M_8$ such that $\ell_{sr}-\{t\}=\{r,\pi(r),s\}$. Observe that $\ell_{sj},\ell_{sk}$ exist, and $\ell_{si}$ may or may not exist.

Similarly, consider the matroid $M_8\setminus s$. For each $r\in \{i,j,k\}$, if the line through $t,r$ is present, then let $\ell_{tr}$ denote the line of $M_8$ such that $\ell_{tr}-\{s\}=\{r,\pi(r),t\}$. Observe that at least two of $\ell_{ti},\ell_{tj},\ell_{tk}$ exist.

\begin{claim} 
Suppose both $\ell_{sr},\ell_{tr}$ exist for some $r\in \{i,j,k\}$. Then $\ell_{sr}=\ell_{tr}=\{r,\pi(r),s,t\}$.
\end{claim}
\begin{cproof}
Observe that $\{r,\pi(r)\}\subseteq \ell_{sr}\cap \ell_{tr}$, so $|\ell_{tj}\cap \ell_{sj}|\geq 2$, implying in turn that $\ell_{sr}=\ell_{tr}=\{r,\pi(r),s,t\}$.
\end{cproof}

\begin{claim} 
$\{r\in \{i,j,k\}:\ell_{sr} \text{ exists}\}$ and $\{r\in \{i,j,k\}:\ell_{tr} \text{ exists}\}$ have at most one index in common.
\end{claim}
\begin{cproof}
For if not, then by Claim~1, $\{r,\pi(r),s,t\},\{r',\pi(r'),s,t\}$ are lines of $M_8$ for distinct $r,r'\in \{i,j,k\}$, a contradiction as the two lines are distinct and meet in $2$ points.
\end{cproof}

\begin{claim} 
$\ell_{si}$ does not exist, $\ell_{ti}$ exists, and exactly one of $\ell_{tj},\ell_{tk}$ exists.
\end{claim}
\begin{cproof}
We know that at least two of $\ell_{ti},\ell_{tj},\ell_{tk}$ exist. We also know that $\ell_{sj},\ell_{sk}$ exist. Thus, the claim follows from Claim~2.
\end{cproof}

Thus, in $M_8\setminus t$ there is no line through $s,i$, and in $M_8\setminus s$ there is a line through $t,i$, and either there is a line through $t,k$, or a line through $t,j$, but not both. In the former case we must have $i'\in \{j,\pi(j)\}$, and in the latter $i'\in \{k,\pi(k)\}$.
\begin{enumerate}
\item[Case 1:] $i'\in \{j,\pi(j)\}$. In this case, we have the geometric representations for $M_8\setminus t,M_8\setminus s$ displayed in Figure~\ref{fig:M8-t-s-2}. \begin{figure}[h]
\centering
\includegraphics[scale=.3]{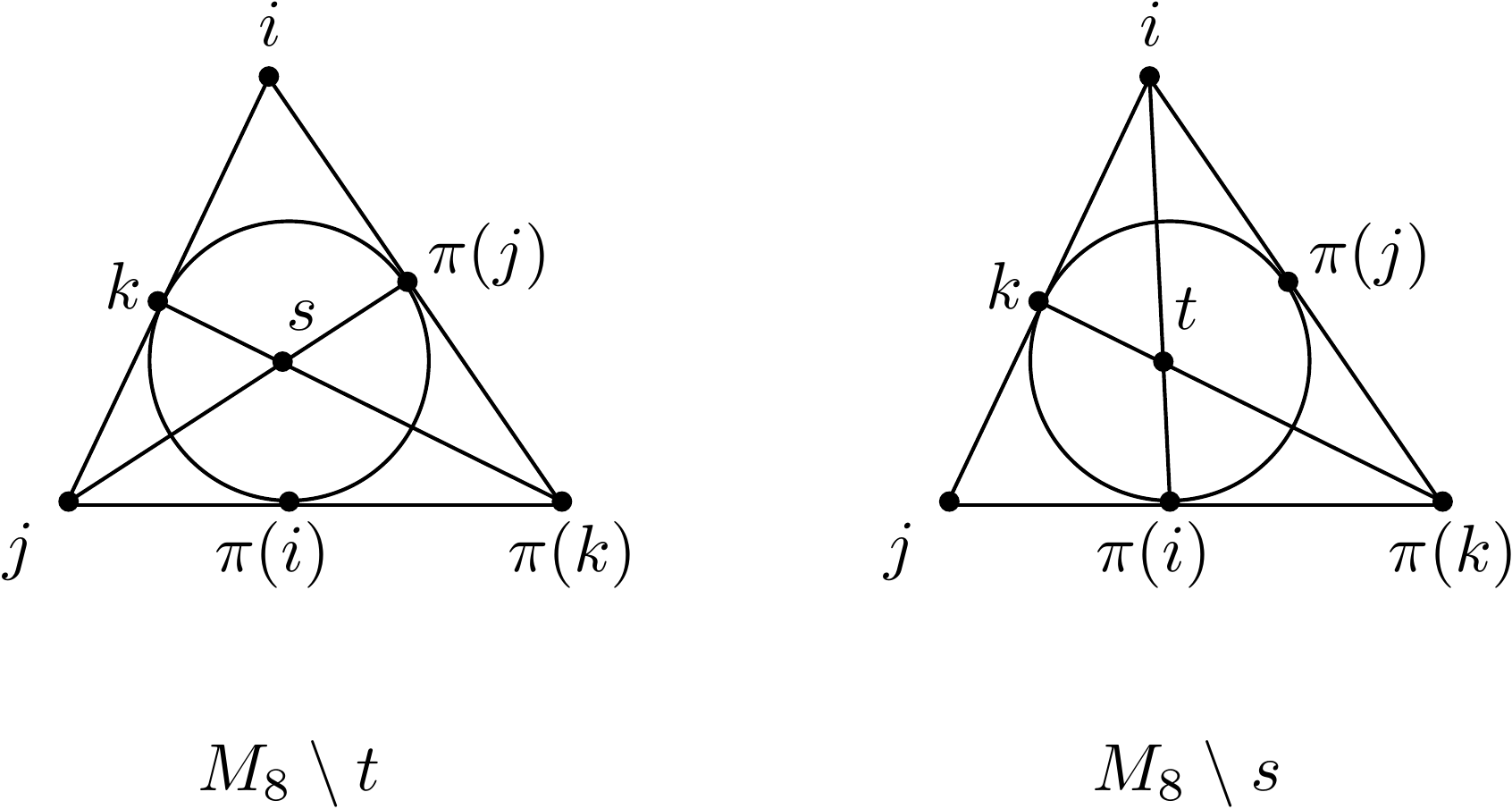}
\caption{Geometric representations of $M_8\setminus t,M_8\setminus s$ (from Case 1 of the proof of \Cref{8-elements}).}
\label{fig:M8-t-s-2}
\end{figure} We know that $M_8$ contains the lines $\ell_0,\ell_i,\ell_j,\ell_k,\ell_{sj},\ell_{sk},\ell_{ti},\ell_{tk}$. We have already obtained the descriptions for $\ell_0,\ell_i,\ell_j,\ell_k$. By Claim~1, we have $\ell_{sk}=\ell_{tk}=\{k,\pi(k),s,t\}$. Since $\ell_{sj},\ell_{sk}$ are distinct lines, and $s\in \ell_{sj}\cap \ell_{sk}$, we get that $t\notin \ell_{sj}$, so $\ell_{sj}=\{j,\pi(j),s\}$. Similarly, since $\ell_{ti},\ell_{tk}$ are distinct lines, and $t\in \ell_{ti}\cap \ell_{tk}$, we get that $s\notin \ell_{ti}$, so $\ell_{ti}=\{i,\pi(i),t\}$. Thus, we have the partial geometric representation for $M_8$ displayed in Figure~\ref{fig:M8-partial-gr}. \begin{figure}[h]
\centering
\includegraphics[scale=.3]{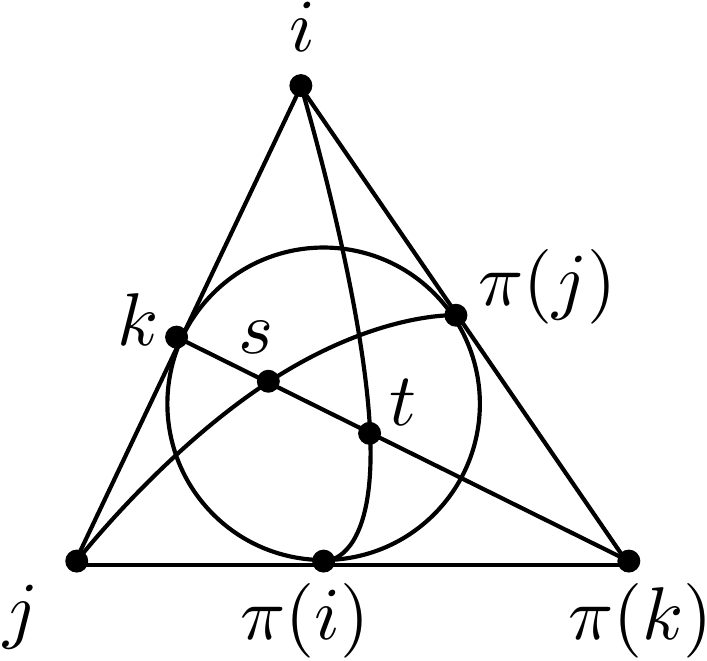}
\caption{Partial geometric representations of $M_8$, where the $7$ lines $\ell_0,\ell_i,\ell_j,\ell_k,\ell_{sj},\ell_{sk},\ell_{ti}$ have been drawn (from Case 1 of the proof of \Cref{8-elements}).}
\label{fig:M8-partial-gr}
\end{figure} In this partial representation, the only pairs of points not contained in a line are $\{s,i\},\{s,\pi(i)\},\{t,j\},\{t,\pi(j)\}$. It can be readily checked now that no additional line can exist, so the partial representation in Figure~\ref{fig:M8-partial-gr} is in fact complete, thereby leading to the outcome in Figure~\ref{fig:hype-rep-4}~(a).
\item[Case 2:] $i'\in \{k,\pi(k)\}$. Similarly, we get that $M_8$ is the matroid represented in Figure~\ref{fig:hype-rep-4}~(b).
\end{enumerate} Observe that by \Cref{4-point-line}, in Case 1, $\ell_{sk}$ is a $4$-point line so $M_8|\{k,\pi(k),s,t\}\cong U_{2,4}$, and in Case 2, $\ell_{sj}$ is a $4$-point line so $M_8|\{j,\pi(j),s,t\}\cong U_{2,4}$.
\end{proof}

We are now ready to prove \Cref{9-elements}, whose statement is repeated here for convenience:\begin{quote}
\emph{Let $M$ be a matroid over $9$ elements whose ground set can be partitioned into bases $Q_1,Q_2,Q_3$. Then we may choose $Q_1,Q_2,Q_3$ such that $M|(Q_i\cup Q_j)\not\cong M(K_4)$ for some distinct $i,j\in [3]$.}
\end{quote}

\begin{proof}[Proof of \Cref{9-elements}]
Observe that $M$ is a loopless matroid of rank $3$. If $M$ is not simple, then it contains parallel elements $e,f$ which inevitably belong to distinct $Q_i,Q_j$, so $M|(Q_i\cup Q_j)\not\cong M(K_4)$. Otherwise, $M$ is simple. Let $s,t$ be distinct elements of $Q_3$. By \Cref{symmetric-exchange}, there exists $i\in Q_2$ such that $Q_2\tr \{i,s\},Q_3\tr\{i,s\}$ are bases, and there exists $i'\in Q_2$ such that $Q_2\tr \{i',t\},Q_3\tr\{i',t\}$ are bases of $M$. Observe that $Q_1,Q_2,Q_3$, and $Q_1,Q_2\tr \{i,s\},Q_3\tr\{i,s\}$, and $Q_1,Q_2\tr \{i',t\},Q_3\tr\{i',t\}$ each form a partition of the ground set into $3$ bases. Thus, we may assume that $M|(Q_1\cup Q_2)=M(K_4)$, $M|(Q_1\cup (Q_2\tr \{i,s\}))\cong M(K_4)$, and $M|(Q_1\cup (Q_2\tr \{i',t\}))\cong M(K_4)$. Let $M_8:=M|(Q_1\cup Q_2\cup \{s,t\})$, a simple matroid over ground set $\{i,\pi(i),j,\pi(j),k,\pi(k),s,t\}$ such that $M_8\setminus \{s,t\}=M(K_4)$, $M_8\setminus \{i,t\}=M|(Q_1\cup (Q_2\tr \{i,s\}))\cong M(K_4)$, and $M_8\setminus \{i',s\}=M|(Q_1\cup (Q_2\tr \{i',t\}))\cong M(K_4)$. It therefore follows from \Cref{8-elements} that $M_8|\{r,\pi(r),s,t\}\cong U_{2,4}$ for some $r\in \{j,k\}$.

Next let $Q'_1,Q'_2$ be a partition of $E(K_4)$ into two spanning trees of $K_4$, where $Q'_1$ contains the perfect matching $\{r,\pi(r)\}$. Then $Q'_1,Q'_2$ is a partition of $Q_1\cup Q_2$ into two bases of $M|(Q_1\cup Q_2)$, and therefore of $M$. Thus, $Q'_1,Q'_2,Q_3$ is a partition of the ground set of $M$ into $3$ bases. Since $\{r,\pi(r),s,t\}\subseteq Q'_1\cup Q_3$, it follows that $M|(Q'_1\cup Q_3)$ has a $U_{2,4}$ restriction, so $M|(Q'_1\cup Q_3)\not\cong M(K_4)$, so $Q'_1,Q'_2,Q_3$ is the desired partition.
\end{proof}

\section{Analyzing an example}\label{sec:D27}
\begin{figure}[h]
\centering\includegraphics[scale=.3]{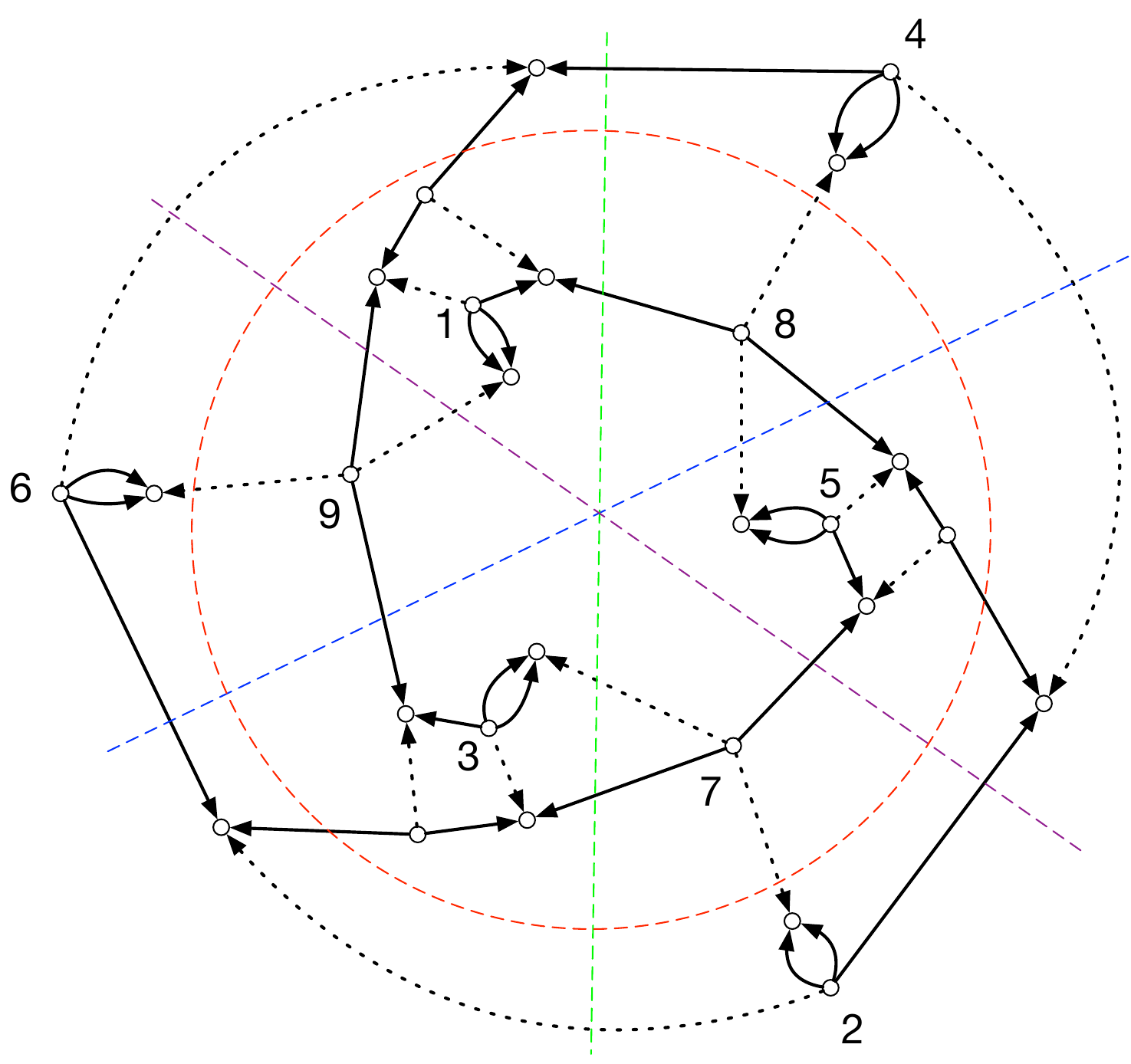}
\caption{The sink-regular $(3,4)$-bipartite digraph $D_{27}$ and the sink-regular weighted $(2,3)$-bipartite digraph $(D_{27},w_{27})$.
Solid arcs have weight $1$, and dashed arcs have weight $0$. 
$Q_1:=\{1,2,3\}, Q_2:=\{4,5,6\}$, and $Q_3:=\{7,8,9\}$ partition the active vertices into bases of $M_1(D_{27},\1)$, such that $M_1(D_{27},\1)|(Q_1\cup Q_2)$ is isomorphic to $M(K_4)$, and is therefore not strongly base orderable.
}
\label{fig:D27}
\end{figure}

Figure~\ref{fig:D27} displays the sink-regular $(3,4)$-bipartite digraph $D_{27}=(V,A)$ on $27$ vertices, introduced in \S\ref{subsec:examples}, with $\rho(3,D_{27})=3$ and active vertices $[9]$. Moreover, with the solid arcs having weight $1$ and the dashed arcs having weight $0$, we get a sink-regular weighted $(2,3)$-bipartite digraph $(D_{27},w_{27})$ with $\rho(2,D_{27},w_{27})=3$ and active vertices $[6]$. Together, $D_{27}$ and $(D_{27},w_{27})$ address several questions raised in the previous two sections. Let us elaborate.

Let $Q_1:=\{1,2,3\}, Q_2:=\{4,5,6\}$, and $Q_3:=\{7,8,9\}$. It can be readily checked that each $Q_i,i=1,2,3$ is a basis for $M_1(D_{27},\1)$, and thus for $M_0(D_{27},\1)$. Moreover, since $M_1(D_{27},w_{27})=M_1(D_{27},\1)|(Q_1\cup Q_2)$, it follows that $Q_1,Q_2$ are bases for $M_1(D_{27},w_{27})$. Moreover, it can be readily checked that in $M_1(D_{27},w_{27})$, the symmetric basis exchanges between $Q_1,Q_2$ are the pairs $\{4,1\},\{4,2\},\{4,3\},\{5,3\},\{6,3\}$, which do not include a perfect pairing between $Q_1,Q_2$. Consequently, $Q_1,Q_2$ prove that $M_1(D_{27},w_{27})$, and therefore $M_1(D_{27},\1)$, is not strongly base orderable. (Observe that $M_1(D_{27},w_{27})\cong M(K_4)$.) This shows that the assumption of \Cref{SBO-weighted-packing} does not hold for $(D,w)=(D_{27},\1)$. The example also shows that the second assumption of \Cref{2-SBO-weighted-packing} does not hold for $(D,w)=(D_{27},\1)$, $Q=Q_3$, and $Q'=Q_1\cup Q_2$.

\section{Directions for further research}\label{sec:future}

Let us present several directions for future research.

\subsection{$M_1(D,\1)$ and strongly base orderability}

\begin{QU}\label{SBO-questions}
Let $\tau\geq 3$ be an integer, and $D=(V,A)$ a sink-regular $(\tau,\tau+1)$-bipartite digraph. Are there disjoint bases $Q_1,\ldots,Q_\tau$ of $M_1(D,\1)$ such that $M_1(D,\1)|(Q_1\cup \cdots\cup Q_{\tau-1})$ is strongly base orderable?
\end{QU}

Observe that if the answer to \Cref{SBO-questions} is affirmative, which is the case for $\rho(\tau,D)=3$ by \Cref{rho=3-basis-partition}, then by \Cref{2-SBO-weighted-packing}, $A$ can be partitioned into $\tau$ dijoins. One can then apply the Decompose, Lift, and Reduce procedure for (unweighted) digraphs to obtain similar conclusions for all digraphs.

\subsection{Disjoint rounded $1$-factor witnesses}

Let $D=(V,A)$ be a sink-regular $(\tau,\tau+1)$-bipartite digraph. A natural approach for finding $\tau$ disjoint dijoins is to first partition the ground set of $M_1(D,\1)$ into $\tau$ disjoint bases $Q_1,\ldots,Q_\tau$, and then look for disjoint rounded $1$-factors $J_1,\ldots,J_\tau$ such that $\dc(J_i)=Q_i$ for $i\in [\tau]$. (Note that our proof of $[[3,3]]$ did not quite follow this approach; the basis partition had to be changed.) As a first step in this direction, we would need to address the following question. (For this question, it is not clear whether being a basis of $M_1(D,\1)$ rather than $M_0(D,\1)$ is helpful.)

\begin{QU}
Let $\tau\geq 3$ be an integer, and $D=(V,A)$ a sink-regular $(\tau,\tau+1)$-bipartite digraph. Let $Q_1,\ldots,Q_\tau$ be disjoint bases of $M_0(D,\1)$. 
When are there disjoint rounded $1$-factors $J_1,\ldots,J_\tau$ such that $\dc(J_i)=Q_i$ for $i\in [\tau]$?
\end{QU}

\subsection{Finding three disjoint dijoins in planar digraphs, and Barnette's Conjecture}

We mentioned the following result in the introduction.

\begin{theorem}[\cite{Chudnovsky16}]\label{Chudnovsky}
Let $(D=(V,A),w)$ be a weighted digraph, where every dicut has weight at least two. Suppose $D$ is planar, and $D[\{a\in A:w_a\neq 0\}]$ is a spanning subdigraph of $D$ that is connected as an undirected graph. Then there exists a $w$-weighted packing of dijoins of size two.
\end{theorem}

\begin{CN}\label{Barnette-esq}
Let $D=(V,A)$ be a sink-regular $(3,4)$-bipartite digraph that is planar, and let $Q_1,Q_2,Q_3$ be disjoint bases of $M_1(D,\1)$. Then there exists a rounded $1$-factor $J$ such that $\dc(J)=Q_1$ and $D\setminus J$ is connected.
\end{CN}

\begin{theorem}\label{3-dijoins-planar}
Suppose \Cref{Barnette-esq} is true. Then $[[3;\pl]]$ is true. That is, every planar digraph where every dicut has size at least three, has three disjoint dijoins.
\end{theorem}
\begin{proof}
By \Cref{reduction}~(1), it suffices to prove $[3;\pl]$ for sink-regular $(3,4)$-bipartite digraphs. To this end, let $D=(V,A)$ be a sink-regular $(3,4)$-bipartite digraph that is planar. By \Cref{M1-basis-partition}, there exist disjoint bases $Q_1,Q_2,Q_3$ of $M_1(D,\1)$. As \Cref{Barnette-esq} is assumed true, there exists a rounded $1$-factor $J_1$ such that $\dc(J_1)=Q_1$ and $D\setminus J_1$ is connected. Let $b:=2\chi_V+\chi_{Q_2}+\chi_{Q_3}$. Since $\dc(J_1)=Q_1$, $A-J_1$ is a perfect $b$-matching.
It follows from \Cref{k-dijoin} that $A-J_1$ is a $2$-dijoin. In particular, in the weighted digraph $(D,w)$ where $w=\chi_{A-J_1}$, every dicut has weight at least two. As $D\setminus J_1$ is connected, it follows that $D[\{a\in A:w_a\neq 0\}]$ is a spanning subdigraph of $D$ that is connected as an undirected graph. Thus, by \Cref{Chudnovsky}, $(D,w)$ has a $w$-weighted packing of dijoins of size two, that is, $A-J_1$ can be partitioned into two dijoins, say $J_2,J_3$. Thus we have three disjoint dijoins $J_1,J_2,J_3$ in $D$.
\end{proof}

Note the resemblance between \Cref{Barnette-esq} and Barnette's Conjecture, which states that for every $3$-connected cubic bipartite graph $G$ that is planar, there exists a perfect matching $M$ such that $G\setminus M$ is connected, i.e., $G$ has a Hamilton circuit~\cite{Barnette69}.

\subsection{Fractional weighted packing of dijoins}

Let $\mathcal{C}$ be a clutter over ground set $A$. Let $w\in \mathbb{Z}^A_{\geq 0}$. A \emph{fractional $w$-weighted packing of $(\mathcal{C},w)$ with value $\nu$} consists of a fractional assignment $\lambda_C\geq 0$ to every $C\in \mathcal{C}$ such that $\1^\top \lambda=\nu$, and $\sum(\lambda_C:a\in C\in \mathcal{C})\leq w_a$ for every $a\in A$. Let $\nu^\star(\mathcal{C},w)$ be the maximum value of a  fractional $w$-weighted packing of $(\mathcal{C},w)$. Observe that $\nu^\star(\mathcal{C},w)\geq \nu(\mathcal{C},w)$. By Weak LP Duality, $\tau(\mathcal{C},w)\geq \nu^\star(\mathcal{C},w)$. $\mathcal{C}$ is \emph{ideal} if for all $w\in \mathbb{Z}^A_{\geq 0}$, $\tau(\mathcal{C},w)= \nu^\star(\mathcal{C},w)$~\cite{Cornuejols94}. It is known that a clutter is ideal if, and only if, the blocker is ideal~\cite{Lehman79,Fulkerson71}.

The theorem of Lucchesi and Younger on weighted packings of dicuts implies that the clutter of minimal dicuts of a digraph is ideal~\cite{Lucchesi78}, implying by \Cref{dicut-dijoin} that the clutter of minimal dijoins of a digraph is also ideal. In other words,

\begin{theorem}[see \cite{Cornuejols01}, \S1.3.4]\label{optimal-fractional-packing}
Let $(D,w)$ be a weighted digraph where the minimum weight of a dicut is $\tau$. Then there exists a fractional $w$-weighted packing $\lambda$ of dijoins of value $\tau$.
\end{theorem}

\begin{QU}\label{question:dyadic}
Can we choose $\lambda$ such that \begin{enumerate}[(1)]
\item $\|\lambda\|_{\infty}\geq \frac12$?
\item $\lambda$ is $\frac12$-integral?
\item $\lambda$ is \emph{dyadic}, i.e. $\frac{1}{2^k}$-integral for some integer $k\geq 0$?
\end{enumerate} What if $w=\1$?
\end{QU}

The following consequence of our results relates to \Cref{question:dyadic}~(1).

\begin{theorem}
Let $D=(V,A)$ be a digraph where the minimum size of a dicut is $\tau$. Then there exists a fractional packing $\lambda$ of dijoins of value $\tau$ such that $\|\lambda\|_{\infty}\geq 1$.
\end{theorem}
\begin{proof}
If $\tau=1$, then this holds trivially. Otherwise, $\tau\geq 2$. By \Cref{two-dijoins}, there exists a dijoin $J$ such that for every dicut $\delta^+(U)$, $|\delta^+(U)-J|\geq \tau-1$. Let $w:=\chi_{A-J}$. Then the inequality implies that the minimum weight of a dijoin of $(D,w)$ is $\tau-1$. Thus, by \Cref{optimal-fractional-packing}, $(D,w)$ has a fractional $w$-weighted packing $\lambda$ of dijoins of value $\tau-1$. Update $\lambda$ by setting $\lambda_{J}:=1$. The updated $\lambda$ is the desired fractional packing.
\end{proof}

Given an ideal clutter $\mathcal{C}$ with covering number $\tau$, the value of the optimization problem $\lambda(\mathcal{C}):=\max\{\|\lambda\|_\infty:\lambda \text{ is a fractional packing of value $\tau$}\}$ was studied recently by Ferchiou in \cite{Ferchiou2021}, where he proved a beautiful min-max theorem involving $\lambda(\mathcal{C})$ (see Theorem 34).

It is shown in \cite{Shepherd05} that given a weighted digraph $(D,w)$ where the minimum weight of a dicut is~$\tau$, there exists a $\frac12$-integral $w$-weighted packing of dijoins of value $\frac{\tau}{2}$, giving some hope for a positive answer to \Cref{question:dyadic}~(2).

Let us now provide some rationale for \Cref{question:dyadic}~(3).

\begin{theorem}[\cite{Abdi-dyadic}]\label{tau=2-dyadic}
Let $\mathcal{C}$ be an ideal clutter with covering number $\tau$, where $\tau=1,2$. Then there exists a dyadic fractional packing of value $\tau$.
\end{theorem}

\begin{theorem}
Let $D=(V,A)$ be a digraph where the minimum size of a dicut is $\tau$, and $\tau\leq 3$. Then there exists a fractional packing of dijoins of value $\tau$ that is dyadic.
\end{theorem}
\begin{proof}
If $\tau\leq 2$, then there exists a packing of dijoins of size $\tau$, so we are done. Otherwise, $\tau= 3$.
It follows from \Cref{two-dijoins} that $A$ can be partitioned into a dijoin $J_1$ and a $(\tau-1)$-dijoin $J_2$. Let $\lambda_1$ be the fractional $\chi_{J_1}$-weighted packing of $(D,\chi_{J_1})$ that assigns a value of $1$ to $J_1$, and a value of $0$ to every other dijoin.
Consider the weighted digraph $(D,\chi_{J_2})$; the minimum weight of a dicut is $\tau-1=2$. Thus, by \Cref{tau=2-dyadic}, $(D,\chi_{J_2})$ has a fractional $\chi_{J_2}$-weighted packing $\lambda_2$ of dijoins of value $2$ that is dyadic. Observe that $\lambda_1+\lambda_2$ is a fractional packing of dijoins of $D$ value $\tau$ that is dyadic, as required.
\end{proof}

\subsection{The $\tau=2$ Conjecture, and fixing the refuted Edmonds-Giles Conjecture}

A clutter $\mathcal{C}$ over ground set $A$ has the \emph{max-flow min-cut (MFMC) property} if $\tau(\mathcal{C},w)=\nu(\mathcal{C},w)$ for all $w\in \mathbb{Z}^A_{\geq 0}$~\cite{Seymour77}, and has the \emph{packing property} if $\tau(\mathcal{C},w)=\nu(\mathcal{C},w)$ for all $w\in \{0,1,\infty\}^A$~\cite{Cornuejols00}. Observe that the MFMC property implies idealness. The \emph{Replication Conjecture} predicts that the packing property also implies the MFMC property~\cite{Conforti93}. Lehman's seminal theorem on \emph{minimally non-ideal} clutters~\cite{Lehman90} implies that the packing property implies idealness~\cite{Cornuejols00}, providing some evidence for the conjecture.

There is a conjecture that implies the Replication Conjecture. A clutter $\mathcal{C}$ is \emph{minimally non-packing} if it does not have the packing property but every proper minor does. The \emph{$\tau=2$ Conjecture} predicts that every ideal minimally non-packing clutter has a cover of size two~\cite{Cornuejols00}. The conjecture is known to imply the Replication Conjecture~\cite{Cornuejols00}. For the clutter of minimal dijoins of a weighted digraph, the conjecture reduces to the following.

\begin{CN}\label{tau=2-CON}
Let $(D,w)$ be a weighted digraph such that $\mathcal{C}(D,w)$ is minimally non-packing. Then $(D,w)$ has a dicut of weight two.
\end{CN}

This conjecture would follow from the following conjecture, which we propose as a fix to the refuted Edmonds-Giles Conjecture.

\begin{CN}\label{decomposition-question}
Let $(D,w)$ be a weighted digraph where the minimum weight of a dicut is~$\tau$, where $\tau\geq 3$. Then there exist weighted digraphs $(D,c),(D,c')$, where $w=c+c'$, the minimum weight of a dicut in $(D,c)$ is $1$, and the minimum weight of a dicut in $(D,c')$ is $\tau-1$.
\end{CN}

By applying the Decompose, Lift, and Reduce procedure, and \Cref{weighted-two-dijoins}, this conjecture would follow if the answer to the following question were affirmative.

\begin{QU}\label{common-base-packing-CON}
Let $\tau\geq 3$ be an integer, and $(D=(V,A),w)$ a sink-regular weighted $(\tau,\tau+1)$-bipartite digraph. Can the set of active vertices be partitioned into an admissible set and a $(\tau-1)$-admissible set?
\end{QU}

\section*{Acknowledgements}

Ahmad Abdi and G\'{e}rard Cornu\'{e}jols would like to thank the organizers of the 2021 HIM program Discrete Optimization during which part of this work was developed. The authors would like to thank Andr\'{a}s Frank, Bertrand Guenin, Bruce Shepherd, Levent Tun\c{c}el, L\'{a}szlo V\'{e}gh, and Giacomo Zambelli for fruitful discussions about this work. Special thanks goes to the dedicated reviewers who identified some errors in an earlier draft, and whose comments vastly improved the presentation of the current manuscript.

{\small \bibliographystyle{abbrv}\bibliography{references}}

\newpage

\appendix

\section{Decompose, Lift, and Reduce Procedure: the proof}\label{sec:DLRproof}

 In \S\ref{sec:decompose}, we see how every weighted digraph without a ``pseudo-cut-vertex" can be \emph{decomposed} into ``irreducible" weighted digraphs. Then, after introducing a gadget in \S\ref{sec:gadget}, we see in \S\ref{sec:lifting} how every irreducible weighted digraph where every dicut has weight at least $\tau$, can be \emph{lifted} to a weighted $(\tau,\tau+1)$-bipartite digraph. Putting the two together, we obtain the \emph{Decompose-and-Lift operation} in \S\ref{sec:DnL}.

\subsection{Decomposing}\label{sec:decompose}

Given a weighted digraph $(D,w)$, and a vertex $v$, denote by $(D,w)\setminus v$ the weighted digraph obtained after deleting $v$ and all the arcs incident with it, and dropping the corresponding weights from $w$.

\begin{DE}
Let $(D,w)$ be a weighted digraph with no dicut of weight $0$. A \emph{pseudo-cut-vertex} is a vertex $v$ such that $(D,w)\setminus v$ has a dicut of weight $0$.
\end{DE}

Moving forward, we shall need the following ``triangle inequality".

\begin{RE}\label{triangle-inequality}
Let $\tau\geq 2$ be an integer. Then $(a+b)\mod{\tau}$ is either $(a\mod{\tau})+(b\mod{\tau})$ or $(a\mod{\tau})+(b\mod{\tau})-\tau$. Consequently, for any finite set $S$ of integers, $\big(\sum_{a\in S} a\big)\mod{\tau}\leq \sum_{a\in S} (a\mod{\tau})$.
\end{RE}

For a vector $w\in \mathbb{Z}^A$, and a subset $A'\subseteq A$, denote by $w|_{A'}$ the subvector of $w$ restricted to the entries in $A'$.

\begin{LE}\label{pseudo-cut-vertex}
Let $(D=(V,A),w)$ be a weighted digraph that has no dicut of weight $0$, and has a pseudo-cut-vertex. Then there exist weighted digraphs $(D_1,w_1),(D_2,w_2)$ without dicuts of weight $0$ such that the following statements hold: \begin{enumerate}[(1)]
\item $|V(D_1)|,|V(D_2)|\leq |V|-1$, $A(D_1)\cup A(D_2)=A$, every arc of $(D,w)$ of nonzero weight belongs to exactly one of $A(D_1),A(D_2)$, and $w_1=w|_{A(D_1)},w_2=w|_{A(D_2)}$,
\item if $D$ is planar, then so is each $D_i,i\in [2]$
\item $\rho(\tau,D_i,w_i)\leq \rho(\tau,D,w)$ for every integer $\tau\geq 2$ and $i\in[2]$,
and
\item $\mathcal{C}(D,w)=\mathcal{C}(D_1,w_1)\times\mathcal{C}(D_2,w_2)$.
\end{enumerate}
\end{LE}
\begin{proof}
After replacing every arc $a$ of nonzero weight with $w_a$ arcs of weight $1$ with the same head and tail, if necessary, we may assume that $w\in \{0,1\}^A$.
Let $u$ be a pseudo-cut-vertex.
Let $\delta^+(U'_1)$ be a dicut of $(D,w)\setminus u$ of weight $0$; let $U'_2:=V- u- U'_1$. Let $U_1:=U'_1\cup \{u\}\subseteq V$ and $U_2:=U'_2\cup \{u\}\subseteq V$. Let $(D_1,w_1)$ be obtained from $(D,w)$ by replacing $U_2$ with a single vertex $u_1$, where all the arcs of $D$ with both ends in $U_2$ are removed, all the arcs with exactly one end in $U_2$ are now attached to $u_1$ and have the same weight, and all the other arcs remain intact with the same weight. Similarly, let $(D_2,w_2)$ be obtained from $(D,w)$ by replacing $U_1$ with a single vertex $u_2$, where all the arcs of $D$ with both ends in $U_1$ are removed, all the arcs with exactly one end in $U_1$ are now attached to $u_2$ and have the same weight, and all the other arcs remain intact with the same weight. Observe that $A(D_1)\cup A(D_2)=A$, and $A(D_1)\cap A(D_2)$ is equal to the set of arcs from $U'_1$ to $U'_2$ all of which have weight $0$ in $(D,w)$. Subsequently, {\bf (1)} holds.

\begin{claim} 
$D[U_i],i=1,2$ is connected as an undirected graph.
\end{claim}
\begin{cproof}
Suppose for a contradiction $D[U_i]$ is disconnected as an undirected graph, and let $W\subseteq U_i$ be a connected component of $D[U_i]$ that excludes $u$, i.e. $W\subseteq U'_i$. Then $\delta_D(W)$ contains only arcs that go between $U'_1$ and $U'_2$, so $\delta_D(W)$ yields a dicut of $(D,w)$ of weight $0$, a contradiction as every dicut of $(D,w)$ has nonzero weight.
\end{cproof}

\begin{claim} 
If $D$ is planar, then so is $D_i,i=1,2$. That is, {\bf (2)} holds.
\end{claim}
\begin{cproof}
By definition, $D_i$ is obtained from $D$ by shrinking $U_{3-i}$. By Claim~1, $D[U_{3-i}]$ is connected as an undirected graph, so $D_i$ can be viewed as a contraction minor of $D$, thereby proving the claim as contraction preserves planarity.
\end{cproof}

\begin{claim} 
$\rho(\tau,D_i,w_i)\leq \rho(\tau,D,w)$ for every integer $\tau\geq 2$ and $i=1,2$. That is, {\bf (3)} holds.
\end{claim}
\begin{cproof}
For every vertex $v$ of $D_i$ other than $u_i$, $w(\delta^+_{D_i}(v))-w(\delta^-_{D_i}(v))=w(\delta^+_{D}(v))-w(\delta^-_{D}(v))$. Moreover,
$w(\delta^+_{D_i}(u_i))-w(\delta^-_{D_i}(u_i)) = w(\delta^+_{D}(U_{3-i}))-w(\delta^-_{D}(U_{3-i}))$. Subsequently, \begin{align*}
&\tau\cdot\rho(\tau,D_i,w_i)\\
&= \sum_{v\in V(D_i)} (w(\delta_{D_i}^+(v))-w(\delta_{D_i}^-(v))\mod{\tau})\\
&=\sum_{v\in U'_i} (w(\delta_D^+(v))-w(\delta_D^-(v))\mod{\tau}) + (w(\delta_{D_i}^+(u_i))-w(\delta_{D_i}^-(u_i))\mod{\tau})\\
&=\sum_{v\in U'_i} (w(\delta_D^+(v))-w(\delta_D^-(v))\mod{\tau}) + (w(\delta_{D}^+(U_{3-i}))-w(\delta_{D}^-(U_{3-i})) \mod{\tau})\\
&=\sum_{v\in U'_i} (w(\delta_D^+(v))-w(\delta_D^-(v))\mod{\tau}) \\
&\qquad+ \left(\left(
\sum_{v\in U_{3-i}} (w(\delta_{D}^+(v))-w(\delta_{D}^-(v)))\right) \mod{\tau}\right)\\
&\leq \sum_{v\in U'_i} (w(\delta_D^+(v))-w(\delta_D^-(v))\mod{\tau}) \\
&\qquad+
\sum_{v\in U_{3-i}} (w(\delta_{D}^+(v))-w(\delta_{D}^-(v)) \mod{\tau}) \quad \text{ by \Cref{triangle-inequality}}\\
&=\tau\cdot\rho(\tau,D,w),
\end{align*} so $\rho(\tau,D_i,w_i)\leq \rho(\tau,D,w)$.
\end{cproof}

\begin{claim} 
$\mathcal{C}(D,w)=\mathcal{C}(D_1,w_1)\times\mathcal{C}(D_2,w_2)$, that is, {\bf (4)} holds.
\end{claim}
\begin{cproof}
By \Cref{product-coproduct}, it suffices to prove $b(\mathcal{C}(D,w)) = b(\mathcal{C}(D_1,w_1)) \otimes b(\mathcal{C}(D_2,w_2))$.
Since $(D_i,w_i)$ is obtained from $(D,w)$ by identifying vertices, every dicut of $(D_i,w_i)$ is also a dicut of $(D,w)$. Thus, every set in $b(\mathcal{C}(D_1,w_1))\otimes b(\mathcal{C}(D_2,w_2))$ contains a set of $b(\mathcal{C}(D,w))$. Conversely, let $\delta_D^+(W)$ be a dicut of $(D,w)$. Define $W'$ as follows:
 \begin{enumerate}
\item[Case 1:] $u\in W$ and $(V-W)\cap U'_2\neq \emptyset$. In this case, let $W':=W\cup U'_1$.
\item[Case 2:] $u\in W$ and $(V-W)\cap U'_2= \emptyset$. In this case, let $W':=W$.
\item[Case 3:] $u\notin W$ and $W\cap U'_1\neq \emptyset$. In this case, let $W':=W\cap U'_1$.
\item[Case 4:] $u\notin W$ and $W\cap U'_1= \emptyset$. In this case, let $W':=W$.
\end{enumerate}
We know that every arc of $D$ between $U'_1,U'_2$ goes from $U'_1$ to $U'_2$ and has weight $0$. Thus, $\delta_{D}^+(W')$ remains a dicut of $D$ whose set of weight-$1$ arcs is contained in the set of weight-$1$ arcs of $\delta_D^+(W)$. Moreover, in cases 2 and 3, $\delta_D^+(W')$ is also a dicut of $D_1$, while in cases 1 and 4, $\delta_D^+(W')$ is also a dicut of $D_2$. In both cases, we proved that $\delta_D^+(W)\cap \{a\in A:w_a=1\}$ contains the set of weight-$1$ arcs of a dicut of some $(D_i,w_i)$. Thus, every set in $b(\mathcal{C}(D,w))$ contains a set of $b(\mathcal{C}(D_1,w_1))\otimes b(\mathcal{C}(D_2,w_2))$, as required.
\end{cproof}

We have proved (1)-(4). Observe that (4) implies that $(D_i,w_i)$ has no dicut of weight $0$, thereby finishing the proof.
\end{proof}

\begin{DE}
A weighted digraph is \emph{irreducible} if it has no dicut of weight $0$, and no pseudo-cut-vertex.
\end{DE}

\begin{theorem}[Decomposing]\label{decomposition-irreducible}
Let $(D=(V,A),w)$ be a weighted digraph that has no dicut of weight $0$. Then there exist irreducible weighted digraphs $\{(D_i,w_i):i\in I\}$ for a finite index set $I$, such that the following statements hold: \begin{enumerate}[(1)]
\item $\bigcup_{i\in I}A(D_i)=A$, every arc $(D,w)$ of nonzero weight belongs to exactly one of $(D_i,w_i),i\in I$, and $w_i=w|_{A(D_i)}$ for each $i\in I$,
\item if $D$ is planar, then so is each $D_i,i\in I$,
\item $\rho(\tau,D_i,w_i)\leq \rho(\tau,D,w)$ for every integer $\tau\geq 2$ and $i\in I$, and
\item $\mathcal{C}(D,w)=\prod_{i\in I}\mathcal{C}(D_i,w_i)$.
\end{enumerate}
\end{theorem}
\begin{proof}
This decomposition is obtained by repeatedly applying \Cref{pseudo-cut-vertex}, a process that is bound to terminate since at every iteration, the number of vertices of each ``piece" strictly decreases.
\end{proof}

The notion of a pseudo-cut-vertex in weighted digraphs can be viewed as an extension of the notion of a cut-vertex in graphs. In this vein, irreducibility in weighted digraphs is an extension of $2$-connectivity in graphs.

\begin{theorem}[Unweighted Decomposing]\label{unweighted-decomposition-irreducible}
Let $D=(V,A)$ be a digraph that is connected as undirected graph. Then there exist digraphs $\{D_i:i\in I\}$ for a finite index set $I$, each of which is $2$-connected as an undirected graph, such that the following statements hold: \begin{enumerate}[(1)]
\item $\{A(D_i):i\in I\}$ partition $A$,
\item if $D$ is planar, then so is each $D_i,i\in I$,
\item $\rho(\tau,D_i)\leq \rho(\tau,D)$ for every integer $\tau\geq 2$ and $i\in I$, and
\item $\mathcal{C}(D)=\prod_{i\in I}\mathcal{C}(D_i)$.
\end{enumerate}
\end{theorem}
\begin{proof}
This follows immediately from applying \Cref{decomposition-irreducible} to the weighted digraph $(D,\1)$.
\end{proof}

\subsection{A gadget needed for lifting}\label{sec:gadget}

Having described decomposing, we move on to \emph{lifting} wherein an irreducible weighted digraph is lifted to a weighted $(\tau,\tau+1)$-bipartite digraph, for some integer $\tau\geq 2$ that is a lower bound on the weight of every dicut of the original weighted digraph. By a routine argument, we shall assume that every original arc has weight $0$ or $1$. We then replace each original vertex $v$ that is neither a source nor a sink of weighted degree $\tau$, with a certain \emph{gadget}, and the original arcs attached to $v$ are then joined to vertices of the gadget, constructed in \Cref{gadget} below.

The gadget is a weighted digraph $(D,w)$, where $D$ is plane, with certain weighted degree conditions, and the vertices of attachment appear in some clockwise ordering on the boundary of the embedding; see Figure~\ref{fig:gadget}. In the case that the underlying digraph of the original weighted digraph is plane, the clockwise ordering agrees with the clockwise ordering of the original arcs attached to $v$ in the plane embedding. There are $4$ types of original arcs that can be attached to $v$, leading to $4$ types of vertices of attachment, depending on whether the original arc leaves or enters $v$, denoted by $+$ or $-$, and whether the original arc has weight $0$ or $1$. We shall represent the clockwise ordering and the $4$ types with a sequence $\mathfrak{s}$ with entries in $\{(+,0),(+,1),(-,0),$ $(-,1)\}$.

\begin{figure}[ht]
\centering
\includegraphics[scale=0.3]{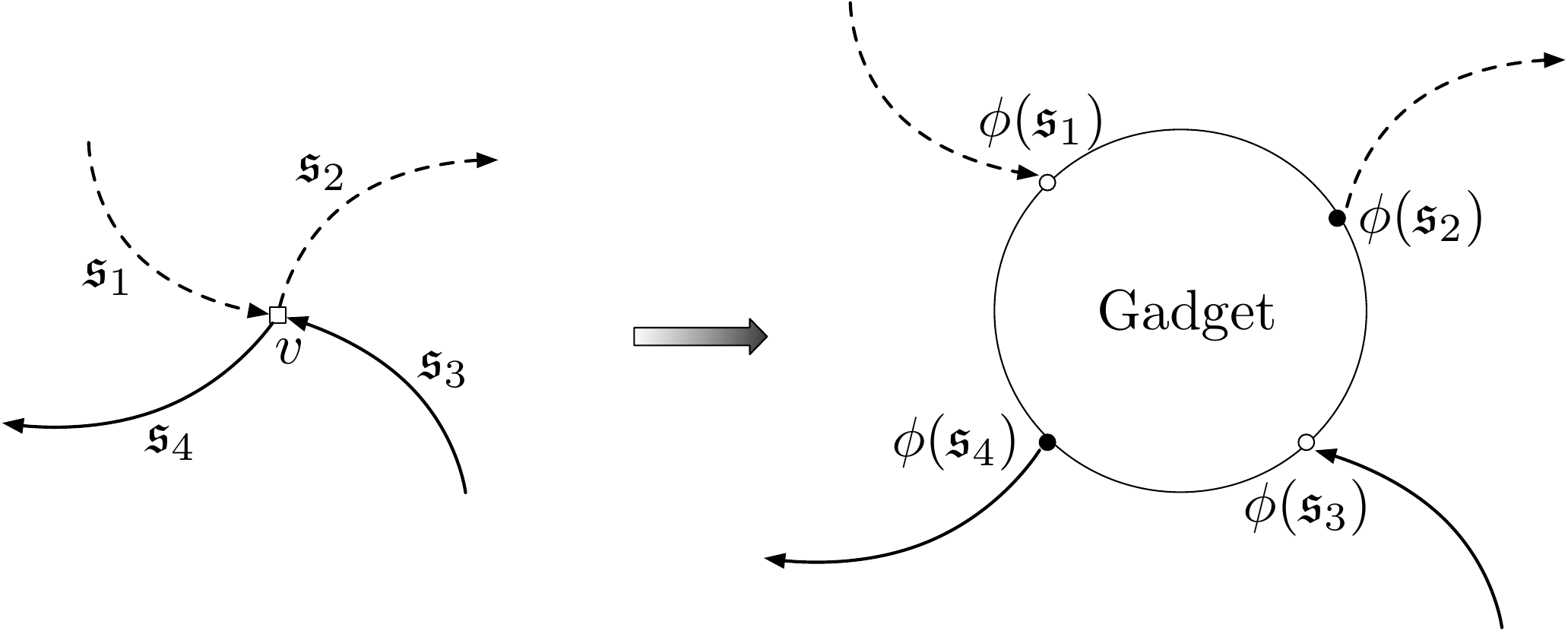}
\caption{An illustration of a gadget of \Cref{gadget} replacing original vertex $v$. Solid arcs have weight $1$, and dashed arcs have weight $0$. The sequence entries are equal to $\mathfrak{s}_1=(-,0)$, $\mathfrak{s}_2=(+,0)$, $\mathfrak{s}_3=(-,1)$, and $\mathfrak{s}_4=(+,1)$.}
\label{fig:gadget}
\end{figure}

\begin{LE}\label{gadget}
Let $\tau\geq 2$ be an integer, and $\mathfrak{s}$ a finite sequence with entries in $\{(+,0),(+,1),(-,0),$ $(-,1)\}$. Let $\mathfrak{s}(i,j)$ be the number of entries of $\mathfrak{s}$ equal to $(i,j)$. Take integers $\ell^+,\ell^-\geq 0$ such that $\ell^+-\ell^-\equiv \mathfrak{s}(+,1)-\mathfrak{s}(-,1) \pmod{\tau}$. Then there exists a weighted digraph $(D,w):=(D(\tau,\mathfrak{s},\ell^+,\ell^-), w(\tau,\mathfrak{s},\ell^+,\ell^-))$ such that the following statements hold: \begin{enumerate}[(1)]
\item There are no two opposite arcs, and every arc has weight $1,2,\lfloor\frac{\tau-1}{2}\rfloor,\lceil\frac{\tau-1}{2}\rceil$ or $\lceil\frac{\tau}{2}\rceil$. In particular, if $\tau\geq 3$, then every arc has nonzero weight.
\item $D$ is a plane bipartite digraph, and every vertex of $(D,w)$ has weighted degree $\tau-1,\tau$ or $\tau+1$.
\item Every dicut of $(D,w)$ has weight at least $\tau-1$.
\item $(D,w)$ has exactly $\ell^+ + \ell^-$ vertices of weighted degree $\tau+1$, $\ell^+$ of which are sources and so $\ell^-$ of which are sinks, and no two of which are adjacent in $D$.
\item $(D,w)$ has exactly $\mathfrak{s}(+,1)+\mathfrak{s}(-,1)$ vertices of weighted degree $\tau-1$.
\item There is an injection $\phi:\mathfrak{s}\to V(D)$, where $\phi(\mathfrak{s}_1),\phi(\mathfrak{s}_2),\ldots,\phi(\mathfrak{s}_{|\mathfrak{s}|})$ appear in clockwise ordering on the boundary of a plane drawing of $D$, and $\phi(\mathfrak{s}_i)$ is a source of weighted degree $\tau-1$ if $\mathfrak{s}_i=(+,1)$, it is a source of weighted degree $\tau$ if $\mathfrak{s}_i=(+,0)$, it is a sink of weighted degree $\tau-1$ if $\mathfrak{s}_i=(-,1)$, and it is a sink of weighted degree $\tau$ if $\mathfrak{s}_i=(-,0)$.
\end{enumerate}
\end{LE}
\begin{proof}
We construct $(D,w)$ in four steps. (For a worked out construction, see \Cref{EG:gadget} appearing after the proof.)
Let $k':=\mathfrak{s}(-,0) + \mathfrak{s}(-,1) + \tau \mathfrak{s}(+,0) + (\tau-1)\mathfrak{s}(+,1)$, and $k$ an integer sufficiently large \hypertarget{gadget-3-conditions}{such that} \begin{enumerate}
\item[i.] $k-k'+\ell^-\equiv 0 \pmod{\tau}$, and so $k-\mathfrak{s}(-,0)+\ell^+ \equiv 0\pmod{\tau}$ because $\ell^+-\ell^- \equiv \mathfrak{s}(+,1)-\mathfrak{s}(-,1) \pmod{\tau}$,
\item[ii.] there exist integers $n_1,\ldots,n_{\frac{k-k'+\ell^-}{\tau}}$ such that $\frac{\tau}{2}\leq n_i\leq \tau$ and $\sum_{i} n_i = k-k'$, and so $0\leq \tau-n_i\leq n_i$ and $\sum_i (\tau-n_i)=\ell^-$,
\item[iii.] there exist integers $m_1,\ldots,m_{\frac{k-\mathfrak{s}(-,0) +\ell^+}{\tau}}$ such that $\frac{\tau}{2}\leq m_j\leq \tau$ and $\sum_{j} m_j = k-\mathfrak{s}(-,0)$, and so $0\leq \tau-m_j\leq m_j$ and $\sum_j (\tau-m_j)=\ell^+$.
\end{enumerate} 
Let $I(+,j):=\{i:\mathfrak{s}_i=(+,j)\}$ and $I(-,j):=\{i:\mathfrak{s}_i=(-,j)\}$.

\paragraph{Step 1: The weighted rectangle.} Let us start with a bipartite digraph with \begin{itemize}
\item sources: $a_0,a_1,\ldots,a_k,a'_0, a'_1,\ldots,a'_k$,
\item sinks: $b_0,b_1,\ldots,b_k,b'_0, b'_1,\ldots,b'_k$,
\item undirected circuit: $(a_0,b_0,a_1,b_1,\ldots,a_k,b_k,a'_k,b'_k,\ldots,a'_1,b'_1,a'_0,b'_0)$,
\item additional arcs: $a_1b'_1,a_2b'_2,\ldots,a_kb'_k$.
\end{itemize} We shall work with the plane embedding of the digraph displayed in Figure~\ref{fig:rectangle}. We assign the following arc weights as displayed in Figure~\ref{fig:rectangle}:\begin{itemize}
\item (dashed, single stroke) arcs of weight $\lceil\frac{\tau-1}{2}\rceil$: $a_ib_i,a'_ib'_i$ for $i=0,1,\ldots,k$.
\item (dashed, double stroke) arcs of weight $\lfloor\frac{\tau-1}{2}\rfloor$: $a_ib_{i-1},a'_{i-1}b'_i$ for $i=1,\ldots,k$. Note that
if $\tau=2$, then these arcs become weight-$0$ arcs.
\item (solid, double stroke) arcs of weight $\lceil\frac{\tau}{2}\rceil$: $a_0b'_0,a'_kb_k$.
\item (solid, single stroke) arcs of weight $1$: $a_ib'_i$ for $i=1,\ldots,k$.
\end{itemize} We call this weighted digraph the \emph{weighted rectangle}. Observe that if $\tau\geq 3$, then every arc has nonzero weight.

We claim that every dicut of the weighted rectangle has weight at least $\tau-1$. To this end, note that every cut of the weighted rectangle has weight at least $\tau-2$, as it contains a Hamilton circuit where every arc has weight $\geq \lfloor \frac{\tau-1}{2}\rfloor$. Moreover, every cut of weight $\tau-2$, if any, must separate $\{a_i,b'_i,b_i,a'_i\}$ from $\{a_{i+1},b'_{i+1},b_{i+1},a'_{i+1}\}$ for some $0\leq i\leq k-1$, implying in turn that it is not a dicut. Thus, every dicut of the weighted rectangle has weight at least $\tau-1$.

\begin{figure}[ht]
\centering
\includegraphics[scale=0.3]{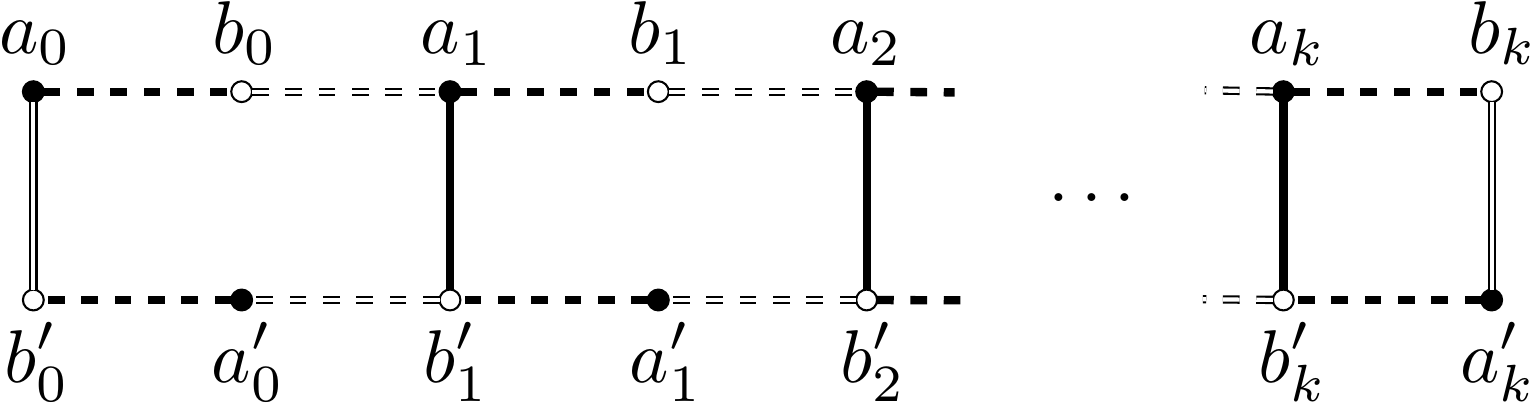}
\caption{The weighted rectangle, where $b'_0$ is placed at coordinates $(0,0)$, $b_k$ at $(1+2k,1)$, and all the other vertices are evenly spaced in the plane. Filled-in vertices correspond to sources, and the other vertices to sinks. The dashed single stroke, dashed double stroke, solid double stroke, and solid single stroke arcs have weight $\lceil\frac{\tau-1}{2}\rceil, \lfloor\frac{\tau-1}{2}\rfloor, \lceil\frac{\tau}{2}\rceil$, and $1$, respectively. The dashed double stroke arcs have weight $0$ iff $\tau=2$.}
\label{fig:rectangle}
\end{figure}

Moving forward, for each $i\in \{1,2,\ldots,|\mathfrak{s}|\}$, let \begin{align*}
f(i)&:=|\{j:\mathfrak{s}_j=(-,0) \text{ or } (-,1), j<i\}|\\
 &\quad + (\tau-1)\cdot |\{j:\mathfrak{s}_j=(+,1), j<i\}|\\
 &\quad +\tau\cdot |\{j:\mathfrak{s}_j=(+,0), j<i\}|.\end{align*}

\paragraph{Step 2: Adding rungs.}  For each $i\in I(-,0)$, add an arc $a'_{f(i)}b_{f(i)}$ of weight $1$; note that $0\leq f(i)\leq k'-1$.

\paragraph{Step 3: Adding sources.} In this step, we add sources to the rectangle whose neighbors belong to the top long side of the rectangle.

For each $i\in I(+,j),j=0,1$, introduce a new source $s_i$ to the weighted rectangle, incident only with weight-$1$ arcs, whose neighbors are the sinks $\{b_{f(i)+r}:r=0,1,\ldots,\tau-1-j\}$.
Observe that the neighbors of $s_i,i\in I(+,0)\cup I(+,1)$ form disjoint subintervals of $(b_0,b_1,\ldots,b_{k'-1})$. Note that each $s_i,i\in I(+,j)$ has weighted degree $\tau-j$.

Then add new sources $s'_1,\ldots,s'_{\frac{k-k'+\ell^-}{\tau}}$, incident only with weight-$1$ arcs, whose neighbors form a partition of the sequence of sinks $(b_{k'},\ldots,b_{k-1})$ into subintervals of sizes $n_1, \ldots, n_{\frac{k-k'+\ell^-}{\tau}}$, respectively. (\hyperlink{gadget-3-conditions}{See i and ii.}) Now, for each $s'_i$, double $\tau-n_i$ distinct arcs incident with $s'_i$ (or, increase their weight by $1$). Then each $s'_i$ has weighted degree $\tau$. Observe that the total number of arcs doubled is~$\ell^-$.

\paragraph{Step 4: Adding sinks.} In this step, we add sinks to the rectangle whose neighbors belong to the bottom long side of the rectangle.

Introduce new sinks $t_1,\ldots,t_{\frac{k-\mathfrak{s}(-,0)+\ell^+}{\tau}}$, incident only with weight-$1$ arcs, whose neighbors form a partition of the sequence of sources $(a'_i:0\leq i\leq k-1)\setminus (a'_{f(i)}: i\in I(-,0))$ into subintervals of sizes $m_1,\ldots,m_{\frac{k-\mathfrak{s}(-,0)+\ell^+}{\tau}}$, respectively. (\hyperlink{gadget-3-conditions}{See i and iii.}) Now, for each $t_j$, double $\tau-m_j$ distinct arcs incident with it (or, increase their weight by $1$). Then each $t_j$ has weighted degree $\tau$. Observe that the total number of arcs doubled is $\ell^+$.

\paragraph{The weighted digraph and its plane embedding.} After performing Steps 1-4, we obtain a weighted digraph $(D,w)$. We claim this is the desired weighted digraph. By construction, $D$ is a weighted bipartite digraph, and if $\tau\geq 3$ then every arc has nonzero weight. $D$ is planar with the following appropriate straight line plane embedding: Given the embedding of the rectangle in Figure~\ref{fig:rectangle}, place \begin{itemize}
\item the vertices $s_i,i\in I(+,0)\cup I(+,1)$ at coordinates $(1+2f(i),2)$,
\item $s'_j,j\in \left[\frac{k-k'+\ell^-}{\tau}\right]$ at $(1+2k'+2\sum_{i=1}^{j-1} n_i,2)$,
\item $t_j,j\in \left[\frac{k-\mathfrak{s}(-,0)+\ell^+}{\tau}\right]$ at $(1+2\sum_{i=1}^{j-1} m_i,-1)$.
\end{itemize}
For each $\mathfrak{s}_i$, define $\phi(\mathfrak{s}_i)$ as follows: if $i\in I(+,0)\cup I(+,1)$ let $\phi(\mathfrak{s}_i)=s_i$, and if $i\in I(-,0)\cup I(-,1)$ let $\phi(\mathfrak{s}_i)=b_{f(i)}$. It can be readily checked that $\phi$ satisfies (6) for the embedding above.

\paragraph{Weighted degrees:} \begin{itemize}
\item $t_1,\ldots,t_{\frac{k+\ell^+}{\tau}}$ have weighted degree $\tau$.
\item $s_i,i\in I(+,j)$ has weighted degree $\tau-j$, and $s'_i,i\in \left[\frac{k-k'+\ell^-}{\tau}\right]$ has weighted degree $\tau$.
\item $a_0,b_k,b'_0,a'_k$ have weighted degree $\lceil\frac{\tau-1}{2}\rceil+\lceil\frac{\tau}{2}\rceil=\tau$.
\item $a_1,a_2,\ldots,a_k,b'_1,b'_2,\ldots,b'_k$ have weighted degree $\lceil\frac{\tau-1}{2}\rceil + \lfloor\frac{\tau-1}{2}\rfloor +1=\tau$.
\item $b_0,b_1,\ldots,b_{k'-1}$ have weighted degree $\lceil\frac{\tau-1}{2}\rceil + \lfloor\frac{\tau-1}{2}\rfloor=\tau-1$ or $\lceil\frac{\tau-1}{2}\rceil + \lfloor\frac{\tau-1}{2}\rfloor+1=\tau$. More specifically, $b_j,j=0,1,\ldots,k'-1$ has weighted degree $\tau-1$ if, and only if, $b_j\in \{b_{f(i)}:i\in I(-,1)\}$.
\item $b_{k'},b_{k'+1},\ldots,b_{k-1}$ have weighted degree $\lceil\frac{\tau-1}{2}\rceil + \lfloor\frac{\tau-1}{2}\rfloor+1=\tau$ or $\lceil\frac{\tau-1}{2}\rceil + \lfloor\frac{\tau-1}{2}\rfloor+2=\tau+1$. More specifically, for each $j=k',k'+1,\ldots,k-1$, $b_j$ has weighted degree $\tau+1$ if, and only if, $b_j$ is incident to double arcs in Step 3. Since the number of such arcs is $\ell^-$, we get that of $b_{k'},b_{k'+1},\ldots,b_{k-1}$, exactly $\ell^-$ have weighted degree $\tau+1$, and the rest have weighted degree $\tau$.
\item $a'_0,a'_1,\ldots,a'_{k-1}$ have weighted degree $\lceil\frac{\tau-1}{2}\rceil + \lfloor\frac{\tau-1}{2}\rfloor+1=\tau$ or $\lceil\frac{\tau-1}{2}\rceil + \lfloor\frac{\tau-1}{2}\rfloor+2=\tau+1$. More specifically, $a'_j,j=0,1,\ldots,k-1$ has weighted degree $\tau+1$ if, and only if, $a'_j$ is incident to double arcs in Step 4. Since the number of such arcs is $\ell^+$, we get that of $a'_0,a'_1,\ldots,a'_{k-1}$, exactly $\ell^+$ have weighted degree $\tau+1$, and the rest have weighted degree $\tau$.
\end{itemize} In summary, every vertex has weighted degree $\tau-1,\tau,\tau+1$, with \begin{itemize}
\item exactly $\ell^+ + \ell^-$ vertices of weighted degree $\tau+1$, of which $\ell^+$ are sources and so $\ell^-$ are sinks, and no two of which are adjacent (because for $0\leq i,j<k$, $b_i$ and $a'_j$ are not adjacent),
\item exactly $\mathfrak{s}(+,1)+\mathfrak{s}(-,1)$ vertices of weighted degree $\tau-1$. \end{itemize}

\paragraph{Lower bound on dicut weights.}
We already showed in Step 1 that every dicut of the weighted rectangle has weight at least $\tau-1$. The lower bound is maintained after adding the rungs in Step 2. When adding the sources and sinks in Steps 3-4, we ensured each new vertex has weighted degree at least $\tau-1$ and only neighbors on the weighted rectangle, preserving the lower bound on the weight of every dicut.
\end{proof}

\begin{EG}\label{EG:gadget}
We are given a vertex $v$ where $\mathfrak{s}_1=(-,0),\mathfrak{s}_2=(-,1),\mathfrak{s}_3=(+,1),\mathfrak{s}_4=(-,1)$. Suppose $\tau=3,\ell^+=2$ and $\ell^-=0$.
Then $k'=5$, so we may choose $k=5$ and $m_1=m_2=2$ before Step 1 of the proof of \Cref{gadget}.
Consider Figure~\ref{fig:gadget-eg}. The four detached arcs were incident with the vertex $v$ being replaced by the gadget. In the figure, the dashed arc has weight $0$, solid single stroke arcs have weight $1$, and solid double stroke arcs have weight $2$.
In Step 1, we choose the weighted rectangle. The arc $a'_0b_0$ is added in Step 2, the vertex $s_3$ in Step 3, and the vertices $t_1,t_2$ in Step 4.
\begin{figure}[h]
\centering
\includegraphics[scale=.3]{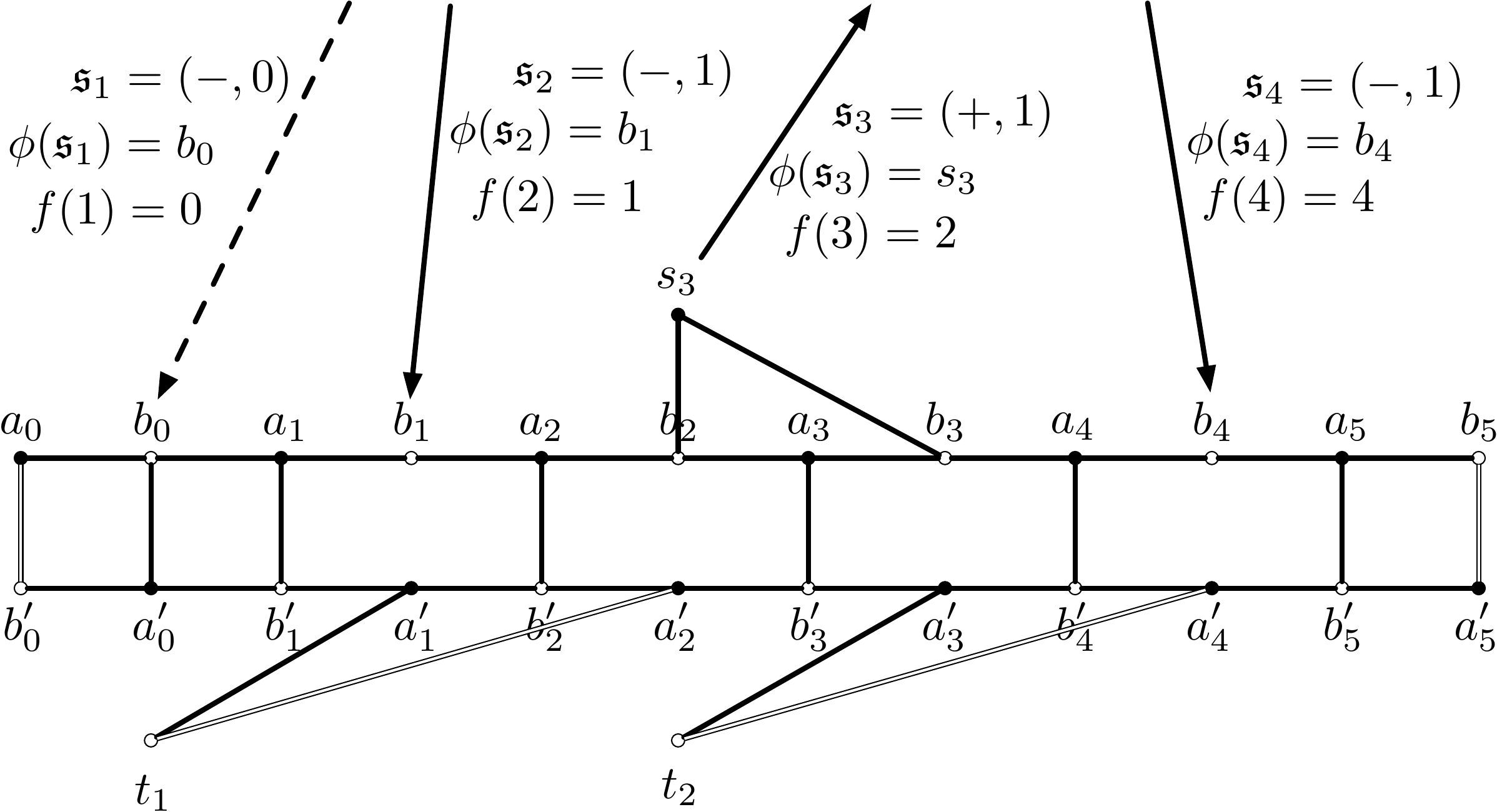}
\caption{An illustration of Steps 1-4 of the proof of \Cref{gadget}. 
}
\label{fig:gadget-eg}
\end{figure}
\end{EG}

\begin{RE}\label{tau=3-simple-gadget}
When $\tau=3$, we can turn the gadget of \Cref{gadget} to a simple (unweighted) digraph. To this end, observe that the gadget has no opposite arcs. Since $\tau=3$, every arc has weight $1$ or $2$. For every arc $a=(u,v)$ of weight $2$, replace it with the unweighted digraph displayed in Figure~\ref{fig:simple-gadget}.
\begin{figure}[h]
\centering
\includegraphics[scale=.3]{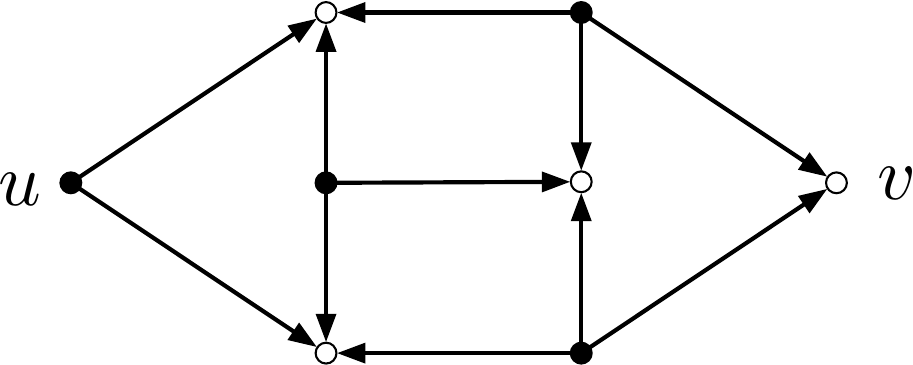}
\caption{Simplifying gadget of \Cref{tau=3-simple-gadget}.}
\label{fig:simple-gadget}
\end{figure} It can be readily checked that the revised gadget, with the all-ones weights, satisfies (1)-(6).
\end{RE}

The remark above does not extend to $\tau\geq 4$, in that there is no simple planar bipartite graph with only vertices of degree $\tau$ except for exactly two vertices of degree $\tau-1$; this follows as a fairly immediate consequence of \emph{Euler's formula}.

\subsection{Lifting}\label{sec:lifting}

\begin{theorem}[Lifting]\label{weighted-lifting}
Let $(D=(V,A),w)$ be an irreducible weighted digraph, where every dicut has weight at least $\tau$, and $\tau\geq 2$. For each $v\in V$, choose integers $\ell^+(v),\ell^-(v)\geq 0$ such that $\ell^+(v)-\ell^-(v)\equiv w(\delta^+(v))-w(\delta^-(v)) \pmod{\tau}$. Then there exists a weighted $(\tau,\tau+1)$-bipartite digraph $(D'=(V',A'),w')$ such that \begin{enumerate}[(1)]
\item $A\subseteq A'$, every arc in $A$ has the same weight in both $(D,w)$ and $(D',w')$, every arc in $A'-A$ has nonzero weight in $(D',w')$ if $\tau\geq 3$, and $D=D'/(A'-A)$,
\item the number of sources of $(D',w')$ of weighted degree $\tau+1$ is $\sum_{v\in V}\ell^+(v)$, and the number of sinks of $(D',w')$ of weighted degree $\tau+1$ is $\sum_{v\in V}\ell^-(v)$,
\item $\mathcal{C}(D,w)$ is a contraction minor of $\mathcal{C}(D',w')$, and
\item if $D$ is planar, then so is $D'$.
\end{enumerate}
\end{theorem}
\begin{proof}
By replacing every arc $a$ of weight $w_a\geq 1$ with $w_a$ arcs of weight $1$ with the same head and tail, if necessary, we may assume that $w\in \{0,1\}^A$.
For each $v\in V$, let $\mathfrak{s}^v$ be a sequence with a distinct $(-,j),j=0,1$ entry for every weight-$j$ arc entering $v$, and a distinct $(+,j),j=0,1$ entry for every weight-$j$ arc leaving $v$, and no other entries. For now, the ordering of the entries in $\mathfrak{s}^v$ is not relevant; this will be relevant for Claim~4 below. Let $D_v:=D(\tau,\mathfrak{s}^v,\ell^+(v),\ell^-(v))$ and $w_v:=w(\tau,\mathfrak{s}^v,\ell^+(v),\ell^-(v))$ as given by \Cref{gadget}. Let $(D'=(V',A'),w')$ be obtained from $(D,w)$ by replacing each $v$ with the gadget $(D_v,w_v)$, where the arc in $D$ incident with $v$ corresponding to $\mathfrak{s}_i$ is now incident with $\phi(\mathfrak{s}_i)$ in the gadget; moreover, every new arc introduced has the same weight as in $(D_v,w_v)$, while all the old arcs have the same weight as in $(D,w)$.

\begin{claim} 
Every dicut of $(D',w')$ has weight at least $\tau$.
\end{claim}
\begin{cproof}
Let $\delta^+_{D'}(U)$ be a dicut of $(D',w')$. If $U$ does not separate any $V(D_v),v\in V$, then $\delta^+_{D'}(U)\subseteq A$ is in fact a dicut of $(D,w)$, so it has weight at least $\tau$. Otherwise, $\delta^+_{D'}(U)$ separates some $V(D_u),u\in V$. Since $\delta^+_{D'}(U)$ yields a dicut of $(D_u,w_u)$, and every dicut of $(D_u,w_u)$ has weight at least $\tau-1$, it follows that $\delta^+_{D'}(U)\cap A(D_u)$ has weight at least $\tau-1$. Since $(D,w)$ has no pseudo-cut-vertex, $u$ is not a pseudo-cut-vertex of $(D,w)$, so the arcs of $\delta^+_{D'}(U)$ of nonzero weight cannot all be contained in $A(D_u)$, implying in turn that $\delta^+_{D'}(U)$ has weight at least $\tau-1+1=\tau$. Thus, in both cases, $\delta^+_{D'}(U)$ has weight at least $\tau$, as required.
\end{cproof}

\begin{claim} 
$(D',w')$ is a weighted $(\tau,\tau+1)$-bipartite digraph satisfying {\bf (1)} and {\bf (2)}.
\end{claim}
\begin{cproof}
It can be readily checked that $(D',w')$ is a weighted bipartite digraph where every vertex has weighted degree $\tau$ or $\tau+1$, and no two vertices of weighted degree $\tau+1$ are adjacent (to see the latter, note that every two vertices of weighted degree $\tau+1$ from different gadgets are clearly not adjacent, and every two vertices of weighted degree $\tau+1$ from the same gadget are not adjacent by construction). By Claim~1, every dicut has weight at least $\tau$. Thus, $(D',w')$ is a weighted $(\tau,\tau+1)$-bipartite digraph. (1)-(2) are satisfied by construction.
\end{cproof}

\begin{claim} 
Every dicut of $(D,w)$ is also a dicut of $(D',w')$. Conversely,
every dicut of $(D',w')$ is either a dicut of $(D,w)$, or it contains a weight-$1$ arc of $(D',w')$ outside of $(D,w)$. That is, $b(\mathcal{C}(D,w))$ is a deletion minor of $b(\mathcal{C}(D',w'))$, and so {\bf (3)} holds.
\end{claim}
\begin{cproof}
Clearly every dicut of $(D,w)$ is also a dicut of $(D',w')$. Conversely, let $\delta^+_{D'}(U)$ be a dicut of $(D',w')$. If $U$ does not separate any $V(D_v),v\in V$, then $\delta^+_{D'}(U)\subseteq A$ is a dicut of $(D,w)$. Otherwise, $\delta^+_{D'}(U)$ separates some $V(D_u),u\in V$. Since $\delta^+_{D'}(U)$ yields a dicut of $(D_u,w_u)$, and every dicut of $(D_u,w_u)$ has weight at least $\tau-1$, it follows that $\delta^+_{D'}(U)\cap A(D_u)$ has weight at least $\tau-1\geq 1$, so $\delta^+_{D'}(U)$ contains a weight-$1$ arc of $(D_u,w_u)$, which is a weight-$1$ arc of $(D',w')$ outside of $(D,w)$, as required.
\end{cproof}

\begin{claim} 
If $D$ is planar, then one can ensure that $D'$ is also planar.  That is, {\bf (4)} holds.
\end{claim}
\begin{cproof}
Suppose $D$ is planar, and fix a plane drawing of it. Recall that $D'$ is obtained from $D$ by replacing each $v$ by a plane gadget $D_v$, and rewiring the arcs attached to $v$ to distinct vertices of $D_v$. Thus, to ensure that $D'$ remains planar, it suffices to ensure that when rewiring the arcs attached to $v$ to the vertices of $D_v$, the arcs are attached to vertices on the boundary of the gadget, and the clockwise ordering of the arcs given in the plane drawing of $D$ has been respected in the rewiring stage. This can be guaranteed by ensuring the ordering of $\mathfrak{s}^v$ follows a clockwise ordering of the arcs incident with $v$ in the plane drawing of $D$.
\end{cproof}

Claims~2-4 finish the proof.
\end{proof}

\begin{theorem}[Unweighted Lifting]\label{lifting}
Let $D=(V,A)$ be a digraph that is $2$-connected as an undirected graph, where every dicut has size at least $\tau$, and $\tau\geq 3$. For each $v\in V$, choose integers $\ell^+(v),\ell^-(v)\geq 0$ such that $\ell^+(v)-\ell^-(v)\equiv \deg^+(v)-\deg^-(v) \pmod{\tau}$. Then there exists a $(\tau,\tau+1)$-bipartite digraph $D'=(V',A')$ such that \begin{enumerate}[(1)]
\item $A\subseteq A'$ and $D=D'/(A'-A)$,
\item the number of sources of $D'$ of degree $\tau+1$ is $\sum_{v\in V}\ell^+(v)$, and the number of sinks of $D'$ of degree $\tau+1$ is $\sum_{v\in V}\ell^-(v)$,
\item $\mathcal{C}(D)$ is a contraction minor of $\mathcal{C}(D')$, and
\item if $D$ is planar, then so is $D'$.
\end{enumerate}
\end{theorem}
\begin{proof}
We apply \Cref{weighted-lifting} to $(D,{\bf 1})$ and $\tau\geq 3$ to obtain a weighted $(\tau,\tau+1)$-bipartite digraph $(D'=(V',A'),w')$ where every arc of $A'$ has nonzero weight. After replacing every arc $a$ of weight $w'_a$ by $w'_a$ arcs of weight $1$ with the same head and tail, we obtain the desired $(\tau,\tau+1)$-bipartite digraph.
\end{proof}

\subsection{Proofs of Theorems~\ref{weighted-DnL} and \ref{DnL}}\label{sec:DnLproof}

\begin{proof}[Proof of \Cref{weighted-DnL}]
We proceed in two stages.

\paragraph{Stage 1:} Decompose $(D,w)$ into irreducible weighted digraphs. Since $(D,w)$ has no dicut of weight $0$, we may apply \Cref{decomposition-irreducible} to get irreducible weighted digraphs $(D_i,w_i),i\in I$ satisfying \Cref{decomposition-irreducible}~(1)-(4). By (1), $w_i=w|_{A(D_i)}$ for each $i\in I$. By (2), if $D$ is planar, then so is each $D_i,i\in I$. By (3), $\rho(\tau,D_i,w_i)\leq \rho(\tau,D,w)$ for each $i\in I$. By (4), $\mathcal{C}(D,w)=\prod_{i\in I}\mathcal{C}(D_i,w_i)$.

\paragraph{Stage 2:} Lift each $(D_i,w_i),i\in I$. Since $(D_i,w_i)$ is irreducible, and every dicut has weight at least $\tau$, we may apply \Cref{weighted-lifting}. For each vertex $v\in V(D_i)$, choose $\ell^+_i(v),\ell^-_i(v)\geq 0$ such that $\ell^+_i(v) -\ell^-_i(v)\equiv w_i(\delta_{D_i}^+(v))-w_i(\delta_{D_i}^-(v)) \pmod{\tau}$. Then by \Cref{weighted-lifting}, there exists a weighted $(\tau,\tau+1)$-bipartite digraph $(D'_i,w'_i)$ satisfying \Cref{weighted-lifting}~(1)-(4).
By (1), if $w_i>\0$ and $\tau\geq 3$, then $w'_i>\0$.
By (2), the number of sources of $(D'_i,w'_i)$ of weighted degree $\tau+1$ is $\sum_{v\in V(D_i)}\ell^+_i(v)$, and the number of sinks of $(D'_i,w'_i)$ of weighted degree $\tau+1$ is $\sum_{v\in V(D_i)}\ell^-_i(v)$.
By (3), $\mathcal{C}(D_i,w_i)$ is a contraction minor of $\mathcal{C}(D'_i,w'_i)$.
By (4), if $D_i$ is planar, then so is $D'_i$.
It remains to fix the choices of $\ell^+_i(v)$ and $\ell^-_i(v)$ for $v\in V(D_i)$, which we do according to one of the following two criteria:
\begin{enumerate}[i.]
\item For each $v\in V(D_i)$, let $\ell^+_i(v):=w_i(\delta_{D_i}^+(v))$ and $\ell^-_i(v):=w_i(\delta_{D_i}^-(v))$. Then $\sum_{v\in V(D_i)}\ell^+_i(v)=\sum_{v\in V(D_i)}\ell^-_i(v)$, so $(D'_i,w'_i)$ is balanced.
\item For each $v\in V(D_i)$, define $\ell^+_i(v):=w_i(\delta_{D_i}^+(v))-w_i(\delta_{D_i}^-(v)) \mod{\tau}$ and $\ell^-_i(v):=0$. Then $(D'_i,w'_i)$ is sink-regular with exactly $\tau\rho(\tau,D_i,w_i)$ sources of degree $\tau+1$. In particular, $\rho(\tau,D'_i,w'_i)=\frac{1}{\tau}(\tau\rho(\tau,D_i,w_i))=\rho(\tau,D_i,w_i)$. Since we have $\rho(\tau,D_i,w_i)\leq \rho(\tau,D,w)$ from Stage 1, we get that $\rho(\tau,D'_i,w'_i)\leq \rho(\tau,D,w)$.
\end{enumerate}

\paragraph{Summary} We claim $(D'_i,w'_i),i\in I$ are the desired weighted digraphs. {\bf (1)} Suppose $w>\0$ and $\tau\geq 3$. Stage 1 guarantees $w_i>\0,i\in I$, and so Stage 2 guarantees $w'_i>\0,i\in I$. {\bf (2)} If $D$ is planar, then each $D_i,i\in I$ is planar, so each $D'_i,i\in I$ is planar. {\bf (3)} holds for $\mathcal{C}_i:=\mathcal{C}(D_i,w_i),i\in I$. {\bf (i)-(ii)} Moreover, we can choose each $(D'_i,w'_i),i\in I$ to satisfy either (i) or (ii). \end{proof}

\begin{proof}[Proof of \Cref{DnL}]
We apply \Cref{weighted-DnL} to $(D,{\bf 1})$ and $\tau\geq 3$ to obtain weighted $(\tau,\tau+1)$-bipartite digraphs $(D'_i,w'_i),i\in I$ for a finite index set $I$, where every arc of $(D'_i,w'_i)$ has nonzero weight by \Cref{weighted-DnL}~(1). For each $i\in I$, replace every arc $a$ of $(D'_i,w'_i)$ of weight $w'_{i,a}$ by $w'_{i,a}$ arcs of weight $1$ with the same head and tail, we obtain the desired $(\tau,\tau+1)$-bipartite digraph.
\end{proof}

\section{An elementary proof of $[[\wt,\tau,2]]$}\label{sec:rho=2-elementary}

The following remark shows that in a digraph connected as an undirected graph, a dicut is uniquely determined by its \emph{shore}.

\begin{RE}\label{dicut-shore}
Let $D=(V,A)$ be a digraph that is connected as an undirected graph. Let $\delta^+(U),\delta^+(W)$ be dicuts such that $\delta^+(U)=\delta^+(W)$. Then $U=W$.
\end{RE}
\begin{proof}
Since $\delta^+(U)=\delta^+(W)$ and $\delta^-(U)=\delta^-(W)=\emptyset$, we have $\delta(U)=\delta(W)$, so $\delta(U\tr W)=\delta(U)\tr \delta(W)=\emptyset$. Since $D$ is a connected graph, either $U\tr W=\emptyset$ or $U\tr W=V$. However, the latter cannot occur as it would imply that $U,W$ are complementary, which is not possible as both $\delta^+(U),\delta^+(W)$ are nonempty dicuts. Thus, $U\tr W=\emptyset$, so that $U=W$, as required.
\end{proof}

\begin{DE} Given a digraph $D=(V,A)$ connected as an undirected graph, a dicut $\delta^+(U)$ is \emph{trivial} if $|U|=1,|V|-1$, and is \emph{non-trivial} otherwise. \end{DE}

We are now ready to prove \Cref{rho=2-sink-reg}, whose statement we repeat for convenience: \begin{quote}
\emph{Let $(D=(V,A),w)$ be a sink-regular weighted $(\tau,\tau+1)$-bipartite digraph such that $\rho(\tau,D,w)\leq 2$. Then there exists a $w$-weighted packing of dijoins of size $\tau$.}
\end{quote}
\begin{proof}[Alternative proof of \Cref{rho=2-sink-reg}]
 After replacing every arc $a$ of nonzero weight with $w_a$ arcs of weight $1$ with the same head and tail, if necessary, we may assume that $w\in \{0,1\}^A$. Let $A_i:=\{a\in A:w_a=i\}$ for $i=0,1$. We proceed by induction on the number $n\geq 0$ of non-trivial minimum weight dicuts (of weight $\tau$).

\paragraph{Induction step}
We postpone the base case $n=0$ to later. For now, let us first prove the induction step. Assume that $n\geq 1$. Let $\delta^+(U)$ be a minimum weight dicut that is non-trivial. Let $(D_1,w_1)$ be obtained from $(D,w)$ by replacing $V-U$ with a single vertex $u_1$, where all the arcs of $D$ with both ends in $V-U$ are removed, all the arcs with exactly one end in $V-U$ are now attached to $u_1$ and have the same weight, and all the other arcs remain intact with the same weight. Similarly, let $(D_2,w_2)$ be obtained from $(D,w)$ by replacing $U$ with a single vertex $u_2$, where all the arcs of $D$ with both ends in $U$ are removed, all the arcs with exactly one end in $U$ are now attached to $u_2$ and have the same weight, and all the other arcs remain intact with the same weight. Note that the weight-$1$ arcs of $(D_i,w_i),i=1,2$ share the weight-$1$ arcs in $\delta^+(U)$; let us label $\delta_D^+(U)\cap A_1=\{a_1,\ldots,a_\tau\}$.

Observe that every (non-trivial) dicut of $(D_i,w_i)$ is also a (non-trivial) dicut of $(D,w)$ of the same weight. Thus, every dicut of $(D_i,w_i)$ has weight at least $\tau$. Moreover, $u_1$ is a sink of $(D_1,w_1)$ of weighted degree $\tau$, and $u_2$ is a source of $(D_2,w_2)$ of weighted degree $\tau$. Subsequently, $\rho(\tau,D,w)=\rho(\tau,D_1,w_1)+\rho(\tau,D_2,w_2)$, and $(D_i,w_i)$ is a sink-regular weighted $(\tau,\tau+1)$-bipartite digraph satisfying $\rho(\tau,D_i,w_i)\leq \rho(\tau,D,w)\leq 2$, and the number of non-trivial minimum weight dicuts of $(D_i,w_i)$ is at most $n-1$.

We may therefore apply the induction hypothesis to conclude that $(D_i,w_i)$ has a $w_i$-weighted packing of dijoins of size $\tau$, consisting of $J^i_1,\ldots,J^i_\tau$, labeled so that $J^i_j\cap \delta_D^+(U)=\{a_j\}$ for $j\in [\tau]$. We claim that $(J^1_j\cup J^2_j:j\in [\tau])$ is a $w$-weighted packing of dijoins of $(D,w)$ of size $\tau$. To prove this, it suffices to prove that each $J^1_j\cup J^2_j$ is a dijoin of $D$. Fix $j\in [\tau]$ and $J:=J^1_j\cup J^2_j$. Pick a dicut $\delta_D^+(W)$ of $D$. \begin{enumerate}
\item[Case 1:] $U\cap W=\emptyset$: In this case, $\delta_D^+(W)$ is also a dicut of $D_2$, so $\delta_D^+(W)\cap J^2_j\neq \emptyset$.
\item[Case 2:] $U\cup W=V$: In this case, $\delta_D^+(W)$ is also a dicut of $D_1$, so $\delta_D^+(W)\cap J^1_j\neq \emptyset$.
\item[Case 3:] $W\cap U\neq \emptyset$ and $W\cup U\neq V$: In this case, $\delta_D^+(U\cap W)$ is a dicut of $D$ and $\delta_D^+(U\cup W)$ is a dicut of $D$. By submodularity of dicuts, $$
|J\cap \delta_D^+(U)|+|J\cap \delta_D^+(W)| \geq
|J\cap \delta_D^+(U\cap W)|+|J\cap \delta_D^+(U\cup W)|
$$ We know that $J\cap \delta_D^+(U)=\{a_j\}$. Moreover, $\delta_D^+(U\cap W)$ is also a dicut of $D_1$ so $\delta_D^+(U\cap W)\cap J^1_j\neq \emptyset$, and $\delta_D^+(U\cup W)$ is also a dicut of $D_2$ so $\delta_D^+(U\cup W)\cap J^2_j\neq \emptyset$. Thus, $$
|J\cap \delta_D^+(W)| \geq
|J\cap \delta_D^+(U\cap W)|+|J\cap \delta_D^+(U\cup W)|-|J\cap \delta_D^+(U)|\geq 1+1-1=1,
$$ so $J\cap \delta_D^+(W)\neq \emptyset$.
\end{enumerate} In all cases, we showed $J\cap \delta_D^+(W)\neq \emptyset$. Since this holds for every dicut of $D$, it follows that $J$ is a dijoin, as claimed. This completes the induction step.

\paragraph{Base case}
It remains to prove the case $n=0$. That is, every minimum weight dicut is trivial. If $\rho(\tau,D,w)\leq 1$, then by \Cref{rho=1-sink-reg}, there exists an (equitable) $w$-weighted packing of dijoins of size $\tau$, so we are done. Otherwise, $\rho(\tau,D,w)=2$, so by \Cref{disc}~(1), $\disc(V)=2$.

\begin{claim} 
Let $\delta^+(U)$ be a dicut of $D$. Then $\disc(U)\leq 1$. Moreover, equality holds if, and only if, $U=V-\{v\}$ for a sink $v$.
\end{claim}
\begin{cproof}
By \Cref{disc}~(3), $\disc(U)\leq \disc(V)-1=1$. Moreover, if $\disc(U)=1$, then $\delta^+(U)$ is a minimum weight dicut, so since $n=0$, $\delta^+(U)$ is a trivial dicut, implying in turn that $U=V-\{v\}$ for a sink $v$ (note that $\disc(\{u\})=-1$ for every source $u$). Conversely, if $U=V-\{v\}$ for a sink $v$, then $\delta^+(U)$ is a dicut of $D$ with $\disc(U)=\disc(V)-\disc(v)=2-1=1$.
\end{cproof}

Let $\mathcal{U}$ be the set of $U\subseteq V$ such that $\disc(U)=0$ and $\delta^+(U)$ is a dicut of $D$.
Let $\mathcal{U}_{\min}$ be the set of minimal sets in $\mathcal{U}$.

\begin{claim} 
For all distinct $U,W\in \mathcal{U}_{\min}$, $U\cup W$ contains all the sources of $D$.
\end{claim}
\begin{cproof}
Pick distinct $U,W\in \mathcal{U}_{\min}$.
Note that $\delta^+(U),\delta^+(W)$ are non-trivial dicuts (since the shore of every trivial dicut has discrepancy $\pm1$). Since $n=0$, we get that $\delta^+(U),\delta^+(W)$ are dicuts of weight at least $\tau+1$. Thus, by \Cref{disc}~(2), \begin{align*}
|a(U)|&=w(\delta^+(U))+\tau\cdot \disc(U)= w(\delta^+(U))\geq \tau+1\\
|a(W)|&=w(\delta^+(W))+\tau\cdot \disc(W)= w(\delta^+(W))\geq \tau+1.
\end{align*} By \Cref{disc}~(1), $a(V)=\tau\cdot\disc(V)=2\tau$, so the inequalities above imply that $a(U)\cap a(W)\neq \emptyset$, so $U\cap W\neq \emptyset$. If $U\cup W=V$, then we are clearly done. Otherwise, $U\cup W\neq V$. Thus, $\delta^+(U\cap W)$ and $\delta^+(U\cup W)$ are dicuts of $D$.
By Claim~1, $\disc(U\cap W)\leq 1$, and since $\delta^+(U),\delta^+(W)$ are non-trivial dicuts, equality does not hold. Moreover, since $U,W\in \mathcal{U}_{\min}$ and are distinct, $U\cap W$ is a proper subset of $U,W$, and $U\cap W\notin \mathcal{U}$, so $\disc(U\cap W)\neq 0$. Thus, $\disc(U\cap W)\leq -1$.
By Claim~1, we also have $\disc(U\cup W)\leq 1$. Moreover, by modularity of discrepancy, $\disc(U\cap W)+\disc(U\cup W)=\disc(U)+\disc(W)=0$. Thus, $\disc(U\cap W)=-1$ and $\disc(U\cup W)=1$, and by Claim~1, $U\cup W=V-\{v\}$ for a sink $v$, so the claim follows.
\end{cproof}

By \Cref{R1F-decomposition}, we can partition $A_1$ into $\tau$ rounded $1$-factors. Amongst all such partitions, pick one $F_1,\ldots,F_\tau$ that maximizes $$
p(F_1,\ldots,F_\tau):=\sum_{U\in \mathcal{U}_{\min}} |\{i\in [\tau]:F_i \text{ has a dyad center in } U\}|\leq \tau |\mathcal{U}_{\min}|.
$$

\begin{claim} 
$p(F_1,\ldots,F_\tau)=\tau |\mathcal{U}_{\min}|$.
\end{claim}
\begin{cproof}
Suppose otherwise. After a possible relabeling, we may assume that $F_1$ has no dyad center in some $U\in \mathcal{U}_{\min}$. Consequently, the two dyad centers of $F_1$ belong to $V- U$, and so by Claim~2 belong to every $U'\in \mathcal{U}_{\min}-\{U\}$.

Since $|a(U)|\geq \tau+1$ (by following the argument in the proof of Claim~2) and every active vertex is a dyad center of exactly one of $F_1,\ldots,F_\tau$, one of $F_2,\ldots,F_\tau$, say $F_2$, has both dyad centers in $U$. Note that $\dc(F_2)\cap \dc(F_1)=\emptyset$. Thus, by \Cref{alternating}~(1), there exists an $(F_2,F_1)$-alternating path $P$.
Assume that $P$ is a $(u,w)$-path, where $\dc(F_2)=\{u,v\}\subseteq U$ and $\dc(F_1)=\{w,t\}\subseteq V-U$. Let $F'_1:=F_1\tr P$ and $F'_2:=F_2\tr P$.
Then by \Cref{alternating}~(2), $F'_1$ is a rounded $1$-factor such that $\dc(F'_1)=\{u,t\}$, and $F'_2$ is a rounded $1$-factor such that $\dc(F'_2)=\{w,v\}$. Observe that $\{u,t\}$ and $\{w,v\}$ intersect $U$, as well as every set in $\mathcal{U}_{\min}-\{U\}$ because it contains $\{w,t\}$. As a result, $$
p(F_1,F_2,F_3,\ldots,F_\tau)<p(F'_1,F'_2,F_3,\ldots,F_\tau),
$$ thereby contradicting the maximal choice of $F_1,\ldots,F_\tau$.
\end{cproof}

\begin{claim} 
Each $F_i,i\in [\tau]$ is a dijoin.
\end{claim}
\begin{cproof}
By \Cref{R1F}~(3), we need to show that for every dicut $\delta^+(U)$,
$$
|\dc(F_i)\cap U|\geq 1+\disc(U).
$$ If $\disc(U)< 0$ then we are done. Otherwise, $\disc(U)\geq 0$. By Claim~1, $\disc(U)\leq 1$ with equality holding if, and only if, $U=V-\{v\}$ for a sink $v$. The inequality above clearly holds if $U=V-\{v\}$ for a sink $v$, so let us focus on the case $\disc(U)=0$. Then $U\in \mathcal{U}$. Clearly, we may assume that $U\in \mathcal{U}_{\min}$, in which case the inequality above holds by Claim~3, as required.
\end{cproof}

Claim~4 finishes the proof for the base case $n=0$.
\end{proof}

\end{document}